\documentclass{gtpart}

\usepackage{diagrams}
\usepackage{amsmath, amscd, amssymb, mathrsfs}
\usepackage{tocloft}

\usepackage[symbol*]{footmisc}
\usepackage{tikz}
\usepackage{tikz-cd} 
\usetikzlibrary{shapes.misc}
\usetikzlibrary{shapes.misc}
\usetikzlibrary{arrows,shapes,positioning}
\usetikzlibrary{decorations.markings}
\tikzstyle arrowstyle=[scale=1]
\tikzstyle directed=[postaction={decorate,decoration={markings,
    mark=at position .65 with {\arrow[arrowstyle]{stealth}}}}]
\tikzstyle reverse directed=[postaction={decorate,decoration={markings,
    mark=at position .65 with {\arrowreversed[arrowstyle]{stealth};}}}]
\tikzset{cross/.style={cross out, draw=black, minimum
    size=2*(#1-\pgflinewidth),
    inner sep=0pt, outer sep=0pt},  cross/.default={1pt}}

\newcommand{\cA}{{\mathscr A}}

\newcommand{\GG}{\mathbb{G}}

\newcommand{\cD}{{\mathcal D}}

\newcommand{\cW}{{\mathcal W}}

\newcommand{\cC}{{\mathscr C}}

\newcommand{\cF}{{\mathscr F}}

\newcommand{\cK}{{\mathcal K}}
\newcommand{\hcK}{{\widehat{\mathcal K}}}
\newcommand{\cont}{\mathsf{c}}

\newcommand{\cV}{{\mathcal V}}

\newcommand{\bM}{{\overline{M}}}
\newcommand{\Mbar}{{\overline{M}}}

\newcommand{\direct}{\mathsf{dir}}

\newcommand{\tw}{{\mathsf{tw}}}
\newcommand{\sgn}{{\mathsf{sgn}}}
\newcommand{\usgn}{{\underline{\sgn}}}

\newcommand{\abs}{{\mathsf{abs}}}

\newcommand{\MC}{{\mathsf{MC}}}

\newcommand{\pt}{{\mathsf{pt}}}

\newcommand{\res}{{\mathsf{res}}}
\newcommand{\hres}{{\mathsf{hres}}}
\newcommand{\bbK}{\mathbb{K}}

\newcommand{\bbC}{\mathbb{C}}
\newcommand{\bbZ}{\mathbb{Z}}
\newcommand{\bbQ}{\mathbb{Q}}

\newcommand{\triv}{{\mathsf{triv}}}

\newcommand{\Tate}{\mathsf{Tate}}

\newcommand{\pairing}{\langle\,-\,,\,-\,\rangle}
\newcommand{\Muk}{{\sf Muk}}
\newcommand{\fr}{{\sf fr}}
\newcommand{\hS}{{\sf hS}}
\newcommand{\comb}{{\sf comb}}
\newcommand{\g}{{\mathfrak g}}
\newcommand{\h}{{\mathfrak h}}
\newcommand{\hg}{\widehat{\mathfrak g}}
\newcommand{\hh}{\widehat{\mathfrak h}}
\newcommand{\hcV}{\widehat{\mathcal V}}
\newcommand{\hbA}{\widehat{\beta}^A}
\newcommand{\betaA}{\beta^A}
\newcommand{\wSym}{\widehat{\Sym}_\hbar}
\newcommand{\sym}{{\sf sym}}

\newcommand{\ra}{\rightarrow}

\newcommand{\lra}{\longrightarrow}

\newcommand{\iso}{\cong}

\newcommand{\pj}{\mathbb{P}}

\newcommand{\del}{\partial}

\newcommand{\biota}{\bar{\iota}}

\DeclareMathOperator{\Sym}{Sym}
\DeclareMathOperator{\Fuk}{Fuk}

\DeclareMathOperator{\Aut}{Aut}

\DeclareMathOperator{\wt}{wt}

\DeclareMathOperator{\Hom}{Hom}
\DeclareMathOperator{\End}{End}

\DeclareMathOperator{\imag}{Im}
\DeclareMathOperator{\id}{id}
\DeclareMathOperator{\ev}{ev}

\newarrow{DashTo}{}{dash}{}{dash}>

\renewcommand{\phi}{\varphi}

\newsavebox{\sembox}
\newlength{\semwidth}
\newlength{\boxwidth}

\newcommand{\series}[1]{%
\sbox{\sembox}{\ensuremath{#1}}%
\settowidth{\semwidth}{\usebox{\sembox}}%
\sbox{\sembox}{\ensuremath{\left[\usebox{\sembox}\right]}}%
\settowidth{\boxwidth}{\usebox{\sembox}}%
\addtolength{\boxwidth}{-\semwidth}%
\!\left[\hspace{-0.3\boxwidth}%
\usebox{\sembox}%
\hspace{-0.3\boxwidth}\right]\!\hspace{.06667em}%
}

\newsavebox{\lbox}
\newlength{\lwidth}
\newlength{\lboxwidth}

\newcommand{\laurent}[1]{%
\sbox{\lbox}{\ensuremath{#1}}%
\settowidth{\lwidth}{\usebox{\lbox}}%
\sbox{\lbox}{\ensuremath{\left(\usebox{\lbox}\right)}}%
\settowidth{\lboxwidth}{\usebox{\lbox}}%
\addtolength{\lboxwidth}{-\lwidth}%
\!\left(\hspace{-0.3\lboxwidth}%
\usebox{\lbox}%
\hspace{-0.3\lboxwidth}\right)\!\hspace{.06667em}%
}

\newtheorem{Theorem}{Theorem}[section]
\newtheorem{Conjecture}{Conjecture}[section]
\newtheorem{Lemma}{Lemma}[section]
\newtheorem{Proposition}{Proposition}[section]
\theoremstyle{definition}
\newtheorem{Definition}{Definition}[section]
\theoremstyle{remark}
\newtheorem{Remark}{Remark}[section]

\makeautorefname{Theorem}{Theorem}
\makeautorefname{Conjecture}{Conjecture}
\makeautorefname{Lemma}{Lemma}
\makeautorefname{Proposition}{Proposition}
\makeautorefname{Definition}{Definition}
\makeautorefname{Remark}{Remark}

\makeatletter
\let\c@Lemma\c@Theorem
\let\c@Conjecture\c@Theorem
\let\c@Proposition\c@Theorem
\let\c@Definition\c@Theorem
\let\c@Remark\c@Theorem
\makeatother

\begin{document}

\title{Effective Categorical Enumerative Invariants}

\author[C\u ald\u araru and Tu]{% 
Andrei C\u ald\u araru\;\; Junwu Tu}

\address{Andrei C\u ald\u araru, Mathematics Department,
University of Wisconsin--Madison, 480 Lincoln Drive, Madison, WI
53706--1388, USA, \texttt{andreic\@math.wisc.edu}.\newline
Junwu Tu, Institute of Mathematical Sciences, ShanghaiTech University,
Shanghai, 201210, China, \texttt{tujw\@shanghaitech.edu.cn}.} 
\begin{abstract}
  {\footnotesize
    We introduce enumerative invariants $F_{g,n}$ $(g\geq0$, $n \geq 1)$
    associated to a cyclic $A_\infty$ algebra and a splitting of its
    non-commutative Hodge filtration. These invariants are defined by
    explicitly computable Feynman sums, and encode the same
    information as Costello’s partition function of the corresponding
    field theory.

    Our invariants are stable under Morita equivalence, and therefore
    can be associated to a Calabi-Yau category with splitting data.  This
    justifies the name categorical enumerative invariants (CEI) that we
    use for them. 

    CEI conjecturally generalize all known enumerative invariants
    in symplectic geometry, complex geometry, and singularity theory.
    They also provide a framework for stating enumerative mirror
    symmetry predictions in arbitrary genus, whenever homological
    mirror symmetry holds.
  }
\end{abstract}

\maketitle
{\footnotesize \tableofcontents}

\section{Introduction}

The purpose of this paper is to introduce enumerative-type invariants
associated to a pair $(A,s)$ consisting of a smooth cyclic
$A_\infty$-algebra $A$ and a splitting $s$ of the non-commutative
Hodge filtration on the cyclic homology of $A$.  These invariants are
stable under Morita equivalence~\cite{AmoTu} and therefore the
definition can be used for any smooth and proper Calabi-Yau category
with a splitting of the Hodge filtration.  This justifies the name
{\em categorical enumerative invariants} (CEI) we use for these
invariants.

Following work of Costello~\cite{Cos1} the algebra $A$ determines an
open topological-conformal field theory (open TCFT) which extends in a
universal way to an open-closed TCFT.  The state space of the closed
sector of this open-closed TCFT is the Hochschild chain complex
$C_*(A)$ of $A$, with circle action encoded by the Connes differential
$B$.  The splitting $s$ gives a choice of trivialization of this
circle action.

Our CEI encode the same information as the partition function that
Costello associates to the above data~\cite{Cos2}.  The main difference
is that our invariants are effectively computable by explicit
Feynman-type sums that combine expressions in the closed TCFT and the
trivialization of the circle action.  See Section~\ref{sec:formulas} for
explicit formulas in low genera.

\subsection{Descendent Gromov-Witten invariants}

In order to explain what type of objects CEI are it will be helpful
to review the classical construction of descendent Gromov-Witten
invariants associated to a compact symplectic manifold $X$.

The state space of Gromov-Witten theory on $X$ is the graded vector
space $H=H^*(X,\Lambda)$, the singular cohomology of $X$ with
coefficients in an appropriate Novikov field $\Lambda$.  Gromov-Witten
theory defines descendent invariants of the form
\[ \langle \tau_{k_1}(\gamma_1),\ldots,\tau_{k_n}(\gamma_n)\rangle_g^X
  \in \Lambda\]
for any genus $g \geq 0$, number of insertions $n\geq 0$, insertions
$\gamma_1,\ldots,\gamma_n\in H$, and integers $k_1,\ldots,k_n\geq 0$
counting $\psi$-classes.

Let $u$ be a formal variable of homological degree $-2$, which will be
used to keep track of $\psi$-class insertions, and denote by $H_-$,
$H_+$, and $H_\Tate$ the graded vector spaces
\begin{align*}
  H_- & = H[u^{-1}], \\
  H_+ & = H\series{u}, \\
  H_\Tate & = H\laurent{u}. 
\end{align*}
The space $H_+$ is naturally the linear dual of $H_-$ by restriction
of the residue pairing, see~(\ref{subsec:circle})
\[ \langle\,-\,,\,-\,\rangle_\res : H_\Tate \otimes H_\Tate \ra \Lambda. \]

For fixed $(g,n)$ the descendent invariants above are packaged into a
single invariant $F_{g,n}^X \in \Sym^n(H_-)$ defined by the
requirement that
\[ \langle F_{g,n}^X, (\gamma_1 u^{k_1})\cdots(\gamma_n u^{k_n})
  \rangle_\res = \langle\tau_{k_1}(\gamma_1),\ldots,
  \tau_{k_n}(\gamma_n)\rangle_g^X \] 
for $\gamma_1, \ldots, \gamma_n\in H$, $k_1, \ldots, k_n\geq 0$.
(Finiteness properties of descendent invariants ensure
that $F_{g,n}^X$ is well-defined.)

For each $(g,n)$ the invariant $F_{g,n}^X$ can be determined by a
finite computation.  For this reason we choose these invariants as the
primary objects that we will generalize.  However, for many
applications it is useful to further package the collection of all
invariants $F_{g,n}^X$ into a single formal function $\cD^X$, called
the full descendent potential.  All descendent invariants of $X$ can
be recovered from $\cD^X$.  See Lemma~\ref{lem:mc-closed} for an
explicit formula for $\cD^X$ in terms of the invariants $F_{g,n}^X$.

\begin{Remark}
  The reader unfamiliar with Gromov-Witten theory might think that we
  introduce the above structures for purely formal reasons, in order
  to handle the complicated nature of the invariants $F_{g,n}^X$.
  There are deeper reasons to study these constructions, as observed
  by Givental~\cite{Giv}.  The above structures are fundamental in
  Gromov-Witten theory, as well as in the non-commutative context.
  Especially in the latter situation these structures appear from the
  very beginning: $u$ is the circle parameter associated to the Connes
  operator $B$ on the Hochschild complex, $H_-$ and $H_+$ are the
  cyclic and negative cyclic homology complexes of $A$, respectively,
  and the residue pairing is a generalization of the Mukai
  pairing~\cite{CalWil},~\cite{Shk}.
\end{Remark}

\subsection{Categorical enumerative invariants}
\label{cond:dagger}
Let $\bbK$ be a field of characteristic zero, and let $A$ be a
$\Z/2\Z$-graded cyclic $A_\infty$-algebra over $\bbK$.  We assume that
$A$ satisfies the conditions below:
\[
  \begin{array}{cl}
    (\dagger) \mbox{\quad\quad\quad\quad}&
                \begin{array}{l} A \mbox{ is smooth, finite
                  dimensional, unital, and satisfies the} \\
                  \mbox{Hodge-de Rham degeneration property}
                \end{array}
\end{array}
\] 
(By~\cite{Kal} Hodge-de Rham degeneration is automatic
if $A$ is $\Z$-graded.)

Let $d$ be the Calabi-Yau dimension of $A$, and denote by $H$ the
graded vector space $H=HH_*(A)[d]$.  This is the state space of the
closed TCFT associated to $A$ in~\cite{Cos1}.  Define $H_-$, $H_+$,
$H_\Tate$ as before, and define the residue pairing using the Mukai
pairing on $H$ as in~(\ref{subsec:circle}).

Let $s$ be a choice of splitting of the non-commutative Hodge
filtration of $A$ (Definition~\ref{def:splitting}).  We will
construct categorical enumerative invariants (CEI)
\[ F_{g,n}^{A,s} \in \Sym^n(H_-) \]
indexed by pairs of integers $(g,n)$ satisfying $g\geq 0$, $n\geq 1$,
$2g-2+n>0$.  These CEI are defined by explicit Feynman sums over a new
type of graphs, called partially directed stable graphs
(Definition~\ref{defi:dir-graph}).  The CEI $F_{g,n}^{A,s}$ are expected
to play in the non-commutative setting a role similar to that of the
invariants $F_{g,n}^X$ in Gromov-Witten theory.  In particular, basic
CEI invariants with insertions $\gamma_1, \ldots, \gamma_n \in H$ can
be defined from $F_{g,n}^{A,s}$ by
\[ \langle
  \tau_{k_1}(\gamma_1),\ldots,\tau_{k_n}(\gamma_n)\rangle_g^{A,s} = 
  \langle F_{g,n}^{A,s}, (\gamma_1 u^{i_1})\cdots(\gamma_n u^{i_n})
  \rangle_\res.\] 

Despite the apparent similarity with notation for Gromov-Witten (GW)
or Fan-Jarvis-Ruan-Witten (FJRW) invariants, the basic numerical
invariants that are extracted from CEI are complex, not rational
numbers. The more classical enumerative invariants are conjecturally
the Taylor coefficients of the function obtained by computing the
above basic CEI in a versal family of deformations of the input data
$(A,s)$.

\begin{Remark}
Costello~\cite{Cos2} and Kontsevich-Soibelman~\cite{KonSoi}
were the first to observe that the proper input for the construction
of categorical enumerative invariants should include not just the
algebra $A$ but also a choice of splitting of the non-commutative
Hodge filtration. As explained above, this is a choice of
trivialization of the circle action given by the Connes operator. From
the operadic point of view, the importance of such a trivialization
was explored in~\cite{Dru,DSV,DSV2,KMS,Des}.

Our construction of CEI is heavily influenced by work of
Costello~\cite{Cos2}.  The invariants we define contain the same
information as Costello's partition function associated to the TCFT
corresponding to $A$, with circle action trivialized by $s$ (in our
notation, the total descendent potential $\cD^{A,s}$,
Definition~\ref{def:main}).  Our main contribution is to give explicit
formulas for the individual components $F_{g,n}^{A,s}$ of $\cD^{A,s}$.
\end{Remark}

Two new ideas enable the definition of CEI in an effectively
computable manner:
\begin{itemize}
\item[(1)] Costello's original construction uses geometric, symmetric
  string vertices $\cV_{g,n}\in C_*(M_{g,0,n})$ which have no inputs.
  Because of this his string vertices can not be represented by
  ribbon graphs, and do not act directly on a positive boundary TCFT.
  We replace them by a Koszul resolution
  \[ \hcV_{g,n}\in \bigoplus_{k\geq 1} C_*(M_{g,k,n-k}) \]
  which does admit a combinatorial representation by ribbon graphs. 

\item[(2)] We prove a constructive formality result for the
  differential graded Lie algebras that appear in the theory of
  Batalin-Vilkovisky algebras with trivialized circle action.  The
  operations in the resulting trivialized algebra are expressed as sums
  over partially directed graphs.  This gives explicit formulas for
  the final step of Costello's construction, which was previously
  handled by a non-constructive argument.
\end{itemize}

Idea (1) is originally due to Costello, and it was made precise
in~\cite{CCT}.  Idea (2) appeared in a slightly different form
in~\cite{CalTu2}; however, the definition of PD graphs and the
arguments involved in the homological perturbation argument were
incomplete and are corrected in the current version.

{\bf The main theorem}\ 
In order to state our main theorem we need to introduce
an embedding
\[ \biota: \Sym^n H_- \ra \Hom ( H_+, \Sym^{n-1} H_- ).\]
One can think of as an analogue of the Koszul or de Rham
differentials, see~(\ref{subsec:Koszul}).

Our main result provides the following explicitly computable formula
for $\biota (F_{g,n}^{A,s})$.  It determines $F_{g,n}^{A,s}$ because
$\biota$ is injective.

\begin{Theorem}
  \label{thm:intro}
  Let $A$ be a cyclic $A_\infty$ algebra satisfying $(\dagger)$, and let
  $s$ be a splitting of the non-commutative Hodge filtration of $A$ in
  the sense of Definition~\ref{def:splitting}.  Then the genus $g$,
  $n$-point CEI of the pair $(A,s)$ is determined by
  \[\biota (F_{g,n}^{A,s}) = \sum_{m\geq 1} \sum_{\GG\in
      \Gamma(\!(g,1,n-1)\!)_m} (-1)^{m-1} \frac{\wt(\GG)}{\Aut(\GG)} \prod_{v}
    {\sf Cont} (v) \prod_{e} {\sf Cont} (e) \prod_{l} {\sf Cont} (l),\]
  where
  \begin{enumerate}
  \item The second summation is over partially directed stable graphs
    (Definition~\ref{defi:dir-graph}) of genus $g$, with $1$ input leg,
    $n-1$ output legs, and $m$ vertices.
  \item The weight $\wt(\GG)\in \mathbb{Q}$ is a rational number
    associated to $\GG$, see~(\ref{subsec:weight}).
  \item The contributions at vertices ${\sf Cont}(v)$ come from
    combinatorial string vertices (Definition~\ref{defi:vertices}) and
    depend only on $A$, while the edge and leg contributions are given
    by explicit formulas depending on the splitting map $s$.
  \end{enumerate}
\end{Theorem}

We refer the reader to Theorem~\ref{thm:main} for more details on the
above formula.

\begin{Remark}
CEI, as defined in this paper, depend on three pieces of data:
\begin{enumerate}
\item[--] an $A_\infty$ algebra $A$;
\item[--] a cyclic structure on $A$;
\item[--] a splitting $s$ of the non-commutative Hodge filtration of
$A$.
\end{enumerate}
The cyclic structure is in some sense too strict: it does not
pull-back via quasi-isomorphisms. This makes proving Morita invariance
of CEI not obvious.  Amorim-Tu~\cite{AmoTu2} show that CEI only depend
on the underlying Calabi-Yau structure induced by the cyclic
structure, from which the full Morita invariance of CEI can be
proved. Thus we may replace $A$ by any Calabi-Yau category $\cC$
Morita equivalent to it.
\end{Remark}

{\bf Related constructions}\ \, A different approach to obtain
invariants from Calabi-Yau categories was proposed by
Barannikov-Kontsevich~\cite{Bar} using the idea of variation of
semi-infinite Hodge structures (VSHS). Their work was extending
pioneering works of Saito~\cite{Sai1},~\cite{Sai2} on unfoldings of
singularities. This approach was further advanced by
Ganatra-Perutz-Sheridan~\cite{GPS} to prove genus zero Hodge-theoretic
mirror symmetry from homological mirror symmetry.  Unlike the VSHS
approach, which can only recover the genus zero part of the theory, our
approach accesses all genera.  See~\cite[Conjecture 1.8]{CLT} for a
precise comparison conjecture between the two approaches.

\subsection{Comparisons to existing theories}
\label{subsec:compa}

The primary motivation for studying CEI theory goes back to the
origin of the homological mirror symmetry conjecture.  Kontsevich, in
his visionary 1994 address to the International Congress of
Mathematicians~\cite{Kon}, predicted that Gromov-Witten theory should
be of a categorical nature.  More precisely, he proposed that it
should be possible to extract the Gromov-Witten invariants of a
compact symplectic manifold $X$ directly from its Fukaya category
$\Fuk(X)$.  This is one of the goals that CEI theory aims to
accomplish.

A second goal of CEI theory is to attempt to unify virtually all known
enumerative invariants in the literature, by taking the algebra $A$ to
be Morita equivalent to
\begin{itemize}
\item[(A)] the Fukaya category, for Gromov-Witten theory;
\item[(B)] the wrapped Fukaya category, for FJRW theory;
\item[(C)] the category of matrix factorizations, for Givental-Saito
  theory.
\end{itemize}
See below for cases (D) equivariant Givental-Saito theory and (E)
BCOV theory.

Yet a third motivation to develop CEI theory is that it provides a
potential path for proving enumerative mirror symmetry at all genera.
Enumerative mirror symmetry for CEI follows tautologically from
homological mirror symmetry.  This was Kontsevich's original intended
explanation for the equality of numerical invariants on the two sides
of mirror symmetry observed by Candelas-de la
Ossa-Green-Parkes~\cite{COGP}.  However, despite the existence of this
tautological result, we cannot yet claim that homological mirror
symmetry implies enumerative mirror symmetry at all genera: the main
difficulty lies in identifying the categorical invariants with the
geometric ones. Thus it is important to study comparison results
between CEI and existing theories.

For example, in the case of the comparison in (A) it is expected that
under the identification
\[ \mathcal{OC}: HH_*\left (\Fuk(X)\right)[-d] \cong H^*(X,\Lambda) \]
the descendent CEI potential agrees with the corresponding potential
in Gromov-Witten theory,
\[ \cD^{\Fuk(X),\,s^{\mathcal{OC}}} = \cD^X.\]
(The full descendent potential $\cD^{A,s}$ of CEI theory is defined
like the full descendent potential of Gromov-Witten theory, see
Lemma~\ref{lem:mc-closed}.)  Here $\mathcal{OC}$ denotes the
open-closed map in symplectic geometry, and the splitting
$s^{\mathcal{OC}}$ is defined using a cyclic extension of
$\mathcal{OC}$~(\cite{Gan,Hug}). We refer the reader to~\cite{AmoTu2}
for a definition of $s^{\mathcal{OC}}$ and a precise statement of the
above comparison with Gromov-Witten theory.

{\bf Equivariant Givental-Saito theory}\ \, In the categorical
approach all the above theories can be directly generalized to the
setting where a finite group of symmetries is present: the input
category is simply replaced by the corresponding smash product
construction. In particular, this allows us to define orbifold
Givental-Saito theory at all genera, for which no direct definition
exists. Explicitly, we take $A$ to be Morita equivalent to

\begin{itemize}
\item[(D)] the category of $G$-equivariant matrix factorizations, for
orbifold Givental-Saito theory.
\end{itemize}

This yields new $B$-model invariants associated with Landau-Ginzburg
orbifolds for an arbitrary finite symmetry group $G$. In particular,
it enables us to formulate the LG/LG mirror symmetry conjecture for
arbitrary finite symmetry groups between mirror dual invertible
polynomials.

\begin{Conjecture}
  Let $(W,G)$ and $(W^t,G^t)$ be mirror dual pairs of Landau-Ginzburg
  models~\cite{Kra}. Then we conjecture that there is an equality of
  descendent potentials
  \[ \cD^{{\sf MF}_G(W),\,s^{\sf can}} = \cD^{(W^t,G^t)}_{\sf FJRW},\]
  under the state space identification of
  Chiodo-Ruan~\cite{ChiRua}. Here the splitting $s^{\sf can}$,
  originally introduced in~\cite{HLSW} in the geometric setting, is the
  unique $G_W^{\sf max}$-equivariant splitting of the category ${\sf
    MF}_G(W)$ in~\cite{Tu}.
\end{Conjecture}

This conjecture was proved by He-Li-Shen-Webb~\cite{HLSW} in the
geometric setting, but only for the trivial orbifold group $G=\{\id\}$.

{\bf BCOV invariants}\ \, Last, but not least, we may also take $A$
to be Morita equivalent to
\begin{itemize}
\item[(E)] the derived categories of the fibers of a smooth projective
  family of Calabi-Yau varieties, for BCOV invariants~\cite{BCOV}.
\end{itemize}
This is arguably the most interesting application of CEI theory, as it
allows one to precisely define BCOV invariants in mathematics.
(The importance of BCOV invariants in physics comes from their
interpretation as partition functions in topological string
theory~\cite{BCOV}.)  This idea was already explored by
Costello~\cite[Section 13]{Cos2}.  Note that there is also an
alternative, geometric/analytic approach to the construction of BCOV
invariants developed by Costello-Li~\cite{CosLi} and~\cite{Li}.

We use the construction of CEI to make the definition of these
invariants completely transparent. Conjecturally, this gives the first
mathematical construction of the BCOV partition function~\cite{BCOV}
for general smooth and proper Calabi-Yau varieties.

More precisely, for a smooth and proper Calabi-Yau variety $X$ over
$\bbC$ consider the differential graded (dg) category 
\[ \mbox{$\cC_X$} = D^b_{dg}\big({\sf Coh}(X)\big). \] Based on
To\"en's proposal~\cite{Toe}, Blanc~\cite{Bla} developed a version of
complex topological $K$-theory for $\mathbb{C}$-linear differential
graded categories.  Assuming~\cite[Conjecture 1.14 (2)]{GPS}, an
immediate consequence of Blanc's construction is that the category
$\cC_X$ admits an intrinsic splitting of the non-commutative Hodge
filtration.  Under the comparison result~\cite[Theorem 1.1 (b),
(d)]{Bla} this splitting then corresponds to the complex
conjugate splitting of the classical Hodge filtration of
$H^*(X,\mathbb{C})$ (see~\cite[Section 4.6]{Bla}).  We shall denote
this splitting by $s^\mathsf{BT}$ for ``Blanc-To\"en''.

CEI theory in this case yields complex number invariants
\[ \langle \tau_{k_1}(\gamma_1), \ldots,
  \tau_{k_n}(\gamma_n)\rangle_g^{\cC_X,\,s^{\sf BT}} \in \mathbb{C} \]
for $g\geq 0$, $n\geq 1$, $2g-2-n>0$, $k_1,\ldots,k_n\geq 0$, and 
insertions $\gamma_1,\ldots,\gamma_n \in HH_*(X)[d]$.  We
expect that the functions on the moduli space of Calabi-Yau varieties,
obtained from the above numbers by varying $X$ in a versal family,
encode the BCOV invariants from physics~\cite{BCOV}. 

\subsection{A look ahead}

Despite the length of the paper, the main construction can be
understood as a wide-ranging algebraic generalization of a fairly
simple geometric construction which we describe in the next section.
The content of that section is key (as motivation) to understanding
the paper, even though the rest of the paper is independent of it.  A
more detailed guide to the remainder of the paper is provided
in~(\ref{subsec:plan}). 

\subsection{Acknowledgments}
We thank Kevin Costello, Nick Sheridan, Si Li, and Dima Arinkin for
patiently listening to the various problems we ran into at different
stages of the project and for providing insight. Andrei
C\u{a}ld\u{a}raru was partially supported by the National Science
Foundation through grant numbers DMS-1811925, DMS-2152088, DMS-2202365
and by the Wisconsin Alumni Research Foundation through a Vilas
Associate Professorship. Junwu Tu was partially supported by the NSFC
grant 12071290, and STCSM grant 20JC1413300.

\section{A motivating example}
\label{sec:example}

In the first part of this section we work through the explicit example
of computing the invariant
\[ \biota(F_{1,1})(\gamma) = \langle \gamma \rangle_{1,1}^X \]
for a compact symplectic manifold $X$ and an insertion
$\gamma\in H^2(X, \C)$.  Comparing the geometric Gromov-Witten
computation with the formula~(\ref{subsec:11}) for $\biota(F_{1,1})$
provides geometric intuition for the algebraic constructions of CEI
theory that will be developed later in the paper.

In the second part we give an outline of the algebraic generalizations
of the constructions used in the above example.  This outline frames
the remainder of the paper.

\subsection{String vertices}
We will work with moduli spaces $M_{g,k,l}$ parametrizing Riemann
surfaces of genus $g$ with $k+l$ marked points, $k$ of which are
designated as inputs and the rest as outputs.  Just like the ordinary
Deligne-Mumford moduli spaces of curves, the spaces $M_{g,k,l}$ admit
compactifications $\Mbar_{g,k,l}$ parametrizing stable curves.

We will also need a framed version $M_{g,k,l}^\fr$ of $M_{g,k,l}$: its
points parametrize the same Riemann surfaces as before, but around
each marked point a small disk has been excised, and a framing of the
boundary circle thus created has been fixed by marking a distinguished
point on this circle.  One can think of a framing around a point $P$
as fixing (up to homotopy) the complex structure of a small
neighborhood of $P$ on the original surface.  (We use the fact that
the moduli space of holomorphic structures on an annulus with framed
boundary is homotopic to $S^1$.)  There is a natural forgetful map
$M_{g,k,l}^\fr \ra M_{g,k,l}$ which is an $(S^1)^{k+l}$-fibration.

The homotopy type of the complement, inside $\Mbar_{g,k,l}$, of an
$\varepsilon$-tubular neighborhood of the boundary does not depend on
$\varepsilon$ as long as it is small enough.  This complement defines
a (non-closed) rational singular chain $\cV_{g,k,l}$ on $M_{g,k,l}$,
unique up to homotopy, which will be called the string vertex of type
$(g,k,l)$.  The chains $\cV_{g,k,l}$ were first introduced in physics
by Sen-Zwiebach~\cite{SenZwi}; a precise mathematical definition of
them was given by Costello~\cite{Cos2}.  (To be precise, the string
vertex $ \cV_{g,k,l}$ will be constructed as an
$(S^1)^{k+l}$-equivariant chain on $M_{g,k,l}^\fr$. We will largely
ignore this distinction in the informal discussion in this section.)

For our example we need to understand $\cV_{0,1,2}$ and
$\cV_{1,1,0}$.  All string vertices are defined recursively starting
from $\cV_{0,1,2}$, which is taken to be one-half the
fundamental class of $M_{0,1,2} = \pt$, a closed chain in
$C_0(M_{0,1,2})$.  (The coefficient $1/2$ is not important at this
point, but the reason for it is the fact that the two outputs of
$M_{0,1,2}$ are symmetrized.)

To define $\cV_{1,1,0}$ we use the natural operator
\[ \Delta:C_0(M_{0,1,2}) \ra C_1(M_{1,1,0}) \]
which twist-sews the two marked outputs of a Riemann surface 
of type $(0,1,2)$. The twist-sew operation cuts small disks around the
output markings, frames the resulting boundary circles, and sews them
in all the $S^1$ possible ways to do this without necessarily aligning
the framings.  Note that the choice of framing does not 
matter in the end.

It is easy to visualise $\Delta(\cV_{0,1,2})$. It is a closed 1-chain
in $M_{1,1,0}$ whose points parametrize those elliptic curves which,
in the hyperbolic metric, have a shortest geodesic of length a fixed
small positive number (the circumference of the small disks that were
excised from the sphere $\C\pj^1$ before twist-sewing).  When this
circumference is sufficiently small, $\Delta(\cV_{0,1,2})$ will be a
small circle surrounding the cusp point $\infty$ in
$\Mbar_{1,1,0} = S^2$.  (We ignore the fact that these moduli spaces
have non-trivial stack structure, as we are only interested in their
rational chains/homology.)  See the figure below.

The crucial observation, due to Costello, is that $\Delta(\cV_{0,1,2})$
bounds a unique (up to homotopy) chain
$\cV_{1,1,0} \in C_2(M_{1,1,0})$.  In other words, the equation
\[ \del(\cV_{1,1,0}) + \Delta(\cV_{0,1,2}) = 0 \] admits a unique
solution $\cV_{1,1,0}$.  (The above is a particular case of the
quantum master equation that will be used later to define all
string vertices.)  As mentioned earlier, $\cV_{1,1,0}$ is the
complement of the small tubular neighborhood $D$ (in this case, a
disk) of the boundary $\del\Mbar_{1,1,0} = \{\infty\}$ in
$\Mbar_{1,1,0}$.

\[ \includegraphics[scale=.35]{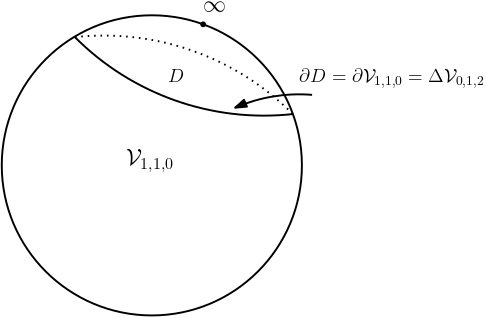} \]

\subsection{Field theories}
There are two field theories associated to a symplectic manifold
$X$.  One is the cohomological field theory (CohFT) of $X$ defined
by Kontsevich-Manin~\cite{KonMan}.  The other is a (conjectural)
geometric TCFT constructed from the Fukaya category of $X$.  We need
to understand the relationship between these two theories.

In the CohFT approach we associate to an insertion $\gamma\in H^2(X)$
the closed differential 2-form $\Omega_{1,1}(\gamma)$ on $\Mbar_{1,1}$
defined by
\[ \Omega_{1,1}(\gamma) = \pi_*(\ev_1^* \gamma). \]
Here $\pi$ and $\ev_1$ are the forgetful maps from $\Mbar_{1,1}(X)$ to
$\Mbar_{1,1}$ and to $X$, respectively; the push-forward map
$\pi_*$ uses the virtual fundamental class of $\Mbar_{1,1}(X)$.  The
Gromov-Witten invariant $\langle \gamma \rangle_{1,1}^X$ is recovered
by integrating $\Omega_{1,1}(\gamma)$ along $\Mbar_{1,1}$:
\[ \langle \gamma \rangle_{1,1}^X= \int_{\Mbar_{1,1}}
  \Omega_{1,1}(\gamma). \]
(We ignore the role of the homology class $\beta\in H_2(X)$, which is
incorporated in the Novikov ring $\Lambda$.)

A standard (folklore) conjecture asserts that, in parallel to the
construction of this CohFT, one should be able to construct a
geometric TCFT from the Fukaya category of $X$.  Its state space $V$
should be the dg-vector space of Floer cochains on $X$, whose
cohomology is identified with $H^*(X,\C) \otimes \Lambda$.

This geometric TCFT should contain information closely related to that
provided by the CohFT.  In particular, one of its structure maps
\[ C_*(M_{1,1,0}) \otimes V \ra \Lambda \]
should compute, for $\alpha\in C_*(M_{1,1,0})$, the integral
\[ \int_\alpha \Omega_{1,1}(\gamma). \]
(We ignore in this part of the discussion the distinction between
framed and unframed moduli spaces. To be correct, the above structure
map should take as input chains in the {\em framed} moduli space.)

Note that {\em a priori} $\Omega_{1,1}(\gamma)$ is not defined when
$\gamma$ is not closed; nevertheless, by abuse of notation we will
denote the value of the above TCFT map by
$\int_\alpha \Omega_{1,1}(\gamma)$ even when $\gamma$ is not closed.

Summarizing the above discussion, the main differences between the
TCFT and the CohFT are:
\begin{enumerate}
\item[--] in the TCFT we allow insertions (forms) $\gamma$ that
  are not closed, while in the CohFT we only allow closed ones
  (cohomology classes);
\item [--] in the CohFT we are allowed to integrate along chains in
  $\Mbar_{1,1}$; the TCFT only allows chains in $M_{1,1}$.
\end{enumerate}

\subsection{The main computation}
Our goal is to express the CohFT computation
\[ \langle \gamma \rangle_{1,1}^X  = \int_{\Mbar_{1,1}}
  \Omega_{1,1}(\gamma) \]
in terms of integrals available in the TCFT.  As a first step we
decompose the above integral as
\[ \int_{\Mbar_{1,1,0}} \Omega_{1,1}(\gamma) = \int_{\cV_{1,1,0}}
  \Omega_{1,1}(\gamma) + \int_D \Omega_{1,1}(\gamma). \]
The first term in this integral is readily available in the TCFT: the
chain $\cV_{1,1}$ is entirely contained in the uncompactified moduli
space $M_{1,1}$.

The second term can not be directly expressed in terms of the TCFT; in
fact its value is not determined by the TCFT alone.  Nevertheless, the
TCFT does assign a value to integrals of the form
\[ \int_{\del D} \Omega_{1,1}(\rho) \]
for 1-forms $\rho$ on $X$, possibly non-closed.  If we assume that the
TCFT compactifies (i.e., it can be extended to include integrals along
chains in $\Mbar_{1,1}$), Stokes' theorem implies that for a 1-form
$\rho$ we will have
\[ \int_{\del D} \Omega_{1,1}(\rho) = \int_{D} d\Omega_{1,1}(\rho) =
  \int_{D} \Omega_{1,1}(d\rho). \]

Let $P$ and $Q$ denote the operators $V \ra \Lambda$ obtained by
integrating a form along $D$ and $\del D$, respectively.  Then the
above equality can be written (up to sign) as
\[ Q = [d, P]. \]
In other words, if the TCFT compactifies, then the operator $Q$ (part
of the uncompactified TCFT structure) must be homotopy-trivial; the
choice of $P$ (which only makes sense in the compactified TCFT) must
be a null-homotopy trivializing $Q$.

Why should we expect the analogue of the operator $Q$ to be homotopy
trivial in the setting we are interested in, where the TCFT's circle
action has been trivialized?  For arbitrary $(g,k,l)$ there is no
reason to expect the analogue of the chain $\del D$ to be homotopy
trivial in the uncompactified space $M_{g,k,l}$ (the case
$(g,k,l) = (1,1,0)$ is misleading).  We want to ``close it up'' by
adding the missing boundary to $M_{g,k,l}$.  Nevertheless, we
want to argue that, if the circle action is homotopy trivial on $V$,
the action of $\del D$ on $V$ via the TCFT will also be homotopy
trivial.

Note that the chain $\del D$ is not arbitrary; it has been
constructed as $\del D = \Delta(\cV_{0,1,2})$.  The operation $\Delta$
includes a circle action in it, as it can be decomposed into the
following three operations applied to the 0-chain $\cV_{0,1,2}$:
\begin{itemize}
\item[1.] cut small disks around the marked outputs of a surface
  parametrized by $\cV_{0,1,2}$, and choose framings of each
  boundary circle; 
\item[2.] rotate the framing of one of the two boundary circles
  through all $S^1$ possible positions;
\item[3.] sew the two boundary circles while aligning the markings.
\end{itemize}
The second operation is the $S^1$ action applied to one of the
framings.  Thus, while the operator $\Delta$ on $C_*(M_{g,k,l}^\fr)$ is not
homotopy trivial in general, its action $Q$ on the TCFT will be, under the
assumption that the $S^1$ action on the TCFT is homotopy trivial.

All that is left at this point is to construct, explicitly, a homotopy
trivialization for $\Delta(\cV_{0,1,2})$ from a given trivialization
of the circle action.  What we will do, in fact, is to give a formula
for a homotopy trivialization $H$ of the composition $\Omega$ of the
actions of operations 2 and 3 above.  The TCFT action of these two
operations is given by
\[ \Omega(x,y) = \langle Bx, y \rangle \]
where $B$ denotes the circle action on $V$ and the pairing is the
Mukai pairing on Hochschild chains, see~(\ref{subsec:pre-example}).
If $B$ is trivialized by a splitting $s$, an explicit formula for a
trivialization $H$ of $\Omega$ is given in
Proposition~\ref{prop:homotopy}.

The formula~(\ref{subsec:11}) for $\biota(F_{1,1})$ expresses
precisely the above decomposition:
\[ \biota(F_{1,1}) = \includegraphics[trim=-1cm 1cm 0 0, scale=.4]{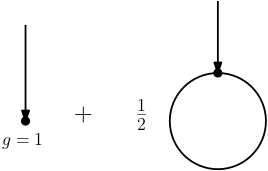}. \]
The first term denotes the integral along $\cV_{1,1,0}$ (the dot
denotes the string vertex of genus 1, with one input and zero
outputs), while the second one is the result of integrating along
$\cV_{0,1,2}$ (again, the vertex denotes the string vertex of genus 0,
with one input and two outputs), and applying to the result the
operator $H$ which trivializes $\Omega$ (the closing of the loop).

Note that it is not {\em a priori} clear that every extension of a
given TCFT to a compactified one comes from a choice of trivialization
of the circle action.  However, results in
\cite{KonSoi,Cos2,Dru,DSV,DSV2,KMS} and especially recent results
in~\cite{Des} suggest that this indeed is the case, so our approach to
constructing the Gromov-Witten compactified TCFT from the
uncompactified one and a trivialization of the circle action is justified.

\subsection{Algebraic generalization}
\label{subsec:alggen}
The entire discussion above admits an algebraic interpretation in
terms of chasing certain special Maurer-Cartan elements in the
following diagram of differential graded Lie algebras (DGLAs)
\[
  \begin{diagram}[height=2em,width=2.5em]
    \g & & \rDashTo  & & \h_A & \rTo^\cK & \h^{\triv}_A \\
    \dTo_\biota & & & & \dTo_\biota & & \dTo_\biota \\
    \hg & \lTo & \hg^{\sf comb} & \rTo^{\rho^A} &  \hh_A &
    \rTo^\hcK &  \hh_A^\triv.
  \end{diagram}
\]
Both the abstract construction of CEI and the explicit formulas for
them stated in Theorem~\ref{thm:intro} can be deduced from this
diagram.

We now give a brief description of the DGLAs and the maps between them
in the above diagram; the construction of this diagram is essentially
the contents of the paper.  The algebras in the top row of this diagram
are already recognizable from Costello's work~\cite{Cos2}.  The
contribution of this paper is the construction of the quasi-equivalent
second row, which is explicitly computable; in particular we give
explicit formulas for the quasi-isomorphisms $\cK$ and $\hcK$ which
are new.

Costello constructed the DGLA $\g$ from singular (topological) chains
on the moduli spaces of curves $M_{g,0,n}$.  This DGLA does not depend
on either the algebra $A$ or the splitting $s$.  (In physics,
Sen-Zwiebach call this {\em background independent} field theory.)  By
contrast, $\h_A$ and its abelianization $\h_A^\triv$ are constructed
from symmetric powers of Hochschild chains of the algebra $A$.  These
algebras do depend on $A$, but not on the splitting $s$.  The map
$\cK$, which is a quasi-isomorphism, can only be constructed in the
presence of $s$.

Costello proved that in $\g$ there is a distinguished solution of the
Maurer-Cartan equation, the so-called string vertex $\cV$.  Its
component in $C_*(M_{g,0,n})$ is precisely the string vertex
$\cV_{g,0,n}$ that has appeared earlier in this section.  Up to homotopy
the string vertex $\cV$ does not depend on any choices whatsoever.

If the dotted arrow were made explicit (and, conceptually, the
difficulty in constructing it, and CEI as a whole, is related to the
problem of extending a positive-boundary closed TCFT to an arbitrary
one) then the image $\beta^A\in \h_A$ of $\cV$ would carry precisely
the same information as Costello's {\em abstract} Gromov-Witten
potential.  Intuitively, $\beta^A$ carries the information of the TCFT
integrals along the string vertices (for all possible insertions),
i.e., one can compute from it terms similar to the integrals
$\int_{\cV_{1,1,0}} \Omega_{1,1}(\gamma)$ and
$\int_{\cV_{0,1,2}} \Omega_{0,3}(\gamma)$ that have appeared in the
calculation of the $(1,1)$-invariant.
  
This abstract potential $\beta^A$, which depends only on the algebra
$A$, is a special solution of the (complicated) quantum master
equation in $\h_A$.  However, $\beta_A$ is not naturally a power
series of Hochschild (or cyclic) homology classes, as we would like
the CEI to be.  To obtain the desired type of invariants we need to
use a splitting $s$ in order to further trivialize the DGLA $\h_A$.  We
set the bracket of $\h_A$ to zero in $\h_A^\triv$ and use $s$ to
construct the formality quasi-isomorphism $\cK:\h_A \ra \h_A^\triv$.
The element $\cK(\beta_A)$, which by definition is a solution of the
Maurer-Cartan equation in $\h_A^\triv$, becomes precisely an element
in the homology of $\h_A^\triv$.  This is precisely where we want the
descendent CEI potential $\cD^{A,s}$ to live, so it is reasonable to
take $\cD^{A,s} = \cK(\beta_A)$.

The problem with this approach is that to the best of our knowledge
there is no way to make the dotted map in the diagram explicit.  (This
problem was recognized by Costello, and his work involves a great deal
of effort to circumvent this issue in a non-constructive way.)  We
solve this difficulty by replacing all the DGLAs above by Koszul
resolutions of themselves, denoted by $\hg$, $\hh_A$, etc.  The
horizontal maps in the second row of the diagram can now be made
explicitly computable.  This involves further replacing the DGLA
$\hg$, which is still built from singular chains on moduli spaces, by
a DGLA constructed from ribbon graphs, following ideas of
Egas~\cite{Ega}.  This yields the combinatorial comparison
quasi-isomorphism $\hg^\comb\ra \hg$.  Note that even though the
vertical Koszul resolution maps $\biota$ are quasi-isomorphisms, their
inverses are not explicit and this prevents one from defining the
dotted map $\g \dashrightarrow \h_A$.

In practice we completely disregard the left vertical of the diagram.
Because both maps $\g \ra \hg$ and $\hg^\comb\ra \hg$ are
quasi-isomorphisms, Costello's theorem tells us that, up to homotopy,
there is a unique solution to the Maurer-Cartan equation in
$\hg^\comb$, and by uniqueness this must be the image of the original
string vertex $\cV$.  We call this solution the combinatorial string
vertex $\cV^\comb\in \hg^\comb$; it can be computed directly in
$\hg^\comb$, without reference to $\g$ or $\hg$.  By the commutativity
of the diagram, the image $\hcK(\rho^A(\cV^\comb)) \in \hh_A^\triv$ is
$\biota(\cD)$.  Breaking it down by genus and number of insertions
gives the desired definition of $\biota(F_{g,n}^{A,s})$.

\subsection{Plan of the paper}
\label{subsec:plan}
Section~\ref{sec:pre} collects preliminaries regarding complexes with
pairing and circle action, partly in order to fix notations. In
Section~\ref{sec:dgla} we present constructions of the DGLAs $\h_A$
and $\hh_A$ associated with the $A_\infty$-algebra $A$. In
Section~\ref{sec:vertices} we present constructions of the
Sen-Zwiebach DGLA $\g$ and its Koszul resolution $\hg$, as well as of
the combinatorial version $\hg^{\sf comb}$ of $\hg$.  We introduce the
notion of (combinatorial) string vertices, which are the components of
the unique solution of the Maurer-Cartan equation in these DGLAs.
Section~\ref{sec:def} gives a detailed account of the definition of
categorical enumerative invariants (CEI) following the ideas
above. Logically this definition relies on Lemma~\ref{lem:extension}
proved in the subsequent section~\ref{sec:trivial-1}.

Sections~\ref{sec:trivial-1} and~\ref{sec:trivial-2} deal with the
construction, in terms of a splitting $s$, of the $L_\infty$
trivializing maps $\cK$ and $\hcK$.  In the final section
(Section~\ref{sec:mainthm}) we obtain the Feynman sum formulas for CEI
described in Theorem~\ref{thm:intro}.

\section{Preliminaries}

\label{sec:pre}
\subsection{Notations and conventions}
\label{subsec:signs}
Unless otherwise stated, all our vector spaces will be $\Z/2\Z$-graded
over a field $\bbK$ of characteristic zero.  (In the $\Z$-graded case
we obtain an analogue of the dimension axiom for Gromov-Witten theory
in Appendix~\ref{app:z-graded}.)  Nevertheless, we think of our
chain complexes as being {\em homologically} graded: differentials
decrease degree. For a $\Z$-graded vector space $V=\oplus_n V_n$ we
use the homological grading convention: the shift $V[k]$ is the graded
vector space whose $n$-th graded piece is $V_{n-k}$. In the
$\Z/2\Z$-graded case we shall still use the same notation $[k]$ even
though only the parity of $k$ matters.

Secondly, we shall always use the shifted sign conventions described
in~\cite{Cho,She}.  The shifted degree $|x|'$ of an element
$x$ is defined to be
\[ |x|' = |x|+1. \]
All the operations in cyclic $A_\infty$-algebras, including the
pairing, and in $L_\infty$-algebras, including DGLAs, differ by the
sign
\[ (-1)^{\sum_{k=1}^n (n-k)|x_k|} \]
from the usual ones when applied to a tensor
$x_1\otimes\cdots\otimes x_n$.  In particular a Maurer-Cartan element
in a DGLA has shifted degree zero.

Thirdly, we shall use {\em sign diagrams} as introduced by
Sheridan~\cite[Appendix C]{She} to deal with signs involved when
composing multi-linear maps. For example, a multi-linear map $\phi$
with $2$ inputs and $3$ outputs will be drawn as:
\[\includegraphics[scale=.7]{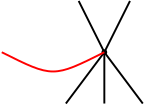}\]
The extra red line has degree $|\phi|$. As explained in {\em
  loc. cit.}, adding the red line makes the total degrees around the
vertex even, which allows one to move around vertices in a sign
diagram.  In other words, signs are invariant under isotopies of sign
diagrams.

\subsection{Representations of annuli}
\label{subsec:circle}
We shall frequently use complexes with circle action, together with a
compatible pairing on them.  From the TCFT point of view these are
precisely the same thing as the representations of the sub-PROP of
annuli (see ~\cite{Cos2} and Section~\ref{sec:vertices}).

A circle action on a chain complex $(C, \partial)$ is given by an
operator $\delta: C \ra C$ of degree one, satisfying $\delta^2=0$ and
$[\partial,\delta]=\partial\delta+\delta\partial=0$.  Associated with
a chain complex with circle action we set
\begin{align*}
  C^\Tate& =\big( C\laurent{u}, \partial+u\delta\big)\\
  C_+& = \big( C\series{u}, \partial+u\delta\big)\\
  C_-& = C^\Tate/C_+=\big( C[u^{-1}], \partial+u\delta\big),
\end{align*}
where $u$ is a formal variable of even degree.  Observe that there is
a short exact sequence
\[ 0 \ra u\cdot C_+ \ra C^\Tate \ra C_- \ra 0.\]

For a given chain complex $(C,\partial)$ the shift $C[1]$ is endowed
with the differential $-\partial$.  For example, the differential of
the complex $C_+[1]$ is given by $-\partial - u\delta$. To avoid
confusions, we shall reserve the notation $\eth$ for all circle
equivariant boundary maps. For example, when considering the
$\Hom$-complex $\Hom (L_+[1], L_-)$ we put the differential
\[ f \mapsto [\eth, f] = (\partial+u\delta) f - (-1)^{|f|} f \big(-
(\partial+u\delta)\big) = (\partial+u\delta) f + (-1)^{|f|} f
(\partial+u\delta).\]

A pairing on a chain complex with circle action is a symmetric
bilinear chain map of even degree
\[ \pairing : C \otimes C \ra \bbK \]
such that the circle operator $\delta$ is self-adjoint, i.e., we have
\[ \langle \delta x, y\rangle = (-1)^{|x|}\langle x, \delta y\rangle
  \mbox{ for all } x, y\in C. \]
A pairing like the one described above induces a so-called {\em higher
  residue pairing} on the associated Tate complex, with values in
$\bbK\laurent{u}$, defined for $x = \sum x_k\cdot u^k$ and
$y = \sum y_l\cdot u^l$ by
\[  \langle x,\, y \rangle_{\sf hres} = \sum_{k,l} (-1)^l \langle
  x_k,\, y_l\rangle \cdot u^{k+l}. \]
Its residue at $u=\infty$\footnote{The choice of taking residue at
  $u=\infty$ instead of $u=0$ (corresponding to $k+l=-1$) is motivated
  by our desire for homogeneity in the construction of the Weyl
  algebra later in~(\ref{subsec:fock}). Although there might
  be a more conceptual reason for this related to the Fourier-Laplace
  transform of Hodge structures~\cite{Sab}. } is called the residue
pairing, valued in $\bbK$ and defined by
\[ \langle x\cdot u^k, y\cdot u^l \rangle_{\sf res} =
  \begin{cases} (-1)^l \langle x, y\rangle \;\;\; &\mbox {if} \;\;
    k+l=1\\ 0 \;\;\; & \mbox{otherwise.}
  \end{cases}
\]
Due to the sign $(-1)^l$ the residue pairing $\pairing_\res$ is
anti-symmetric. It is an even pairing since the original pairing is even.

\subsection{Weyl algebra and Fock spaces}
\label{subsec:fock}

We now review the relationship, discovered by Costello~\cite{Cos2},
between Batalin-Vilkovisky (BV) algebras and Fock spaces for dg-vector
spaces with circle action and pairing.  This builds on earlier work of
Givental~\cite{Giv}.

Let $(C, \del, \delta, \pairing)$ be a complex with circle action and
pairing, as describe above.  Associated with the symplectic vector
space $(C^\Tate, \langle-,-\rangle_\res)$ is the Weyl algebra
$\cW\big( C^\Tate\big)$ defined as
\[   \cW\big(C^\Tate\big)= T\big( C^\Tate\big)\series{\hbar}/\big (
  \alpha\otimes\beta-(-1)^{|\alpha||\beta|}\beta\otimes\alpha =
  \hbar\langle \alpha,\beta\rangle_{\sf res}\big). \]
Here, again, the formal variable $\hbar$ is of even degree. Note that
the relation is homogeneous with respect to the $\Z/2\Z$-grading.

The positive subspace $u\cdot C_+$ is a subcomplex of $C^\Tate$.
Hence the left ideal $\big( u \cdot C_+ \big)_L$ generated by this
subspace is a dg-ideal in $\cW(C^\Tate)$. The quotient
\[ \mbox{$\cF$} = \cW\big(C^\Tate\big)/ \big( u \cdot C_+ \big)_L \]
is known as the Fock space of $\cW(C^\Tate)$; it is a left dg-module
of the Weyl algebra.

The linear subspace $C_-$ of $C^\Tate$, on the other hand, is {\em
not} a subcomplex: the differential $\del+u\delta$ of an element
$u^0x$ is by definition equal to $u^0\del(x)$ when computed in $C_-$,
but it equals $u^0\del(x) + u^1\delta(x)$ in $C^\Tate$.

Nevertheless, disregarding differentials, $C_-$ is still isotropic as
a graded vector subspace of the symplectic space $(C^\Tate,
\pairing_\res)$.  This yields an embedding of algebras (without
differentials!)
\[ \Sym( C_-) \series{\hbar} \hookrightarrow \cW\big(C^\Tate\big). \]
Post-composing with the canonical projection to the Fock space yields
an isomorphism of graded vector spaces
\[ \mbox{$\Sym$}(C_-) \series{\hbar}\stackrel{\sim}{\lra} \cF. \]

\subsection{Main example}
\label{subsec:pre-example}
The main example of a chain complex with circle action and pairing is
provided by the shifted reduced Hochschild chain complex of a cyclic
unital $A_\infty$-algebra $A$,
\[ L = C_*(A)[d].\]
(Any $d\in\Z$ will work, though in practice if $A$ is a cyclic
$A_\infty$-algebra we take $d$ to be its Calabi-Yau dimension.)  Its
homology is the shifted Hochschild homology of $A$,
\[ H = H_*(L) = HH_*(A)[d]. \]
Using the cyclic pairing and condition $(\dagger)$
in~(\ref{cond:dagger}) $H$ is isomorphic to the dual of the Hochschild
cohomology of $A$, also known as $HH^*(A, A^*)$.  The degree of an
element $x_0[x_1|\ldots | x_n] \in L$ is $\sum_{k=1}^n |x_k|' +(d-1)$
with the extra $-1$ being due to the fact that $x_0$ is not shifted in
defining $L$.  The circle action is given by the Connes operator $B$,
and the symmetric pairing on $L$ is the chain-level Mukai pairing
\[ \pairing_\Muk : L\otimes L \ra \bbK.\]
(The chain-level Mukai pairing is given by the action on the TCFT of the
unique homology class in $H_0(M_{0,2,0})$.  An explicit
formula for it can be found in~\cite[Definition 5.19]{She}.)  As described in
the previous section, if we denote by $\cF_A$ the Fock space
associated with the quadruple $(L,b,B,\langle-,-\rangle_\Muk)$ then
we have a natural isomorphism of graded vector spaces
\[ \mbox{$\Sym$}(L_-) \series{\hbar}\stackrel{\sim}{\lra} \cF_A. \]
We shall use this isomorphism to define a DGLA structure on
$\Sym(L_-)\series{\hbar}$ in the next section.

\section{Two DGLAs associated to a representation of annuli} 
\label{sec:dgla}

Throughout this section we continue to work with a cyclic
$A_\infty$-algebra $A$ satisfying condition $(\dagger)$ in
in~\eqref{cond:dagger}. We also work with the notations $L$, $L_-$,
$L_+$ and $L^\Tate$ as introduced in~(\ref{subsec:pre-example}).  We
will associate to this data the two DGLAs $\h$ and $\hh$ described
in~(\ref{subsec:alggen}) and construct the quasi-isomorphism
$\biota$.

\subsection{\texorpdfstring{A first DGLA associated with $A$}{A first
    DGLA associated with A}}
\label{subsec:dgla1}
Costello~\cite{Cos2} observed that under the isomorphism
\[ \mbox{$\Sym$}(C_-) \series{\hbar}\stackrel{\sim}{\lra} \cF. \]
in~(\ref{subsec:fock}) the differential on the Fock space pulls back
to a differential on the symmetric algebra $\Sym L_- \series{\hbar}$
of a special form.  It can be written as $b+uB+\hbar\Delta$ for a BV
differential $\Delta$, which means that $\Delta$ is a degree one,
square zero, second order differential operator of the symmetric
algebra.  As such it is uniquely determined by its action on
$\Sym^{\leq 2} L_-$; moreover, it vanishes on $\Sym^{\leq 1} L_-$, and
on $\Sym^2 L_-$ it is explicitly given by the formula
\[  \Delta(x\cdot y) =\Omega(x,y)=\langle Bx_0, y_0\rangle, \]
for elements $x,y\in L_-$ of the form
\begin{align*}
  x & = x_0 + x_{-1}u^{-1}+x_{-2}u^{-2} + \cdots, \\
  y & = y_0 + y_{-1}u^{-1}+y_{-2}u^{-2} + \cdots.
\end{align*}
Getzler~\cite{Get} proved that a BV differential such as $\Delta$
induces a DGLA structure on the shifted vector space
$(\Sym L_-)\series{\hbar}[1]$.  More precisely, we define a DGLA $\h_A$
whose graded vector space is 
\[ \h_A = \Sym(L_-)\series{\hbar,\lambda}[1]\]
(with another even formal variable $\lambda$ added) and with
differential $b+uB+\hbar\Delta$, extended $\lambda$-linearly.  The Lie
bracket, which by our convention~(\ref{subsec:signs}) is a symmetric
map
\[ \{-,-\} : \h_A[1] \otimes \h_A[1] \ra \h_A[1], \]
is defined to be the failure of $\Delta$ to be a derivation:
\[ \{x,y\} = \Delta(x\cdot y) -(\Delta x \cdot y)-(-1)^{|x|}(x\cdot
  \Delta y) .\]
Since $\h_A[1]= \Sym(L_-)\series{\hbar,\lambda}[2]$,  the
degree $|x|$ here is computed in $\Sym(L_-)$.

\subsection{\texorpdfstring{A second DGLA associated with $A$}{A
    second DGLA associated with A}}
\label{subsec:dgla-2}
There is another DGLA associated with the $A_\infty$-algebra $A$ which
we now construct.  As a $\Z/2\Z$-graded vector space it is given by
\[ \hh_A = \bigoplus_{k\geq 1,l} \Hom^\cont\big (\Sym^k(L_+[1]),
  \Sym^l( L_-)\big)\series{\hbar,\lambda}.\]
Here the superscript ``$\cont$'' denotes continuous homomorphisms in
the $u$-adic topology. To simplify the notation we will write
\[ L_{k,l} = \Hom^{\sf c} \left (\Sym^k(L_+[1]), \Sym^l(L_-)
  \right). \]
The degree of an element $\Phi\in L_{k,l}$ will be denoted by
$|\Phi|$. 

As in~(\ref{subsec:circle}), in order to avoid confusion about signs
we shall use the notation $\eth$ for the differential induced from
$b+uB$. For example, for an element $\Phi \in L_{k,l}$, the differential
$\eth$ acts as
\[ \eth \Phi = (b+uB) \Phi + (-1)^{|\Phi|} \Phi (b+uB).\]
Note that this is indeed the commutator action, due to the shift $[1]$
of $L_+$.

To define $\hh_A$ we need to define its differential and bracket.
We begin with the former, which will be of the form
$\eth+\hbar\Delta+\iota$ for a BV operator $\Delta$ and a new
``input-output'' operator $\iota$ which we now define.

First, the BV operator $\Delta$ in the previous subsection can be
extended to $\Delta: L_{k,l} \ra L_{k,l-2}$ by simply setting
\[ \Delta (\Phi) = \Delta \circ \Phi, \;\;\; \mbox{for} \;\; \Phi \in
  L_{k,l}.\]
One easily verifies that $(\eth+\hbar \Delta)^2=0$ on $\hh_A$.

For an element $\beta\in L_+[1]$ we define a contraction operator
$C_\beta: \Sym^l L_- \ra \Sym^{l-1} L_-$ by
\[ C_\beta \big( \alpha_1\cdots\alpha_l \big) = \sum_{i=1}^l
(-1)^{|\alpha_i|(|\alpha_1|+\cdots+|\alpha_{i-1}|)} \langle
u\beta,\alpha_i\rangle_\res\cdot
\alpha_1\cdots\widehat{\alpha_i}\cdots\alpha_l\]
\begin{Lemma}
  The following identities hold:
  \begin{itemize}
  \item[(a)] $C_\beta$ is a derivation of $\Sym L_-$ of degree
    $|\beta|'$, i.e. we have
    \[C_\beta (\alpha\cdot \gamma) = C_\beta (\alpha) \cdot \gamma
      +(-1)^{|\alpha||\beta|'} \alpha\cdot C_\beta(\gamma)\]
  \item[(b)] $C_\beta\circ \Delta -(-1)^{|\beta|'} \Delta \circ C_\beta
    =0$.
  \item[(c)] $C_{\beta_1} C_{\beta_2} -(-1)^{|\beta_1|'|\beta_2|'}
    C_{\beta_2}C_{\beta_1} =0$.
  \item[(d)] $[b+uB,C_\beta] = -C_{\eth\beta}$ with $\eth\beta =
    -(b+uB)\beta$.
  \end{itemize}
\end{Lemma}

\begin{proof} Direct calculation.
\end{proof}

We now use these contraction operators to construct one more
Koszul-type differential $\iota$ on $\hh_A$.  On individual components
$L_{k,l}$ this is the odd map $\iota: L_{k,l} \ra L_{k+1,l-1}$ defined by
\[  \iota (\Phi) (\beta_1\cdots\beta_{k+1}) = \sum_{j=1}^{k+1}
  (-1)^{|\beta_j|(|\beta_1|+\cdots+|\beta_{j-1}|+|\Phi|)} C_{\beta_j} 
  \Phi(\beta_1\cdots\widehat{\beta_j}\cdots \beta_{k+1}) \]
Graphically, this definition is presented as the following sign diagram
\[\includegraphics[scale=.5]{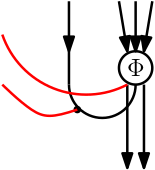}\]
(See~\cite{She} for details on sign diagrams.)

\begin{Lemma}
  \label{lem:square-zero}
  We have $(\eth+\iota + \hbar \Delta)^2=0$.
\end{Lemma}

\begin{proof} With the explicit formula for $\iota$ one checks
directly that $[\eth,\iota]=[\iota,\iota]=[\iota,\Delta]=0$.  The proof
is again by direct calculation and will be omitted.
\end{proof}

Denote by $\Theta:L_-\ra L_+[1]$ the circle action map defined by
\[  \Theta(\alpha) = B(\alpha_0) \]
for $\alpha=\alpha_0+\alpha_{-1}u^{-1}+\cdots \in L_-$.  Note that
$\Theta$ is a map of chain map of degree zero because of the shift $[1]$:
\[ \Theta \eth (\alpha) = Bb\alpha_0 = - bB\alpha_0 = \eth \Theta
(\alpha).\]

For every natural number $r\geq 1$ we use $\Theta$ to define a
composition operation
\[  \circ_r: L_{k_2,l_2} \otimes L_{k_1,l_1} \ra L_{k_1+k_2-r,
    l_1+l_2-r}, \mbox{ for }l_1,k_2\geq r  \]
as follows. For a symmetric tensor $\gamma_1\cdots \gamma_n$ and a
subset $A=\{a_1,\ldots, a_s\}\subset \{1,\ldots, n\}$ with
$a_1<\cdots<a_s$ we write
\[ \gamma_A = \gamma_{a_1}\cdots\gamma_{a_s}.\]
With this notation we define
\[ (\Psi \circ_r \Phi) ( \beta_1\cdots \beta_{k_1+k_2-r})\! =\!
\sum_{P,Q}\sum_{I,J} \epsilon_{I,J}\epsilon_{P,Q}
(-1)^{|\Psi||\Phi(\beta_I)_P|} \Phi(\beta_I)_P\cdot \Psi\big(
\Theta^{\otimes r} ( \Phi(\beta_I)_Q)\otimes \beta_J\big).\]
The summation is over all shuffles $I\coprod J = \{ 1,\ldots,k_1+k_2-r\}$
of type $(k_1,k_2-r)$ and $P\coprod Q = \{1,\ldots, l_1\}$ of type
$(r,l_1-r)$. The signs $\epsilon_{I,J}$ and $\epsilon_{P,Q}$ are the
Koszul signs associated with the permutations  given by the shuffles.
The following sign diagram illustrate this definition in the case with
$k_1=3$, $l_1=5$, $k_2=4$, $l_2=3$ and $r=2$:
\[\includegraphics[scale=.6]{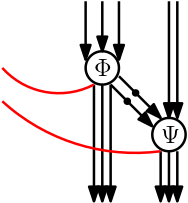}\]
The black dots in the middle arrows indicate that we apply the map
$\Theta$ at each strand.

Note that the composition defined above is {\em not} an associative
operation. Nevertheless, we shall make use of its commutators to
define a Lie bracket structure on $\hh_A$. Indeed, we define a
symmetric map $\{-,-\}_\hbar: \hh_A[1] \otimes \hh_A[1] \ra \hh_A[1]$
by putting
\[ \{ \Psi, \Phi \}_\hbar = (-1)^{|\Psi|} \sum_{r\geq 1}
\big(\Psi\circ_r \Phi - (-1)^{|\Psi||\Phi|} \Phi \circ_r \Psi\big)
\hbar^{r-1}.\]
The degrees $|\Phi|$ and $|\Psi|$ are computed in $\hh_A$. The extra
sign $(-1)^{|\Psi|}$ is due to our convention to use shifted signs in
DGLAs.

\begin{Theorem}
  \label{thm:h-hat}
  The triple $\big( \hh_A,\eth+\iota+\hbar \Delta, \{-,-\}_\hbar
  \big)$ forms a DGLA. 
\end{Theorem}

\begin{proof}
One can verify directly that $\eth$ satisfies the following Leibniz
property against the Lie bracket $\{-,-\}_\hbar$:
\[\eth \{ \Psi, \Phi \}_\hbar + \{\eth \Psi, \Phi \}_\hbar +
  (-1)^{|\Psi|'} \{ \Psi, \eth\Phi \}_\hbar=0.\]

Note that neither $\iota$ nor $\Delta$ are derivations
of the Lie bracket.  Nevertheless, their sum
$\iota+\hbar\Delta$ {\em is} a derivation, that is 
\[(\iota+\hbar \Delta) \{ \Psi, \Phi \}_\hbar + \{(\iota+\hbar \Delta)
\Psi, \Phi \}_\hbar+ (-1)^{|\Psi|'} \{ \Psi, (\iota+\hbar \Delta)\Phi
\}_\hbar=0.\]

To prove this equation, in the case when the power of $\hbar$ is zero,
we want to prove that
\[ \iota\{ \Psi, \Phi \}_1 + \{ \iota \Psi, \Phi\}_1 + (-1)^{|\Psi|'}
  \{ \Psi, \iota\Phi \}_1=0.\]

For this consider the following sign diagram:
\[\includegraphics[scale=.6]{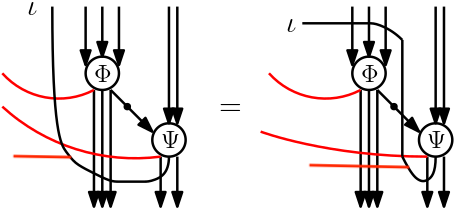}\]

The two diagrams are equal by isotopy invariance. The left hand
side appears in $\iota\{ \Psi, \Phi \}_1$ with the sign $(-1)^{|\Psi|}$,
while the right hand side appears in $ \{ \iota \Psi, \Phi\}_1$ with the
sign $(-1)^{|\Psi|'}$.  This shows that the two terms cancel out.

For general $r\geq 2$ we want to prove that
 \begin{align*} \iota\{ \Psi, \Phi \}_r + & \{ \iota \Psi, \Phi\}_r +
(-1)^{|\Psi|'} \{ \Psi, \iota\Phi \}_r =\\ & -
\Delta\{\Psi,\Phi\}_{r-1}-\{\Delta \Psi,\Phi\}_{r-1}-(-1)^{|\Psi|'}
\{\Psi,\Delta\Phi\}_{r-1}.
 \end{align*}
 
This is illustrated in the following sign diagram:
\[\includegraphics[scale=.6]{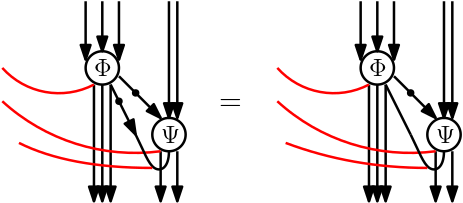}\]
The left hand side appears in $ \{ \iota \Psi, \Phi\}_r $ with the
sign $(-1)^{|\Psi|'}$, while the right hand side appears in $-
\Delta\{\Psi,\Phi\}_{r-1}$ with the same sign. Other cancellations are
similarly proved. 

It remains to prove the Jacobi identity, that for three elements
$\Psi, \Phi, \Upsilon \in \hh_A$ we have
\[ \{\{\Psi,\Phi\}_\hbar,\Upsilon\}_\hbar +(-1)^{|\Upsilon|'|\Phi|'}
\{ \{ \Psi,\Upsilon\}_\hbar, \Phi\}_\hbar +
(-1)^{|\Psi|'|\Phi|'+|\Psi|'|\Upsilon|'} \{ \{
\Phi,\Upsilon\}_\hbar,\Psi\}_\hbar=0 \]
There are in total 12 types of compositions in the above sum which
cancel out in pairs. For example, consider the following sign diagram:
\[\includegraphics[scale=.5]{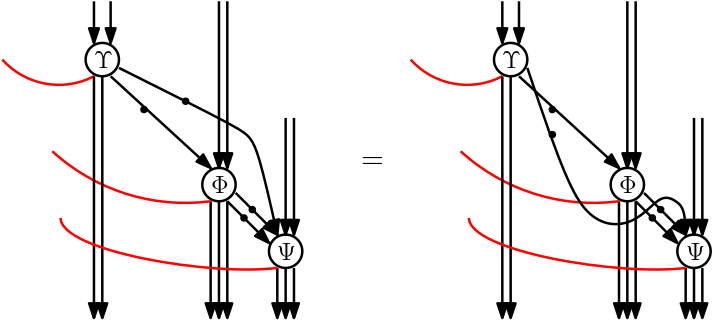}\]
The left hand side appears in $\{\{\Psi,\Phi\}_\hbar,\Upsilon\}_\hbar$
with the sign $(-1)^{|\Psi|}(-1)^{|\Psi|+|\Phi|}=(-1)^{|\Phi|}$, while
the right hand side appears in
$(-1)^{|\Psi|'|\Phi|'+|\Psi|'|\Upsilon|'} \{ \{
\Phi,\Upsilon\}_\hbar,\Psi\}_\hbar$
with the sign 
\[ (-1)^{|\Psi|'|\Phi|'+|\Psi|'|\Upsilon|'} (-1)^{|\Upsilon|} (-1)
  (-1)^{|\Psi||\Phi|+|\Psi||\Upsilon|}=-(-1)^{|\Phi|}, \]
which shows the two terms cancel out. The cancellation of the other
pairs is similar.
\end{proof}

\subsection{\texorpdfstring{The Koszul resolution map $\biota$}{The
    Koszul resolution map iota}}
\label{subsec:Koszul}
We now relate the two DGLAs $\h_A$ and $\hh_A$ associated with $A$ by
constructing a map $\biota$ between them.  To be precise, we will
construct a quasi-isomorphism $\biota:\h_A^+\ra \hh_A$ from a certain
quotient $\h_A^+$ of $\h_A$.  The distinction between $\h_A$ and
$\h_A^+$ turns out to be irrelevant, see the discussion after the statement
of Theorem~\ref{thm:K}.

Denote by
\[ \h_A^+= \left ( \Sym^{\geq 1} L_-\right) \series{\hbar,
    \lambda}[1] \]
the strictly positive symmetric powers of $L_-$. This is not a
sub-DGLA of $\h_A$, since for instance $\Delta (\Sym^2 L_-)$ would not
be contained in $\h_A^+$.  However, the scalar part
$\bbK\series{\hbar,\lambda}[1]$ is in the center of $\h_A$, implying
that $\h_A^+$ is a {\em quotient} of $\h_A$, i.e., there is a short
exact sequence
\[  0 \ra \bbK\series{\hbar,\lambda}[1] \ra \h_A \ra \h_A^+ \ra 0. \]
Define the map $\biota: \h_A^+ \ra \hh_A $ by setting
\[ \biota(\alpha) = (-1)^{|\alpha|} \iota (\alpha).\]
The sign in the definition of $\biota$ is needed in order to turn the
map $\iota$ (which satisfies the anti-commutative property of
Lemma~\ref{lem:square-zero}) into a map of chain complexes.
 
\begin{Lemma}
  \label{lem:biota}
  The morphism $\biota$ is a quasi-isomorphism of DGLAs.
\end{Lemma}

\begin{proof}
We shall prove that the mapping cone of $\biota$ is acyclic. For
this it suffices to prove that for each $n\geq 1$ the following
sequence of dg vector spaces is exact:
  \[ 0\ra \Sym^n L_- \ra \Hom^\cont( L_+[1], \Sym^{n-1} L_-)\ra \cdots
\ra \Hom^\cont( \Sym^n (L_+[1]), \mathbb{C}) \ra 0. \] The horizontal
arrows are the map $\iota$.  We argue that the total complex of this
sequence has zero homology, by considering the spectral sequence
associated with the $u$-filtration.  Since $A$ satisfies the
degeneration property, the $E^1$ page is given by
  \[ 0\ra \Sym^n H_- \ra \Hom^\cont( H_+, \Sym^{n-1} H_-)\ra \cdots
\ra \Hom^\cont( \wedge^n H_+, \mathbb{C}) \ra 0 \] where
$H=HH_*(A)[d]$ is the shifted Hochschild homology of $A$. Because $A$
is assumed to be smooth $H$ is finite dimensional.  From this
finiteness condition we can deduce the exactness of the $E^1$-page as
follows.  Observe that there is an isomorphism for each $k+l=n$
  \[ I_{k,l}: \Sym^k (H_-[1]) \otimes \Sym^l H_- \ra \Hom^\cont(
\Sym^k ( H_+[1]), \Sym^{l} H_-)\] defined using the Mukai pairing on
$H$.  In the case of $k=1$, $l=0$ this map is explicitly given by
  \[ I_{0,0} (x) (y) = \res_{u=\infty} \langle x, u\cdot y\rangle.\]
For general $k$ and $l$ it is defined similarly. The fact that
$I_{k,l}$ is an isomorphism follows from a result of
Shklyarov~\cite{Shk} (following~\cite{CalWil}) proving that the
categorical Mukai pairing is non-degenerate when $A$ is smooth and
proper.  Putting the isomorphisms $I_{k,l}$ together we conclude that
the $E^1$ page is isomorphic to the complex
  \[ 0\ra \Sym^n H_- \ra H_-\otimes \Sym^{n-1} H_- \ra \cdots \ra
\Sym^n (H_-[1]) \ra 0\] endowed with the usual Koszul differential.
Exactness follows.
\end{proof}

\section{The Sen-Zwiebach DGLA and string vertices}
\label{sec:vertices}

There are two other DGLAs $\g$ and $\hg$, constructed in the
topological world of chains on moduli spaces of smooth Riemann
surfaces, whose definitions are completely analogous to those of $\h$
and $\hh$ in the previous section.  The DGLAs $\g$ and $\hg$ should be
thought of as ``master'' versions of the $\h$ and $\hh$ associated
with an $A_\infty$-algebra $A$: for any such $A$ there are canonical
maps $\g\ra \h$ and $\hg \ra \hh$.

As in the algebraic case we will describe two DGLAs $\g$ and $\hg$.
In the first one we assume there are no inputs; all outputs are
treated the same.  In the the second one we distinguish inputs from
outputs, and we assume that there is at least one input at all
times. The first construction is due to Sen-Zwiebach~\cite{SenZwi}.
The second construction is in some sense a ``Koszul'' resolution of
the first one, and it first appeared in the unpublished work of
Costello and the authors~\cite{CCT}.

Throughout this section $C_*(-)$ will denote the functor of
normalized singular chains with coefficients in the ground field
$\bbK$ from topological spaces to dg-vector spaces.

\subsection{Moduli spaces of curves and sewing operations}
For $g\geq 0$ and $k,l\geq 0$ satisfying $2g-2+k+l>0$ we consider the
coarse moduli space $M_{g,k,l}^\fr$ of Riemann surfaces of genus $g$,
with $k+l$ marked and framed points, $k$ of which are designated as
inputs, and the rest as outputs.  Explicitly, a point in
$M_{g,k,l}^\fr$ is given by the data $(\Sigma, p_1,\ldots,p_k,
q_1,\ldots,q_l,\phi_1,\ldots,\phi_k,\psi_1,\ldots,\psi_l)$ where
\begin{itemize}
\item $\Sigma$ is a smooth Riemann surface of genus $g$,
\item the $p$'s and $q$'s are distinct marked points on $\Sigma$,
\item the $\phi$'s and $\psi$'s are framings around the marked
points, given by biholomorphic maps
\begin{align*}
  \phi_i: &\, \mathbb{D}^2 \ra U_{\epsilon_i}(p_i),\;\;1\leq i\leq k \\
  \psi_j: &\, \mathbb{D}^2 \ra U_{\epsilon_j}(q_j),\;\; 1\leq j\leq l
\end{align*}
where $\mathbb{D}^2$ stands for the unit disk in $\C$, and
$U_\epsilon(x)$ is a radius $\epsilon$ disk of the marked point $x$ on
$\Sigma$, under, say, the unique hyperbolic metric on
$\Sigma\setminus \{p_1,\ldots,p_k,q_1,\ldots,q_l\}$.  We require that
the biholomorphic maps $\phi$ and $\psi$ extend to an open
neighborhood of $\mathbb{D}^2$.  We also require that the closures of
all the framed disks be disjoint, i.e., for any two distinct points
$x, y\in \{ p_1,\ldots,p_k,q_1,\ldots,q_l\}$ we have
\[\overline{U_{\epsilon_i}(x)} \cap
\overline{U_{\epsilon_j}(y)}=\emptyset.\]
\end{itemize}
Note that the fibration $M_{g,k,l}^\fr\ra M_{g,k,l}$ which forgets the
framing is a homotopy $(S^1)^{k+l}$-bundle. Indeed, for each index
$1\leq i\leq k$ or $1\leq j\leq l$, there is a corresponding circle
action on the framing $\phi_i$ or $\psi_j$ by precomposing with the
rotation map $e^{i\theta}:\mathbb{D}^2\ra \mathbb{D}^2$.

We shall make use of different types of sewing maps of Riemann
surfaces:
\begin{itemize}
\item {\sl (Sewing inputs and outputs.)}
  Fix a positive integer $1\leq r \leq l$. There is a sewing map
  \[ G_r: M_{g',k',l'}^\fr \times M_{g'',k'',l''}^\fr \ra
    M_{g,k,l}^\fr \]
  with $g=g'+g''+r-1$, $k=k'+k''-r$, $l=l'+l''-r$.  This map is
  defined by removing the first $r$ input open disks
  $\phi'_1(\mathbb{D}^2),\ldots,\phi'_r(\mathbb{D}^2)$ from the first
  surface, removing the first $r$ output open disks
  $\psi''_1(\mathbb{D}^2),\ldots,\psi''_r(\mathbb{D}^2)$ from the
  second surface, and then gluing along the boundary circles using the
  coordinates from the framings in such a way that $\phi'_i(S^1)$ is
  identified with $\psi''_i(S^1)$. Applying the functor $C_*(-)$ and
  pre-composing with the Alexander-Whitney map yields a composition
  map which we denote by
  \[ G_r: C_*(M_{g',k',l'}^\fr) \otimes C_*(M_{g'',k'',l''}^\fr) \ra
    C_*(M_{g,k,l}^\fr) \]
  as well.

\item {\sl (Sewing two outputs.)}
  In a similar way, for two indices $1\leq j' \leq l'$ and
  $1\leq j'' \leq l''$ we have a sewing map between two outputs of
  two Riemann surfaces:
  \[ G_{j'j''} : C_*(M_{g',k',l'}^\fr) \otimes C_*(M_{g'',k'',l''}^\fr)
    \ra C_*(M_{g'+g'',k'+k'',l'+l''-2}^\fr). \]
  
\item {\sl (Self-sewing two outputs.)}
  We can also sew two outputs of a single Riemann surface. Indeed, for
  a pair of indices $1\leq j_1<j_2 \leq l$ we have a self-sewing
  map:
  \[ G_{j_1j_2}: C_*(M_{g,k,l}^\fr) \ra C_*(M_{g+1,k,l-2}^\fr).\]
  We have slightly abused the notation for $G$, but no confusion can
  arise in this way.
\end{itemize}
It is worth pointing out that we never sew two inputs together, and we
never self-sew an input to an output. This asymmetry between inputs
and outputs causes considerable difficulty in the theory.  However, it
is at essential in order to avoid later computing traces on infinite
dimensional vector spaces!

\subsection{The Sen-Zwiebach DGLA}
Consider the moduli space $M_{g,0,n}^\fr$ with only outgoing marked
points.  Let $C_*(M_{g,0,n}^\fr)$ be its associated normalized
singular chain complex.  Taking the homotopy quotient by the $n$
commuting circle actions yields the equivariant chain complex
\[ C_*(M_{g,0,n}^\fr)_{(S^1)^n} =
\big(C_*(M_{g,0,n}^\fr)[u_1^{-1},\ldots,u_n^{-1}], \eth=\partial+
\sum_{j=1}^n u_jB_j\big)\]
The symmetric group $S_n$ acts on $C_*(M_{g,0,n}^\fr)_{(S^1)^n} $
diagonally by the induced action on $C_*(M_{g,0,n}^\fr)$ and by
permutation of the indices of the $u$-variables.  Further taking
$S_n$ coinvariants yields
\[ C_*(M_{g,0,n}^\fr)_\hS = C_*(M_{g,0,n}^\fr)_{(S^1)^n \rtimes
    S_n}.\]
We will now describe a DGLA structure on the super
vector space
\[ \g= \big( \bigoplus_{g,n} C_*(M_{g,0,n}^\fr)_\hS \big)
  \series{\hbar,\lambda} [1].\]
We shall refer to this DGLA as the Sen-Zwiebach DGLA,
see~\cite{SenZwi,Cos2}.  Its construction is parallel to that of
the DGLA $\h$ associated to a cyclic $A_\infty$-algebra
in~(\ref{subsec:dgla1}).  Indeed, we define an operator
$\g[1] \ra \g[1]$ of degree one by setting
\[ \Delta (\alpha) = \sum_{1\leq i< j\leq n} G_{ij}\big( \pi_{ij} (
  B_i\alpha)\big),\]
where $G_{ij}$ is induced from the self-sewing map
$M^\fr_{g,0,n} \ra M_{g+1,0,n-2}^\fr$, and the map $\pi_{ij}$ takes an
equivariant chain $\gamma$ to
\[ \pi_{ij} (\gamma) = \gamma \mid_{u^{-1}_i=u^{-1}_j=0}.\]
That is, $\pi_{ij}$ only keeps the constant term of $\gamma$ in the
expansion in the variables $u_i^{-1}$ and $u_j^{-1}$.  In other words,
the definition of $\Delta$ is such that the sewing operation is
only performed when the circle parameters $u_i$ and $u_j$ in $\alpha$
are both power zero.
 
The Lie bracket map, a symmetric degree one map $\g[1] \otimes \g[1]
\ra \g[1]$, is similarly defined. Explicitly, for $\alpha \in
C_*(M_{g',0,n'}^\fr)_\hS$ and $\beta \in C_*(M_{g'',0,n''}^\fr)_\hS$,
we set
\[ \{\alpha,\beta\} = \sum_{\substack{1\leq i\leq n'\\1\leq j\leq
      n''}} G_{ij}\big(\pi_i (B_i\alpha), \pi_j(\beta)\big).\]

\begin{Theorem}
  \label{thm:sz-dgla}
  The triple $\big( \g, \partial+uB+\hbar \Delta, \{-,-\}\big)$ forms a DGLA.
\end{Theorem}

\begin{proof}
  The proof is analogous to the proof of the fact that a dgBV algebra
  induces DGLA structure~\cite{Get}. The commutative algebra structure in
  this case is induced from taking disjoint unions of Riemann surfaces.
\end{proof}

Note further that the DGLA $\g$ admits a homological $\Z$-grading by
setting $\deg \hbar =\deg \lambda =-2$.

\begin{Definition}
\label{defi:vertices}
A degree zero element $\cV\in \g[1]$ of the form
\[ \cV=\sum_{g,n}\cV_{g,n}\hbar^g\lambda^{2g-2+n}, \;\; \cV_{g,n}\in
  C_{6g-6+2n}(M^\fr_{g,0,n})_\hS\]
is called a {\em string vertex} if it satisfies the following properties:
\begin{enumerate}
\item $\cV$ is a Maurer-Cartan element in $\g$; in components
  $\{\cV_{g,n}\}$  this means that we have
  \[ (\partial+uB) \cV_{g,n}+\Delta \cV_{g-1,n+2} +
    \frac{1}{2}\sum_{\substack{g_1+g_2=g\\n_1+n_2=n+2}}
    \{\cV_{g_1,n_1},\cV_{g_2,n_2}\} = 0. \] 
\item $\cV_{0,3}=\frac{1}{6}\cdot {{\sf pt}}$ with ${{\sf pt}} \in
  C_0( M^\fr_{0,0,3})_\hS$ representing a point class.
\end{enumerate}
\end{Definition} 

\noindent
A fundamental result about string vertices is the following theorem
from~\cite{Cos2}.

\begin{Theorem}
\label{thm:costello} A string vertex exists and is unique up to
homotopy (i.e., gauge equivalence between Maurer-Cartan elements).
\end{Theorem} 

\begin{proof}
  The proof is a standard argument in deformation theory.  The
  obstruction to existence lies in $H_{6g-7+2n}(M^\fr_{g,0,n})_\hS$
  while the deformation space is given by
  $H_{6g-6+2n}(M^\fr_{g,0,n})_\hS$. Both homology groups are known to
  vanish.  We refer to~\cite[Section 9]{Cos2} for a detailed proof.
\end{proof}

\subsection{\texorpdfstring{The Koszul resolution of $\g$}{The Koszul
    resolution of g}}
\label{subsec:hg}

We now describe the second DGLA $\hg$ constructed from moduli spaces
of Riemann surfaces.  Its construction is parallel to that of $\hh$
in~(\ref{subsec:dgla-2}). 

For an integer $k\geq 1$ denote by $\underline{\sgn_k}[-k]$ the rank
one local system over $M_{g,k,l}^\fr$ whose fiber over a Riemann
surface $(\Sigma,
p_1,\ldots,p_k,q_1,\ldots,q_l,\phi_1,\ldots,\phi_k,\psi_1,\ldots,\psi_l)$
is the sign representation of the symmetric group $S_k$ on the set $\{
p_1,\ldots, p_k\}$, shifted by $[-k]$. Denote by
$p_1\wedge\ldots\wedge p_k$ its natural basis vector. For simplicity,
we often write this local system as $\usgn$ when the
integer $k$ is clear from the context.

We shall construct a DGLA structure on the graded vector space
\[ \hg = \bigoplus_{\substack{g,k,l\\ k\geq 1}} C_*\big(M_{g,k,l}^\fr,
  \usgn\big)_\hS\series{\hbar,\lambda}[2],\]
where the homotopy quotient is by $(S^1)^{k+l} \ltimes S_k\times S_l$
with the $S_k$ action twisted by $\usgn$.  The homological
degree of an element
$(\alpha,p_1\wedge\ldots\wedge p_k)\in C_*\big(M_{g,k,l}^\fr,
\usgn\big)_\hS$ is explicitly denoted by
\[ |\alpha|= \deg(\alpha) -k,\]
where $\deg(\alpha)$ is the usual chain degree in the complex
$C_*\big(M_{g,k,l}^\fr)_\hS$, before the shifting by $[k]$ operation.

The complex $C_*\big(M_{g,k,l}^\fr)_\hS$ has the equivariant
differential $\eth= \partial+uB$.  After adding in the local system,
the complex $C_*\big(M_{g,k,l}^\fr,\usgn\big)_\hS$ has
differential
\[ \eth (\alpha,p_1\wedge\ldots\wedge p_k) = (\eth
\alpha,p_1\wedge\ldots\wedge p_k).\]

The self-sewing operator
$\Delta:C _*\big(M_{g,k,l}^\fr,\usgn\big)_\hS \ra
C_*\big(M_{g+1,k,l-2}^\fr,\usgn\big)_\hS$ is defined as
before, by 
\[ \Delta (\alpha,p_1\wedge\ldots\wedge p_k) = \sum_{1\leq i< j\leq
l} \Big( G_{ij}\big( \pi_{ij}(B_i\alpha)\big),p_1\wedge\ldots\wedge
p_k\Big)\]
where $G_{ij}: M^\fr_{g,k,l} \ra M_{g+1,k,l-2}^\fr$ is the self-sewing
map.
    
Define another operator $\iota: C_*\big(M_{g,k,l}^\fr,
\usgn\big)_\hS \ra C_*\big(M_{g,k+1,l-1}^\fr,
\usgn\big)_\hS$ by setting 
\[ \iota \big( \alpha, p_1\wedge\ldots\wedge p_k\big) =
  (-1)^{\deg(\alpha)} \sum_{j=1}^l \big(\iota_j \alpha, q_j \wedge
  p_1\wedge\ldots\wedge p_k \big) \]
where $\iota_j: M^\fr_{g,k,l} \ra M_{g,k+1,l-1}^\fr$ is the map which
simply changes the output vertex $q_j$ to become the first input
vertex and relabels the remaining vertices.

\begin{Lemma}
  We have $(\eth+\iota+\hbar\Delta)^2=0$.
\end{Lemma}

\begin{proof}
One verifies directly from the definitions that we have
\[[\eth,\iota]=[\iota,\iota]=[\iota,\Delta]=0. \]
\end{proof}

We proceed to define the Lie bracket of $\hg$.  It is a symmetric map
\[ \{-,-\}_\hbar: \hg[1] \otimes \hg[1] \ra \hg[1].\]
Recall that for a positive integer $r\geq 1$ we have a map $G_r$ that
sews the first $r$ inputs of a surface in $M^\fr_{g',k',l'}$ with the
first $r$ outputs of a surface in $M^\fr_{g'',k'',l''}$.  It induces a
map
\[  G_r: C_*(M^\fr_{g',k',l'}) \otimes C_*(M^\fr_{g'',k'',l''}) \ra
  C_*(M^\fr_{g,k,l}),\mbox{ for } l''\geq r, \; k'\geq r. \]
More generally, for a pair of multi-indices $J'\subset
K'=\{1,\ldots,k'\}$ and $I''\subset L''=\{1,\ldots,l''\}$, both of size $r$,
there is a sewing map
\[ G_{J'I''} (\alpha,\beta)= G_r ( \sigma_{J',K'\setminus J'} \alpha,
  \sigma_{I'',L''\setminus I''} \beta ),\]
where $\sigma_{J',K'\setminus J'}$ and $ \sigma_{I'',L''\setminus I''}$ are shuffle
permutations of $K'$ and $L''$ respectively.

Define a twisted version
\[ G^B_{r}: C_*(M^\fr_{g',k',l'},\usgn) \otimes
C_*(M^\fr_{g'',k'',l''},\usgn) \ra C_*(M^\fr_{g,k,l},
\usgn),\] by setting
\begin{align*}
  &G^B_{r}\big( (\alpha,p'_1\wedge\cdots\wedge p'_{k'}
    ),(\beta,p''_1\wedge\cdots\wedge p''_{k''}) \big)\\
  =& (-1)^{\deg(\beta)(k'+r)}\big( G_{r}( \alpha,B_{q''_1}\cdots B_{q''_r}
     \beta),p'_{r+1}\wedge\cdots p'_{k'}\wedge p''_{1}\wedge\cdots\wedge
     p''_{k''}\big).
\end{align*}
Here $B_{q_i''}(\beta)$ is the map on chains induced by the circle
operator rotating the framing of the $i$-th output of a Riemann
surface parametrized by $\beta$.  The sign is chosen to match with the
algebraic construction of~(\ref{subsec:dgla-2}), see
Proposition~\ref{prop:universal}. 
 
We can further extend this map to a map $\widetilde{G^B_r}$ of
homotopy quotients by the circle actions, using the formula
\[ \widetilde{G^B_r} (\Psi,\Phi) = G^B_r
(\Psi\mid_{w_1^{-1}=\cdots=w_r^{-1}=0},
\Phi\mid_{u_1^{-1}=\cdots=u_r^{-1}=0}).\]
Here the $w$'s are circle parameters at inputs while the $u$'s are
circle parameters at outputs. In other words we only apply the twisted sewing
map $G_r^B$ when the circle parameters all have power zero at the
points to be sewed

\begin{Lemma}
The twisted sewing map
\[ \widetilde{G^B_r}:  C_*\big(M_{g',k',l'}^\fr,\usgn
  \big)_{(S^1)^{k'+l'}}  \otimes
  C_*\big(M_{g'',k'',l''}^\fr,\usgn \big)_{(S^1)^{k''+l''}}
  \ra C_*\big(M_{g,k,l}^\fr,\usgn\big)_{(S^1)^{k+l}}\]
is a chain map.
\end{Lemma}

\begin{proof}
The twisted sewing map $\widetilde{G^B_r}$ clearly
commutes with the geometric boundary map $\partial$. Since
$\eth=\partial + u B$, it remains to prove its compatibility with the
circle actions. For circle parameters not involved in the sewing, this
is again clear. Let us consider the action by $u B_{q''_1}$. For an
equivariant chain $\Phi$ we have
\[ (u_1 B_{q''_1} \Phi )|_{u_1^{-1}=0} \in \imag (B_{q''_1}).\] This
shows that $ \widetilde{G^B_r} ( \Psi, u B_{q''_1} \Phi ) =0$ since
$B_{q''_1} B_{q''_1}=0$.
\end{proof}

Symmetrizing along the $r$ inputs and outputs yields our desired
composition map
\[ \circ_r : C_*\big(M_{g',k',l'}^\fr,\usgn\big)_\hS
\otimes C_*\big(M_{g'',k'',l''}^\fr,\usgn\big)_\hS \ra
C_*\big(M_{g,k,l}^\fr,\usgn\big)_\hS\] explicitly given by
\[ \Psi\circ_r \Phi =
\sum_{\sigma}\sum_{\tau} \widetilde{G^B_r} (\sigma \Psi,\tau
\Phi). \]
where the summation is over $\sigma\in {\sf
Shuffle}(r,k'-r)$ and $\tau\in {\sf Shuffle}(r,l''-r)$ which act by
permutations of $\{p_1',\ldots,p_{k'}'\}$ and
$\{q''_1,\ldots,q''_{l''}\}$ respectively.

\begin{Theorem}
  \label{thm:sz-koszul}
  Define an odd symmetric map $\{-,-\}_\hbar: \hg[1] \otimes \hg[1]
  \ra \hg[1]$ by 
  \[ \{\Psi,\Phi\}_\hbar= (-1)^{|\Psi|} \sum_{r\geq 1} \big(
    \Psi\circ_r\Phi - (-1)^{|\Psi||\Phi|}\Phi\circ_r \Psi\big)\cdot
    \hbar^{r-1},\]
  with the degrees $|\Psi|$, $|\Phi|$ computed in $\hg$.  The sign
  $(-1)^{|\Psi|}$ is due to our convention to work with shifted signs
  of DGLAs.

  Then the triple
  $\big(\hg, \eth+\iota+\hbar \Delta, \{-,-\}_\hbar\big)$ forms a
  $\Z$-graded DGLA.
\end{Theorem}

\begin{proof}
  This is analogous to the proof of Theorem~\ref{thm:h-hat} and
  will be omitted.
\end{proof}

\begin{Proposition}
  There exists a quasi-isomorphism of DGLAs
  \[ \biota: \g^+ \ra \hg, \]
  where 
  \[ \g^+= \big( \bigoplus_{g,n\geq 1} C_*(M_{g,0,n}^\fr)_\hS \big)
    \series{\hbar,\lambda} [1] \]
  is endowed with the quotient DGLA structure of $\g$ in a matter
  similar to that of $\h$, see~(\ref{subsec:Koszul}).
\end{Proposition}

\begin{proof}
Observe that for fixed $g\geq 0$ and $n>0$ the various
moduli spaces $M^\fr_{g,k,l}$ with $k+l=n$ are all isomorphic. Then
the result follows from the algebraic fact that if $W$ is any
$\Sigma_n$-representation, the associated Koszul complex
\[ 0\ra W_{\Sigma_n} \ra W_{\Sigma_{1,n-1}} \ra
  W_{\Sigma_{2,n-2}}\ra \cdots\ra W_{\Sigma_{n,0}}\ra 0\]
is exact.  Here $W_{\Sigma_{k,l}}$ denotes the space of coinvariants
of $W$ under the action of $\Sigma_k \times \Sigma_l$, where the
first group acts via the sign representation $\sgn_k$ and the second
one acts via the trivial representation.
\end{proof}

\subsection{Combinatorial string vertices}
When $k\geq 1$ and $l\geq 0$ there exists a version of
$C_*(M^\fr_{g,k,l})$ constructed in terms of ribbon graphs with framed
inputs and outputs (also sometimes called fat graphs with black and
white vertices).  This construction is due to
Kontsevich-Soibelman~\cite{KonSoi}, Costello~\cite{Cos1} and
Wahl-Westerland~\cite{WahWes}.  Denote this quasi-isomorphic chain
complex by $C_*^\comb(M_{g,k,l}^\fr)$; it is spanned by isomorphism
classes of framed and oriented ribbon graphs of genus $g$ with $k$
faces and $l$ white vertices.  We also have combinatorial versions
of the operators $\partial$, $B$, $\Delta$, and $\{-,-\}_\hbar$.  Thus,
by the same construction of the previous subsection, we can define a
combinatorial version of $\hg$, denoted by:
\[ \displaystyle{\hg^{\comb}=\left (\bigoplus_{g, k\geq 1, l}
C^{\comb}_*(M_{g,k,l}^\fr,\usgn)_\hS \right )
\series{\hbar,\lambda}}[2].\]
It also forms a DGLA.

\begin{Theorem}
  There exists a roof diagram of quasi-isomorphisms of DGLAs:
  \[\g^+ \stackrel{\biota}{\longrightarrow} \hg \longleftarrow
    \hg^{\comb}.\]
\end{Theorem}

\begin{proof}
  The left arrow is the morphism $\biota$ in
  Theorem~\ref{thm:sz-koszul}.  The right arrow is a direct
  consequence of a result proved by Egas~\cite{Ega}, that the
  combinatorial sewing maps defined in~\cite{WahWes} between ribbon
  graphs agree with the geometric sewing maps.
\end{proof}

By the homotopy invariance of Maurer-Cartan moduli spaces,
Theorem~\ref{thm:costello} implies the following result.

\begin{Theorem}
  \label{thm:comb-vertex}
  In the DGLA
 \[ \hg^{{\comb}}= \left ( \bigoplus_{g, k\geq 1, l}
     C^{\comb}_*(M_{g,k,l}^\fr,\usgn) \right )
   \series{\hbar,\lambda} \]
 there exists a degree $-1$ element
 $\hcV^\comb$, unique up to homotopy, of the form
 \[\hcV^\comb=\sum_{g, k\geq 1, l} \hcV^\comb_{g,k,l}\,\hbar^g
   \lambda^{2g-2+k+l}, \] such that the following conditions are
 satisfied:
 \begin{enumerate}
 \item $\hcV^\comb$ is a Maurer-Cartan element of $\hg^{{\comb}}$,
   i.e., $\hcV^\comb$ satisfies the following equations for each triple
   $(g,k\geq 1,l)$:
   \[ (\partial+uB) \hcV^\comb_{g,k,l}+\Delta\hcV^\comb_{g-1,k,l+2} +
     \iota\hcV^\comb_{g,k-1,l+1} +
     \frac{1}{2}\sum\frac{1}{r!}\{\hcV^\comb_{g_1,k_1,l_1},\hcV^\comb_{g_2,k_2,l_2}\}_r=0.\]
   The last sum is over all $r\geq 1$ and all $(g_1, g_2, k_1, k_2, l_1,
   l_2)$ such that
   \begin{align*}
     g_1+g_2 +r -1 & = g\\
     k_1+k_2 -r & = k\\
     l_1+l_2 -r & = l.
   \end{align*}
   
 \item\[ \hcV^\comb_{0,1,2} =\frac{1}{2}\;\;\;
     \begin{tikzpicture}[baseline={([yshift=-2ex]current bounding
         box.center)},scale=0.3] \draw [thick] (0,0) to (0,2); \draw [thick]
       (-0.2, 1.8) to (0.2, 2.2); \draw [thick] (0.2, 1.8) to (-0.2, 2.2);
       \draw [thick] (0,0) to (-2,0); \draw [thick] (0,0) to (2,0); \draw
       [thick] (-2.2,0) circle [radius=0.2]; \draw [thick] (2.2,0) circle
       [radius=0.2];
     \end{tikzpicture} \]
   corresponding to one-half times a point class in
   $C_0^\comb(M_{0,1,2}^\fr)_\hS$. 
 \end{enumerate}
\end{Theorem} 

\begin{Definition}
  We shall refer to $\hcV^\comb_{g,k,l}$ as {\em combinatorial string
    vertices}. Each of them is a $\bbQ$-linear combination of ribbon
  graphs. 
\end{Definition}

Combinatorial string vertices can be explicitly computed by solving
increasingly large linear systems of equations: the Maurer-Cartan
equation determines $\cV_{g,k,l}$ uniquely (up to $(\del+uB)$-exact
elements) from previously computed string vertices.

As examples, we present explicit formulas for the first few
combinatorial string vertices, ignoring orientations of ribbon graphs
and associated signs.  We follow the conventions of~\cite{WahWes} when
drawing ribbon graphs, except that we also add a cross at the end of
the unique leaf associated to an input cycle of the ribbon
graph. Since each string vertex is an equivariant chain, our
convention, with respect to the circle parameter, is that if we do not
indicate a power of $u$, then the corresponding face or output is
labeled by $u^{0}$.
\begin{align*} \hcV^\comb_{0,1,2}& =\frac{1}{2}\;\;\;
  \begin{tikzpicture}[baseline={([yshift=-1.2ex]current bounding
box.center)},scale=0.3] \draw [thick] (0,0) to (0,2); \draw [thick]
(-0.2, 1.8) to (0.2, 2.2); \draw [thick] (0.2, 1.8) to (-0.2, 2.2);
\draw [thick] (0,0) to (-2,0); \draw [thick] (0,0) to (2,0); \draw
[thick] (-2.2,0) circle [radius=0.2]; \draw [thick] (2.2,0) circle
[radius=0.2];
\end{tikzpicture}
\end{align*}

\[
  \hcV^\comb_{1,1,0}=\frac{1}{24}\;\;\;
  \begin{tikzpicture}[baseline={([yshift=-.4ex]current
bounding box.center)},scale=0.3] \draw [thick] (0,2) circle
[radius=2]; \draw [thick] (-2,2) to (-0.6,2); \draw [thick] (-0.8,2.2)
to (-0.4,1.8); \draw [thick] (-0.8, 1.8) to (-0.4, 2.2); \draw [thick]
(-1.4142, 0.5858) to [out=40, in=140] (1, 0.5); \draw [thick] (1.25,
0.3) to [out=-45, in=225] (2, 0.2); \draw [thick] (2,0.2) to [out=45,
in=-50] (1.732, 1); \node at (0.7, 2.2) {$\scriptstyle{u^{-1}}$};
\end{tikzpicture} \;\;+\;\;\frac{1}{4}\;\;\;
\begin{tikzpicture}[baseline={([yshift=-.4ex]current bounding
box.center)},scale=0.3] \draw [thick] (0,2) circle [radius=2]; \draw
[thick] (0,0) to (0,1.4); \draw [thick] (-0.2, 1.2) to (0.2, 1.6);
\draw [thick] (-0.2, 1.6) to (0.2, 1.2); \draw [thick] (0,0) to
[out=80, in=180] (0.5, 1); \draw [thick] (0.5,1) to [out=0, in=100]
(0.9, 0.4); \draw [thick] (0,0) to [out=-80, in=180] (0.5, -1); \draw
[thick] (0.5, -1) to [out=0, in=-100] (0.9, 0);
\end{tikzpicture}\]

\[ \hcV^\comb_{0,2,1}=
  \frac{1}{2}
  \begin{tikzpicture}[baseline={([yshift=-.4ex]current
bounding box.center)},scale=0.3] \draw (0,0)
node[cross=2pt,label=above:{}] {}; \draw[ultra thick] (0,0) to
(0.8,0); \draw (1,0) circle (.2); \draw (5.2,0)
node[cross=2pt,label=above:{}] {}; \draw [thick] (5.2,0) to (6.2,0);
\draw [thick] (1.2,0) to (2.2,0); \draw [thick] (4.2,0) circle
[radius=2];
\end{tikzpicture}+\frac{1}{2}
\begin{tikzpicture}[baseline={([yshift=-.4ex]current
bounding box.center)},scale=0.3] \draw (2,0)
node[cross=2pt,label=above:{}] {}; \draw (6.2,0)
node[cross=2pt,label=above:{}] {}; \draw [thick] (6.2,0) to (7.2,0);
\draw [thick] (2.1,0) to (3.2,0); \draw [thick] (5.2,0) + (-85:2)
arc(-85:265:2); \draw [thick] (5.2,-2) circle (.2); \draw [ultra
thick,domain=230:265] plot ({5.2+2*cos(\x)}, {2*sin(\x)});
\end{tikzpicture}+\begin{tikzpicture}[baseline={([yshift=-.4ex]current
bounding box.center)},scale=0.3] \draw (2,0)
node[cross=2pt,label=above:{}] {}; \draw (6.2,0)
node[cross=2pt,label=above:{}] {}; \draw [thick] (6.2,0) to (7.2,0);
\draw [thick] (2.1,0) to (3.2,0); \draw [thick] (5.2,0) circle
[radius=2]; \draw (1.8,-1.6) circle (.2); \draw [thick] (3.2,0) to
(1.9,-1.4);
\end{tikzpicture}\]

\[\hcV^\comb_{0,3,0}=
\frac{1}{2} \begin{tikzpicture}[baseline={([yshift=-.4ex]current
bounding box.center)},scale=0.3] \draw (2,0)
node[cross=2pt,label=above:{}] {}; \draw (5.2,1)
node[cross=2pt,label=above:{}] {}; \draw (5.2,-1)
node[cross=2pt,label=above:{}] {}; \draw [thick] (5.2,-1) to (7.2,0);
\draw [thick] (5.2,1) to (5.2,2); \draw [thick] (3.2,0) to (7.2,0);
\draw [thick] (2.1,0) to (3.2,0); \draw [thick] (5.2,0) circle
[radius=2];
\end{tikzpicture}+
\frac{1}{2} \begin{tikzpicture}[baseline={([yshift=-.4ex]current
bounding box.center)},scale=0.3] \draw [thick] (0,0) circle
[radius=1.5]; \draw [thick] (3,0) circle [radius=1.5]; \draw (-.5,0)
node[cross=2pt,label=above:{}] {}; \draw (2.5,0)
node[cross=2pt,label=above:{}] {}; \draw (5.5,0)
node[cross=2pt,label=above:{}] {}; \draw [thick] (5.5,0) to (4.5,0);
\draw [thick] (2.5,0) to (1.5,0); \draw [thick] (-.5,0) to (-1.5,0);
\end{tikzpicture}+
\frac{1}{2} \begin{tikzpicture}[baseline={([yshift=-1ex]current
bounding box.center)},scale=0.3] \draw [thick] (0,0) circle
[radius=1.5]; \draw [thick] (3,0) circle [radius=1.5]; \draw (-.5,0)
node[cross=2pt,label=above:{}] {}; \draw (1.5,1.5)
node[cross=2pt,label=above:{}] {}; \draw (3.5,0)
node[cross=2pt,label=above:{}] {}; \draw [thick] (3.5,0) to (4.5,0);
\draw [thick] (1.5,1.5) to (1.5,0); \draw [thick] (-.5,0) to (-1.5,0);
\end{tikzpicture}\]

\[\hcV^\comb_{0,1,3} =\frac{1}{2}\;
  \begin{tikzpicture}[baseline={([yshift=-1.2ex]current bounding
box.center)},scale=0.3] \draw (0,1.5) node[cross=2pt,label=above:{}]
{}; \draw [thick] (0,0.2) to (0,1.5); \draw [thick] (-0.2,0) to
(-2,0); \draw [thick] (0.2,0) to (2,0); \draw [thick] (-2.2,0) circle
[radius=0.2]; \draw [thick] (2.2,0) circle [radius=0.2]; \draw [thick]
(0,0) circle [radius=0.2]; \draw [ultra thick] (.2,0) to (1,0);
\end{tikzpicture}+\frac{1}{2}\;
  \begin{tikzpicture}[baseline={([yshift=-1.2ex]current bounding
box.center)},scale=0.3] \draw (1,1) node[cross=2pt,label=above:{}] {};
\draw [thick] (1,0) to (1,1); \draw [thick] (-0.2,0) to (-2,0); \draw
[thick] (0.2,0) to (2,0); \draw [thick] (-2.2,0) circle [radius=0.2];
\draw [thick] (2.2,0) circle [radius=0.2]; \draw [thick] (0,0) circle
[radius=0.2]; \draw [ultra thick] (0,-.2) to (0,-1);\end{tikzpicture}
+\frac{1}{2}\;
  \begin{tikzpicture}[baseline={([yshift=-.4ex]current bounding
box.center)},scale=0.3] \draw (1,1) node[cross=2pt,label=above:{}] {};
\draw [thick] (1,0) to (1,1); \draw [thick] (0,0) to (-2,0); \draw
[thick] (0,0) to (2,0); \draw [thick] (-2.2,0) circle [radius=0.2];
\draw [thick] (2.2,0) circle [radius=0.2]; \draw [thick] (0,-1) circle
[radius=0.2]; \draw [thick] (0,0) to (0,-.8); \node at (1.2,-1)
{$\scriptstyle{u^{-1}}$};
\end{tikzpicture}+\frac{1}{6}\;
  \begin{tikzpicture}[baseline={([yshift=-2ex]current bounding
box.center)},scale=0.3] \draw (-1,1) node[cross=2pt,label=above:{}]
{}; \draw [thick] (-1,0) to (-1,1); \draw [thick] (0,0) to (-2,0);
\draw [thick] (0,0) to (2,0); \draw [thick] (-2.2,0) circle
[radius=0.2]; \draw [thick] (2.2,0) circle [radius=0.2]; \draw [thick]
(1,1) circle [radius=0.2]; \draw [thick] (1,0) to (1,.8); \node at
(-1,1.5) {$\scriptstyle{u^{-1}}$};\end{tikzpicture}.\]
Note, for example, that the coefficients $1/24$ and $1/6$ in
$\hcV_{1,1,0}^\comb$ and $\hcV_{0,1,3}^\comb$, respectively, can be
used to recover the integrals of a $\psi$-class on $\bM_{1,1}$ and
$\bM_{0,4}$, respectively.  Indeed, we have
\[ \int_{\bM_{1,1}} \psi = \frac{1}{24}, \;\;\; \int_{\bM_{0,4}} \psi
  = 3! \cdot \frac{1}{6}=1.\]
This is explained with detail in~\cite{CalTu}.

Using computer calculations we are able to compute the string
vertices up to $\cV^\comb_{2,1,0}$. See~\cite{Calrg} for details on
the techniques used to speed up the computations.

\section{The original, non-constructive definition of CEI}
\label{sec:def}

In this section we define the CEI potential $\cD^{A,s}$ of a cyclic
$A_\infty$-algebra $A$ and a splitting $s$.  The definition we give in
this section is essentially the same as Costello's~\cite{Cos2}, and
hence not yet constructive (it uses non-explicit inverses of certain
quasi-isomorphisms).  We assume that the algebra $A$ satisfies
condition $(\dagger)$ in~\eqref{cond:dagger} and we continue to work
with the representation of the PROP of annuli
\[ (L = C_*(A)[d], b, B, \pairing_\Muk). \]

\subsection{Topological-conformal field theories}
Our first goal is to construct a canonical Maurer-Cartan element
$\hbA \in \hh$.  The dg-vector space $L$ carries the structure of a 2d
TCFT, inducing a natural map of DGLAs
\[ \rho^A:\hg^\comb \ra \hh. \]
We take $\hbA$ to be the image of the string vertex
$\cV^\comb\in\hg^\comb$ under this map.  See the discussion
in~(\ref{subsec:alggen}) for a conceptual guide to this construction.

Kontsevich-Soibelman~\cite{KonSoi} and Costello~\cite{Cos1} were the
first to sketch the construction of a 2d TCFT associated with a cyclic
$A_\infty$-algebra.  The state space of this TCFT is the complex
$(L,b)$.  Their construction was later treated in greater detail
(including signs) in Wahl-Westerland~\cite{WahWes}; an improved
treatment of signs (making them explicitly computable) will appear
in~\cite{CalChe}.

\begin{Theorem}
  \label{thm:tcft}
  The shifted reduced Hochschild chain complex $L = C_*(A)[d]$ of a
  cyclic $A_\infty$-algebra $A$ of Calabi-Yau dimension $d$ carries
  the structure of a positive boundary 2d TCFT.  In other words there
  are natural action maps for every $g\geq 0$, $k\geq 1$, $l\geq 0$
  satisfying $2-2g-k-l<0$
  \[ \rho_{g,k,l}^A: C^\comb_*(M_{g,k,l}^\fr) \ra \Hom ( L^{\otimes
      k}, L^{\otimes l})[d(2-2g-2k)]\]
  which are compatible with the sewing operations of framed Riemann
  surfaces.
\end{Theorem}

\begin{Remark}
  The signs involved in proving this theorem are pretty involved,
  especially when $d$ is odd.  In~\cite{WahWes} the authors work with
  the usual Hochschild chain complex, without shifting by $d$.  This
  makes the treatment of signs rather not transparent for odd $d$.
  Following Costello's idea~\cite{Cos2} to include the shift $[d]$
  into the Hochschild chain complex, a more explicit and complete
  description of the signs involved in the construction of $\rho^A$
  will be presented in the upcoming paper~\cite{CalChe}.
\end{Remark}

\begin{Remark}
  The positive boundary condition is the requirement that $k\geq 1$ in
  the definition of the above maps.  This is obvious from the fact
  that ribbon graphs must have at least one face, but it is also
  needed in order to avoid difficulties associated with the fact that
  $L$ is in general infinite dimensional, so it cannot carry a
  non-degenerate bilinear pairing.
\end{Remark}

\begin{Remark}
  Note that the shift by $d(2-2g-2k)$ is even, so it will not affect
  anything in the $\Z/2\Z$-graded cases.  It will become relevant in
  Appendix~\ref{app:z-graded}.
\end{Remark}

Both $C^\comb_*(M_{g,k,l}^\fr)$ and $\Hom ( L^{\otimes k}, L^{\otimes
  l})$  carry circle and symmetric group actions, so we can perform
the ``$\hS$'' operation on them -- twisting the action by the local
system $\usgn$ on inputs, and taking homotopy quotients by the circle
actions and ordinary quotients by the symmetric group actions.

First, twisting the TCFT representation by $\usgn$ we get new maps
\begin{align*}
  \rho_{g,k,l}^{A,\tw}: C^\comb_*(M_{g,k,l}^\fr, \usgn) &\ra
        \Hom((L[1])^{\otimes k}, L^{\otimes l})\\
  \rho_{g,k,l}^{A,\tw} \big( \Gamma, p_1\wedge\cdots\wedge p_k \big) &
        = \rho_{g,k,l}^A(\Gamma) \circ s^{\otimes k}
\end{align*}
where $\Gamma$ is an element of $C^\comb_*(M_{g,k,l}^\fr, \usgn)$
and $s: L[1] \ra L$ is the shift map,
\[ s(a) = (-1)^{|a|} a. \]
Recall that the differential of $L[1]$ is $\partial = -b$, so that
$\partial s =- b s$, and therefore $ \rho_{g,k,l}^{A,{\sf tw}}$ is a
chain map:
\begin{align*}
  (\partial \rho_{g,k,l}^{A,\tw} ) \big( \Gamma,
    p_1\wedge\cdots\wedge p_k \big)  &= b \rho_{g,k,l}^{A}(\Gamma)
      s^{\otimes k} - (-1)^{|\Gamma|} 
      \rho_{g,k,l}^{A}  (\Gamma)s^{\otimes k}(-b)\\
  & = b \rho_{g,k,l}^{A}(\Gamma) s^{\otimes k} -(-1)^{|\Gamma|+k}
      \rho_{g,k,l}^{A}(\Gamma) b s^{\otimes k}\\
  & = (\partial \rho_{g,k,l}^{A}(\Gamma) ) s^{\otimes k}\\
  & = \rho_{g,k,l}^{A,\tw}\big( \partial( \Gamma,p_1\wedge\cdots\wedge p_k)\big).
\end{align*}
We remind the reader that the $\usgn$ local system has a shift $[-k]$
built into it, so the maps $\rho_{g,k,l}^{A,\tw}$ respect grading.

Further taking homotopy quotient with respect to the action of
$(S_k\times S_l)\ltimes (S^1)^{k+l}$ we obtain action maps on
$C^\comb_*(M_{g,k,l}^{\sf fr},\usgn)_\hS$, which we will still denote
by $\rho_{g,k,l}^{A,\tw}$:
\[ \rho_{g,k,l}^{A,\tw}: C^\comb_*(M_{g,k,l}^{\sf
fr},\usgn)_\hS \ra \Hom^\cont \big( \Sym^k(L_+[1]),
\Sym^l L_-\big).\]

It is worth describing explicitly the graded vector space
$C^\comb_*(M_{g,k,l}^{\sf fr},\usgn)_\hS$.  Consider the set of
oriented ribbon graphs of type $g,k,l$, whose input faces and output
white vertices have been marked with non-positive powers of $u$.  An
isomorphism between such graphs is a ribbon graph isomorphism, allowed
to permute the input faces among themselves, as well as the output
white vertices among themselves, as long as the $u$-power markings are
matched.  Such an isomorphism is said to be positively oriented if it
either preserves the orientation of the ribbon graphs and is an even
permutation on the faces, or gives the opposite orientation on the
ribbon graphs and is an odd permutation on the inputs.  It will be
said to be negatively oriented otherwise.  The space
$C^\comb_*(M_{g,k,l}^{\sf fr},\usgn)_\hS$ is the $\Q$-vector space
spanned by such $u$-marked graphs, modulo the relations
$\Gamma_1 = \Gamma_2$ if there is a positively oriented isomorphism
$\Gamma_1\iso \Gamma_2$ and $\Gamma_1 = -\Gamma_2$ if there is a
negatively oriented one.

The degree of a $u$-marked graph whose inputs are labeled by
$\{u^{-a_i}\}$ and outputs by $\{u^{-b_j}\}$ is
\[ \deg(\Gamma, p_1\wedge\cdots \wedge p_k) = \deg(\Gamma) - k +
  2\sum_{i=1}^k a_i + 2 \sum_{j = 1}^l b_j. \]

The action of a $u$-marked graph $\Gamma$ on $\Sym(L_+[1])$ is done by
first matching $u$-powers, and then inserting the corresponding
elements of $L$ into the inputs of the unmarked $\Gamma$ with a sign
correction of $(-1)^{u\text{-power}}$ at each input. The output
$u$-powers multiply the elements of $L$ read from the corresponding
outputs. 

For example, in the case $k=1$, $l =1$, for a graph whose input is
labeled by $u^{-a}$ and output by $u^{-b}$, the action is given
explicitly by
\[ \rho_{g,1,1}^{A,\tw}\big( \Gamma \big)  (\alpha\cdot u^{a'})
  = \begin{cases}
    (-1)^{a+|\alpha|}\cdot \rho_{g,1,1}^A (\Gamma)(\alpha)\cdot u^{-b}
  & \mbox{ if } a=a';\\ 
    0 & \mbox{ otherwise.}
  \end{cases}
\]
(The sign $(-1)^{|\alpha|}$ comes from the operator $s$, while
$(-1)^a$ is the sign correction above.)  For general $k\geq 1$ we need
to shuffle the inputs, while for $l\geq 1$ we need to symmetrize the
outputs.

\begin{Lemma}
  The maps $\rho_{g,k,l}^{A,\tw}$ are chain maps.
\end{Lemma}

\begin{proof}
  We have already checked compatibility with the differential $b$, so
  it only remains to prove the compatibility with $uB$.  The operator
  $B$ acts on a graph $\Gamma$ by acting either on an input (call
  this $B^{\sf in}$) or on an output ($B^{\sf out}$).  Since the
  twisting of the action does not affect the outputs in any way, the
  compatibility with $B^{ \sf out}$ is obvious.

  For the compatibility with $B^{\sf in}$ we will again check the
  $k=1$, $l=1$ case; the more general case of arbitrary $k$, $l$ is
  only slightly more complex in terms of notation. Assume that the
  input of the graph is labeled by $u^{-a}$ and the output by
  $u^{-b}$.  Then we have 
  \begin{align*}
    \rho_{g,1,1}^{A,\tw} \big( uB^{\sf in}( \Gamma ) \big) (\alpha\cdot
    u^{a'})  & =  \delta_{a-1,a'} \cdot (-1)^{a-1+|\alpha|} \cdot \rho_{g,1,1}^A
      (B^{\sf in}\Gamma) (\alpha) \cdot u^{-b}  \\
    & = \delta_{a, a'+1} \cdot (-1)^{a+|B\alpha|} \cdot \rho_{g,1,1}^A
      (\Gamma) (B\alpha) \cdot u^{-b} \\
    & = \rho_{g,1,1}^{A,\tw}(\Gamma) \left (uB (\alpha\cdot u^{a'}) \right).
  \end{align*}
\end{proof}

\begin{Remark}
  This calculation explains why the extra sign $(-1)^{u\text{-power}}$
  is needed in the definition of the action $\rho^{A,\tw}$. This sign
  also matches well with the sesquilinear property in the definition
  of residue pairing, see~(\ref{subsec:circle}).  This sign is also
  used in the following proposition to match the two $\iota$ operators
  in $\hg$ and $\hh_A$, since {\sl a priori} the $\iota$ operator of
  $\hg$ is $u$-linear, while the $\iota$ operator of $\hh_A$ is
  $u$-sesquilinear.
\end{Remark}

\begin{Proposition}
  \label{prop:universal}
  The twisted TCFT structure maps $\{ \rho_{g,k,l}^{A,\tw} \}$ induce
  a morphism of DGLAs
  \[ \rho^{A,\tw}: \hg^\comb \ra \hh_A.\]
\end{Proposition}

\begin{proof}
  This is evident from the construction, although one needs to check
  the signs carefully. We omit the details.
\end{proof}

\begin{Remark}
  This proposition shows that the Sen-Zwiebach DGLA $\g$ (or,
  equivalently, $\hg$) is the ``universal'' DGLA among DGLAs
  constructed from TCFTs. This is more transparent than the point of
  view in~\cite{CCT}, where the DGLAs $\hg$ and $\hh_A$ were
  constructed on equal footing.
\end{Remark}

\subsection{\texorpdfstring{Definition of $\hbA$}{Definition of
    hat-beta-A}}
\label{subsec:defbeta}
We now come to the main point of the first part of the construction of
CEI.  In $\hg$ we have a canonical
Maurer-Cartan element $\hcV^\comb$, the combinatorial string vertex
\[\hcV^\comb=\sum_{g,k\geq 1,l} \hcV_{g,k,l}^\comb\,
  \hbar^g\,\lambda^{2g-2+k+l},\]
of Theorem~\ref{thm:comb-vertex}. Its push-forward under
$\rho^{A,\tw}$ is then a canonical Maurer-Cartan element $\hbA\in
\hh_A$ of the form 
\[ \hbA= \rho^{A,{\sf tw}}_*(\hcV^\comb)=
  \sum_{g,k\geq 1,l}\rho^{A,{\sf tw}}(\hcV_{g,k,l}^\comb)\,
  \hbar^g\,\lambda^{2g-2+k+l}. \]
The tensors
\[ \hbA_{g,k,l} = \rho^{A,\tw}(\hcV^{\sf comb}_{g,k,l}) \in \Hom^\cont
  \big ( \Sym^k(L_+[1]), \Sym^lL_-\big) \]
will play a key role in the definition and computation of the CEI of
the cyclic $A_\infty$-algebra $A$.  They only depend on $A$,
and not on any future choice of splitting.

\subsection{Splittings of the non-commutative Hodge
filtration} 
To obtain invariants of $A$ that are similar to those from
Gromov-Witten theory we need a further ingredient: a choice of
splitting $s$ of the Hodge filtration.

\begin{Definition}
  \label{def:splitting}
  A splitting of the non-commutative Hodge filtration of $A$ is a
  graded map of $\Z/2\Z$-graded vector spaces
  \[ s:H_*(L)\ra H_*(L_+) \]
  satisfying the following two conditions:
  \begin{itemize}
  \item[S1.] {\em (Splitting condition.)} $s$ splits the canonical
    projection $H_*(L_+) \ra H_*(L)$.
  \item[S2.] {\em (Lagrangian condition.)}  $ \langle s(x),\,
    s(y)\rangle_{\sf hres} = \langle x,\,y\rangle_\Muk \mbox{ for any } x,
    y \in H_*(L)$.
  \end{itemize}
\end{Definition}

\begin{Remark}
  In many circumstances it is useful to place further restrictions on
  the allowed splittings.  For example one may impose certain
  homogeneity conditions, or ask for the splitting to be compatible
  with the cyclic structure of the $A_\infty$-algebra.  In this paper
  we only need conditions S1 and S2, but we refer the reader
  to~\cite[Definition 3.7]{AmoTu} for more information on these
  possible restrictions.
\end{Remark}

Since $A$ is assumed to satisfy the Hodge-de Rham degeneration
property, we can endow the homology $H = H_*(L)$ with the trivial
circle action.  According to our conventions on circle actions we then
have the following graded vector spaces with trivial differentials:
\[ H^\Tate = H\laurent{u}, \;\;\; H_+ = H\series{u}, \;\;\; H_- =
  H[u^{-1}].\]
A splitting $s: H \ra H_*(L_+)$, assumed to satisfy the
Lagrangian condition S2, extends by $u$-linearity
to an isomorphism of {\em symplectic} vector spaces. 
\[ s: \big( H^\Tate, \langle-,-\rangle_{{\sf res}}\big) \ra \big(
  H_*(L^\Tate), \langle-,-\rangle_{{\sf res}}\big).\]

\subsection{Definition of CEI}
\label{subsec:defi}
Recall (Lemma~\ref{lem:biota}) that there is a quasi-isomorphism of
DGLAs
\[ \biota: \h_A^+ \ra \hh_A.\]
The Maurer-Cartan moduli space is invariant\footnote{We certainly need
  to work with certain convergence properties when defining
  Maurer-Cartan elements. In the current case, we require
  Maurer-Cartan elements to have strictly positive power of
  $\lambda$.} under such quasi-isomorphisms.  Thus there exists a
Maurer-Cartan element $\betaA\in\h_A^+$,
\[ \betaA =\sum_{g,n\geq 1}\betaA_{g,n}\,\hbar^g\,\lambda^{2g-2+n}, \]
defined by the property that $\biota(\betaA)$ is gauge equivalent
to the element $\hbA$ defined above.  By definition, it is unique up to
gauge equivalences.

Recall from~(\ref{subsec:Koszul}) that we have a short exact
sequence of DGLAs
\[ 0 \ra \bbK[1]\series{\hbar, \lambda} \ra \h_A \ra
\h_A^+ \ra 0. \]
The canonical vector space direct sum decomposition
\[\h_A = \bbK[1]\series{\hbar, \lambda} \oplus \h_A^+ \]
allows us to regard the Maurer-Cartan element $\betaA\in \h_A^+$ as an
element in $\h_A$, by taking its
$\bbK[1]\series{\hbar,\lambda}$-component to be zero\footnote{This
  choice is only made for the sake of proving
  Proposition~\ref{prop:action}.  A more reasonable choice would be to
  force the dilaton equation. But this would present additional
  difficulties in Proposition~\ref{prop:action}.}.  Even though
$\h_A^+$ is not a sub-Lie algebra of $\h_A$, we will prove in
Lemma~\ref{lem:extension} that $\betaA\in \h_A$ still satisfies the
Maurer-Cartan equation.  We defer the proof of this fact until
Section~\ref{sec:trivial-2}. 

There is a well-known relationship between Maurer-Cartan elements of
the DGLA $\h_A$ with closed elements of the differential
$b+uB+\hbar\Delta$, see Lemma~\ref{lem:mc-closed} below.  This
relationship involves computing the exponential $\exp(\beta/\hbar)$
for a Maurer-Cartan element $\beta\in \h_A$.  However, this
exponential does not make sense unless we impose additional finiteness
conditions.  For this reason we extend the definition of the Weyl
algebra $\cW\big(L^\Tate\big)$ to include
an additional even variable $\lambda$:
\[ \cW\big(L^\Tate\big) = T\big( L^\Tate \big)\series{\hbar,
\lambda}/\big (
\alpha\otimes\beta-(-1)^{|\alpha||\beta|}\beta\otimes\alpha=\hbar\langle
\alpha,\beta\rangle_{\sf res}\big).\ \]

We localize the Weyl algebra at $\hbar$ and complete in the
$\lambda$-adic topology to get the localized and completed algebra
\[ \widehat{\cW}_\hbar\big(L^\Tate\big) = \varprojlim_n
\cW\big(L^\Tate\big)[\hbar^{-1}]/(\lambda^n).\] Infinite power series
of the form
\[ \sum_{k\geq0} \alpha_k \lambda^k \hbar^{-k}\] exist in
$\widehat{\cW}_\hbar \big(L^\Tate\big)$ but not in
$\cW\big(L^\Tate\big)$.

By analogy with the construction of $\widehat{\cW}(L^\Tate)$ we
denote by $\widehat{\cF}_{A,\hbar}$ and $\wSym (L_-)
\series{\hbar,\lambda}$ the localized and completed versions of the
Fock space and the symmetric algebra.

With these preparations we can state the following result, which is well
known.

\begin{Lemma}
  \label{lem:mc-closed}
  An element $\beta\in \lambda\cdot \h_A[1]$ of even degree satisfies
  the Maurer-Cartan equation
  \[ (b+uB+\hbar \Delta) \beta + \frac{1}{2} \{ \beta, \beta\} =0\]
  if and only if $\cD = \exp(\beta/\hbar)$ is
  $(b+uB+\hbar\Delta)$-closed in $\wSym(L_-)\series{\hbar,\lambda}$.

  Moreover, two such Maurer-Cartan elements $\beta_1$ and $\beta_2$
  are gauge equivalent if and only if the corresponding elements
  $\cD_1$ and $\cD_2$ are homologous.  All these identities hold in
  the algebra $\wSym(L_-)\series{\hbar,\lambda}$.
\end{Lemma}

Using this lemma, the Maurer-Cartan element $\betaA\in \h_A$ we
obtain a $(b+uB+\hbar\Delta)$-closed element
\[ \exp(\betaA/\hbar) \in \wSym L_-\series{\hbar, \lambda}\] in the
localized and completed symmetric algebra. The fact that the string
vertex is unique up to homotopy shows that the homology class of this
element depends only on the cyclic $A_\infty$-algebra $A$.  We denote
it by
\[ \mbox{$\cD_\abs^A$} = \left [ \mbox{$\exp$}(\betaA/\hbar)\right ] \in
  H_*\big( \mbox{$\wSym$} L_-\series{\hbar, \lambda},
  b+uB+\hbar\Delta\big)\cong H_*(\widehat{\cF}_\hbar) \]
and call it the {\em abstract total descendent potential} of $A$. It
is an even element in
$H_*\big( \wSym L_-\series{\hbar, \lambda}, b+uB+\hbar\Delta\big)$.

Note that the abstract potential $\cD_\abs^A$ is not yet of the form
we want; we'd like it to live in the homology of the same space, but
with respect to the $b$ differential instead of $b+uB+\hbar\Delta$.
Our goal is therefore to use a splitting to construct an isomorphism
\[ H_*\big( \wSym L_-\series{\hbar, \lambda}, b+uB+\hbar\Delta\big)
  \iso H_*\big (  \wSym L_-\series{\hbar, \lambda}, b \big ) = \wSym
  H_-\series{\hbar, \lambda} . \]
In other words we want to homotopy trivialize the operators $B$ and
$\Delta$ on the $b$-homology of $L$.

Given a splitting $s$ of the non-commutative Hodge filtration of $A$
we obtain an induced symplectomorphism
$s:H_*(L)^\Tate \ra H_*(L^\Tate)$, which in turn induces an
isomorphism of Weyl algebras
\[ \Phi^s: \widehat{\mathcal{W}}_\hbar\big( H_*(L)^\Tate\big) \ra
  \widehat{\mathcal{W}}_\hbar \big( H_*(L^\Tate) \big) . \]
Define $\Psi^s$ to be the composition of the three other maps in the
diagram below
\[
  \begin{CD}
    \widehat{\cW}_\hbar\big( H_*(L)^{{\sf Tate}}\big) @>\Phi^s>>
    \widehat{\cW}_\hbar \big( H_*(L^{{\sf Tate}}) \big)=H_*\big(
    \widehat{\cW}_\hbar( L^{{\sf Tate}})\big)\\ @AiAA @VVpV \\
    \widehat{\Sym}_\hbar H_*(L)_-\series{\hbar,\lambda}@>
    \Psi^s>>H_*\big( \widehat{\mbox{$\cF$}}_\hbar\big)\cong
    H_*(\widehat{\Sym}_\hbar L_- \series{\hbar,\lambda})
\end{CD}
\]
It is a graded vector space isomorphism.  Here the left vertical map
is the inclusion of the symmetric algebra generated by the Lagrangian
subspace $H_*(L)_-$ into the Weyl algebra, while the right vertical
map is the canonical quotient map from the Weyl algebra to the Fock
space.  The abstract total descendent potential $ \cD^A_\abs$ lives in
the lower right corner.

\begin{Definition}
  \label{def:main}
  The {\sl total descendent potential} $\cD^{A,s} \in \wSym H_*(L)_-
  \series{\hbar, \lambda}$ of the pair $(A,s)$ is the preimage of the
  abstract total descendent potential $\cD^A_\abs$ under the map
  $\Psi^s$,
  \[ \cD^{A,s} =(\Psi^s)^{-1}( \cD^A_\abs ).\]
  The genus $g$, $n$-point function $F^{A,s}_{g,n}\in \Sym^n H_*(L)_-$
  is defined by the identity
  \[ \sum_{g,n}F^{A,s}_{g,n}\cdot \hbar^{g-1}\,\lambda^{2g-2+n} =
   \ln (\cD^{A,s}).\]
  The right hand term $\ln (\cD^{A,s})$ will be denoted by
  $F^{A,s}$. 
\end{Definition}

\subsection{Givental group action}
Immediate from the above discussion is the compatibility of the
definition of the total descendent potential above with the action of
the Givental group on the space of splittings, which we discuss now.

Denote by $\mathfrak{G}_A$ the Givental group of the pair
$(H_+,\pairing_\Muk)$.  Abstractly, it is the subgroup of
automorphisms of the symplectic vector space
$(H^\Tate, \pairing_{\res})$ preserving the Lagrangian subspace $H_+$
and acting as the identity on $H$.  In explicit terms, an element of
$\mathfrak{G}_A$ is of the form
\[ g =\id+g_1\cdot u+g_2\cdot u^2+\cdots \]
with each $g_j\in \End(H)$ required to satisfy
\[ \langle \,g\cdot x, \, g\cdot y\,\rangle_\res= \langle\, x,\,
  y\,\rangle_\res \quad\mbox{for any } x, y\in H^\Tate. \]
If the set of splittings of the non-commutative Hodge filtration is
nonempty, then it is a left torsor over the Givental group, by letting
an element $g\in \mathfrak{G}_A$ act on a splitting
$s: H_+\ra u^{-1}H_*(L_+)$ by pre-composing with $g^{-1}$:
\[ g\cdot s : H_+\stackrel{g^{-1}}{\longrightarrow} H_+
  \stackrel{s}{\longrightarrow} H_*(L_+).\]
An element $g$ of the Givental group acts naturally on the Fock space
$\wSym H_-\series{\hbar, \lambda}$ where the total descendent
potential lives, via the automorphism of the Weyl algebra induced from
the symplectic transformation $g$:
\[
  \begin{CD}
    \widehat{\mathcal{W}}_\hbar\big (H^\Tate\big ) @>\Phi^g>>
    \widehat{\mathcal{W}}_\hbar\big ( H^\Tate \big ) \\
    @AiAA @VV\pi V \\
    \wSym\, H_-\series{\hbar, \lambda} @>\widehat{g}>> \wSym\,
    H_-\series{\hbar, \lambda}
  \end{CD}
\]

\begin{Proposition}
\label{prop:action}
  The construction of the total descendent potential is compatible
  with the action of the Givental group $\mathfrak{G}_A$.  Explicitly,
  for a splitting $s$ and an element $g\in \mathfrak{G}_A$ we have
  \[ \cD^{A,g\cdot s}= \widehat{g}(\cD^{A,s}).\]
\end{Proposition}

\begin{proof} Consider the following diagram:
 \[
   \begin{diagram}[height=2.2em,width=6em,labelstyle=\scriptstyle]
     \widehat{\mathcal{W}}_\hbar\big( H^\Tate\big) &\rTo^{\Phi^{g^{-1}}} &
     \widehat{\mathcal{W}}_\hbar\big( H^\Tate\big) & \rTo^{\Phi^s} &
     H_*\big( \widehat{\mathcal{W}}_\hbar(L^\Tate)\big) \\
     \uTo^i & & \dTo^\pi\uTo_i& & \dTo_p \\
     \wSym\, H_-\series{\hbar, \lambda} & \rTo^{\widehat{g^{-1}}} &
     \wSym\, H_-\series{\hbar, \lambda} &  \rTo^{\quad\quad \Psi^s }&
     H_*\big( \widehat{\cF}_\hbar\big).
   \end{diagram}
 \]
We claim that
\[ p \Phi^{g\cdot s} i = p \Phi^s i \pi \Phi^{g^{-1}} i.\]
The left hand side equals $p \Phi^s \Phi^{g^{-1}} i$. The composition
$i\pi$ in the middle is not the identity, but rather the projection
onto the negative subspace (image of $i$).  Its kernel is by
definition the left ideal generated by the positive subspace
$uH_+$. Thus we have $i\pi=\id \pmod{ uH_+}$.

Now observe that $\Phi^s$ preserves the ideal $uH_+$, just because by
definition a splitting $s$ only contains non-negative powers in the
variable $u$. Followed by the projection map $p$ gives zero, that is,
\[ p\Phi^s i\pi \big( uH_+ \big)=0.\]
We conclude that $p\Phi^s=p \Phi^s i \pi$, and hence
\[p \Phi^{g\cdot s} i=p \Phi^s \Phi^{g^{-1}} i=p \Phi^s i \pi
  \Phi^{g^{-1}} i.\]
This implies that
\begin{align*}
  \cD^{A,g\cdot s} &=(p \Phi^{g\cdot s}i)^{-1}(\cD^A_\abs)\\
       &= (\pi \Phi^{g^{-1}} i)^{-1} (p \Phi^s i)^{-1}(\cD^A_\abs)\\
         &=(\pi \Phi^{g^{-1}}i)^{-1}(\cD^{A,s})= \widehat{g}(\cD^{A,s})
\end{align*}
\end{proof}
\section{\texorpdfstring{The trivializing morphism $\cK$}{The
    trivializing morphism K}}
\label{sec:trivial-1}

In the previous section we constructed tensors $\hbA_{g,k,l} $,
defined as the components of the image of the combinatorial string
vertex under the twisted TCFT map $\rho^{A,\tw}$.  We want to
write explicit formulas expressing the CEI $F^{A,s}_{g,n}$ in terms of
these tensors and a splitting $s$.

In order to accomplish this goal we shall prove that the DGLAs $\h_A$
and $\hh_A$ are both ``homotopically trivial'' by constructing
explicitly computable $L_\infty$ quasi-isomorphisms (that depend on a
splitting $s$) from them to simpler DGLAs $\h_A^\triv$ and
$\hh_A^\triv$, whose differentials only involve $b$, $B$ and $\iota$
(but not $\Delta$) and the bracket is trivial.  The CEI potential is
then the push-forward of the abstract CEI potential defined earlier
under these trivializing quasi-isomorphisms.

This section is devoted to studying the comparison map
$\cK:\h_A \ra \h_A^\triv$; in the next one we shall construct the
corresponding quasi-isomorphism $\hcK$ between the Koszul resolution
DGLA $\hh_A$ and its trivialization $\hh_A^\triv$.

The DGLA structure on
\[ \h_A = \Sym (L_-) [1] \series{\hbar,\lambda} \]
was constructed in~(\ref{subsec:dgla1}) starting from the data of the
cyclic $A_\infty$-algebra $A$.  The differential of $\h_A$ is
$b+uB+\hbar\Delta$ and the Lie bracket is $\{-,-\}$.  Define another
DGLA $\h_A^\triv$ using the same underlying $\Z/2\Z$-graded vector
space, but endowed with differential $b+uB$ and with trivial (zero)
Lie bracket.  In this section we shall prove the following formality
result.

\begin{Theorem}
  \label{thm:K}
  A chain-level splitting $S: L \ra L_+$ of the non-commutative Hodge
  filtration of $A$ (Definition~\ref{def:R-operator}) determines an
  $L_\infty$ quasi-isomorphism
  \[ \cK: \h_A \ra \h_A^{\triv}.\]
  An explicit formula for this quasi-isomorphism, in terms of $S$,
  will be described below in~(\ref{subsec:constr-K}).
\end{Theorem}

As an immediate corollary we obtain the following result that was used
to define CEI~(\ref{subsec:defi}).

\begin{Lemma}
  \label{lem:extension}
  The element $\beta^A\in \h_A^+$ of~(\ref{subsec:defi}) is a
  Maurer-Cartan element in $\h_A$ under the canonical inclusion of
  dg-vector spaces $\h_A^+\subset \h_A$.
\end{Lemma}

\begin{proof}
  From the explicit construction of $\cK$ (see below) it is clear
  that $\cK$ restricts to an $L_\infty$-morphism
  $\cK^+$ which fits in the commutative
  diagram below
  \[
    \begin{CD} \h_A @>\pi>> \h_A^+ @>>> 0 \\
      @V \cK VV @VV\cK^+ V @.\\
      \h_A^\triv @>\pi>> (\h_A^+)^\triv @>>> 0. 
    \end{CD}
  \]
  Along the bottom row, since the Maurer-Cartan equation is simply
  $(b+uB)x=0$ and the map $\pi$ is surjective at chain level, it is
  obvious that the Maurer-Cartan element $\cK^+(\beta^A)$ can be
  lifted to $\h_A^\triv$. The two vertical maps are both $L_\infty$
  quasi-isomorphisms, therefore there exists a Maurer-Cartan element
  $\alpha\in \h_A$ such that $\pi(\alpha)$ is gauge equivalent to
  $\beta^A$, i.e., there exists a degree one element $g\in \h_A^+$ for
  which $\exp(g)\cdot\pi(\alpha)=\beta^A$. Since $\pi$ is surjective,
  let $\widetilde{g}\in \h_A$ be a lift of $g$ in $\h_A$.  Then the
  Maurer-Cartan element $\exp(\widetilde{g})\cdot\alpha$ lifts the
  Maurer-cartan element $\beta^A$.  Modifying such an element in its
  scalar part does not affect the property of it being Maurer-Cartan
  (as $\bbK[1]\series{\hbar,\lambda}$ is central in $\h_A$), so we can
  choose the lift to have trivial scalar part.
\end{proof}

\subsection{Chain-level splittings and their inverses}

\begin{Definition}
  \label{def:R-operator}
  A chain-level splitting of the Hodge filtration is a map of
  complexes $S: L \ra L_+$ of the form $S= \id +S_1 u+ S_2 u^2
  +\cdots$ with each $S_n \in \End(L)$.  A chain-level splitting $S$
  is called symplectic if it satisfies 
  \[ \langle S \alpha, S \beta \rangle_{\sf hres} = \langle \alpha,
    \beta \rangle_{\sf Muk}\]
  for any $\alpha,\beta\in L$. It is called homologically symplectic
  if the above equation holds for $\alpha,\beta \in \ker b$, i.e. the
  induced map  $S: H_*(L) \ra H_*(L_+)$ satisfies the symplectic condition.
\end{Definition}

In explicit terms the symplectic condition means that the coefficients
of all higher powers of $u$ in the higher residue pairing above
vanish, i.e., 
\begin{align*}
  \sum_{j=0}^N (-1)^j \langle S_j \alpha, S_{N-j} \beta \rangle_{\sf
  Muk} =0, \;\forall N \geq 1, \; \forall \alpha,\beta \in L.
\end{align*}

A chain-level splitting $S: L \ra L_+$ can be extended by
$u$-linearity to an isomorphism of chain complexes
$\big(L\laurent{u}, b\big) \ra L^{\sf Tate}$ which we shall still
denote by $S$.  Both sides carry residue pairings, and the
symplectic condition above is equivalent to saying that $S$ is a
symplectomorphism.  Indeed,  for any $x,y\in L\laurent{u}$ we have
\begin{align*}
  \langle S x, S y \rangle_{\res}  & =  \mbox{ the
          coefficient of $u$ in } \langle S x, S y \rangle_{\hres} \\
    & =  \mbox{ the coefficient of $u$ in } \langle x, y
      \rangle_{\hres} \\
    & = \langle x, y \rangle_\res
\end{align*}

The following lemma proves that any splitting $s: H_*(L) \ra H_*(L_+)$
at homology level, in the sense of Definition~\ref{def:splitting},
lifts to a homologically symplectic chain-level splitting
$S:L \ra L_+$.

\begin{Lemma}
  \label{lem:existence-R}
  Let $s: H_*(L) \ra H_*(L_+)$ be a splitting of the Hodge
  filtration.  Then there exists a chain-level splitting $S: L \ra L_+$
  that lifts the map $s$. In particular, since $s$ is always assumed by
  definition to be symplectic, its lift $S$ will be homologically symplectic.
\end{Lemma}

\begin{proof}
  We shall construct a chain-level lift $S$ of $s$, order by order in
  the powers of $u$.  As a first step we need to find $S_1\in \End(L)$
  satisfying $[b,S_1]=-B$.  Since we work over a field there exist
  maps of complexes
  \[ i: H_*(L) \ra L \;\; \mbox{ and } \;\; p: L \ra H_*(L) \]
  such that $pi=\id$. This induces a quasi-isomorphism
  \[ I: \End^*(H_*(L)) \ra \End^*(L), \;\; \phi \mapsto i \phi p.\]
  The map $B$ induces the zero map on $H_*(L, b)$, because we assume
  that the Hodge-de Rham spectral sequence degenerates at ${}^1 E$.
  Therefore $B$ induces an exact operator on $\End^*(L)$, so we
  can choose an endomorphism $S_1'$ of $L$ so that $[b,S_1']=B$.
  Now we can modify $S_1'$ by an element in the image $\imag (I)$ to
  obtain an endomorphism $S_1$ so that $\id+uS_1$ lifts the desired
  splitting map $s$, modulo $(u^2)$.

  Now assume by induction that we have already found a lift
  $\id+uS_1+\cdots+u^{n-1} S_{n-1}$ of $s$, modulo $(u^n)$.  We are
  looking for $S_n\in \End^*(L)$ satisfying
  \[ [ b, S_n] = -BS_{n-1}.\]
  By construction we have
  \[ [b, BS_{n-1}]= - B [b,S_{n-1}] = BBS_{n-2} =0.\]
  Hence $BS_{n-1}$ induces a map in homology $\End^*(H_*(L))$. Observe
  that this map is exactly the $n$-th page map of the spectral
  sequence associated to the $u$-filtration on $L_+$, and hence again
  this map vanishes by the Hodge-to-de-Rham degeneration
  property.  Thus we can find $S_n'$ satisfying
  $[b,S_n']=-BS_{n-1}$ and we can modify it by elements in the image
  $\imag(I)$ to obtain the desired endomorphism $S_n$.
\end{proof}

Let $S$ be a chain-level splitting of the Hodge filtration, extended
by linearity to a quasi-isomorphism  
\[ S: \big( L\series{u}, b\big) \ra \big(L\series{u}, b+uB\big).\]
Since the $u^0$ component of $S$ is the identity we can write the
inverse $R$ of $S$ explicitly, in the form
\[ R= \id+R_1u+R_2u^2+\cdots,\]
where the $R_i$'s are endomorphisms of the graded vector space $L$.
More precisely, we enforce the sequence of identities
\[ \sum_{i+j=k} R_iS_j = \begin{cases} \id & \mbox{ if } k=0\\ 0 &
    \mbox{ if } k\geq 1
  \end{cases}. \]
These can be recursively solved to get formulas for the operators 
$R_j$'s in terms of the $S_i$'s. For example we have
\[ R_1=-S_1, \;\; R_2= -S_2+S_1S_1, \;\;\mbox{etc.} \]
This gives the desired inverse chain map
\[ R: \big(L\series{u}, b+uB\big) \ra \big( L\series{u}, b\big). \]
 Since $S$ and $R$ are chain maps, we have
\[ [b,S_n]=-BS_{n-1}, \;\; [b,R_n]=R_{n-1}B \;\; \mbox{for
    all}\;\;n\geq 1.\]

\subsection{A first order trivialization}
\label{subsec:firsttriv}
Recall that, by construction, the restriction of the operator
$\Delta: \h_A\ra \h_A$ to $\Sym^2 L_-$ is given by 
\[ \Omega(x,y) =\langle Bx_0, y_0\rangle_{\sf Muk}, \]
see~(\ref{subsec:dgla1}).  Having a chain-level splitting $S$ of the
Hodge filtration (along with its inverse $R$) allows us to
homotopy-trivialize $\Omega$, which in turn will allow us to
homotopy-trivialize $\Delta$.  This is the first step in constructing
the quasi-isomorphism $\cK$. 

Using $S$ and $R$ we can define an even map 
\[ H : L_- \otimes L_- \ra \bbK \]
which will be a homotopy trivialization of $\Omega$.  It is the
defined by takings its components
\[ H_{i,j}: u^{-i} L\otimes u^{-j} L \ra \bbK \]
to be, for $i,j \geq 0$, given by
\[H_{i,j}( u^{-i} x, u^{-j} y ) = \langle\, (-1)^j \sum_{l=0}^j S_l
  R_{i+j+1-l} \,x,\, y\,\rangle.\]
Equivalently, we can define the map $H$ more compactly by the formula
\[H(\alpha, \beta)=- \langle S\, \tau_{\geq 1}\, R (\alpha),
  \beta\rangle_\res\]
for $\alpha, \beta\in L_-$.  Here$\tau_{\geq 1}: L^\Tate \ra u \cdot
L_+$ is the projection onto the part with positive powers of $u$.

\begin{Proposition}
  \label{prop:homotopy}
  For any $\alpha, \beta\in L_-$ we have
  \[ H\big( (b+uB)\alpha,\beta\big) + (-1)^{|\alpha|} H\big( \alpha,
    (b+uB)\beta\big) = -\Omega(\alpha,\beta). \]
  In other words, the operator $H$ is a bounding homotopy of $\Omega$:
  $[b+uB, H] = \Omega$.
\end{Proposition} 

\begin{proof}
  The proof is a direct calculation using the explicit formula for $H$
  and properties of $S$ and $R$. 
\end{proof}

Having constructed $H$ we can now define a first order trivialization
of $\Delta$ as follows.  Symmetrize the homotopy operator $H$ to
obtain an operator $H^{\sym}: \Sym^2 L_- \ra \bbK$,
\[ H^{\sym}(x\cdot y) = \frac{1}{2} \left (H(x,y) + (-1)^{|x||y|}
    H(y,x)\right ).  \]

Since $\Omega$ is by definition symmetric~(\ref{subsec:circle})
we conclude that 
\[ [b+uB,H^{\sym}] = \Omega.\]
Extending $H^\sym$ as a second order differential operator to the full
symmetric algebra $\Sym (L_-)$ yields a new operator which we denote
by $\Delta^{H^\sym}$:
\[ \Delta^{H^{\sym}}: \Sym (L_-) \ra \Sym (L_-). \]
It is now obvious that $\Delta^{H^\sym}$ homotopy-trivializes
$\Delta$: both sides are second order differential operators, thus it
suffices to prove that they are equal on $\Sym^{\leq 2} L_-$, in which
case this is exactly the equation $ [b+uB,H^{\sym}] = \Omega$.

For the purposes of this section any chain-level splitting $S$ (not
necessarily homologically symplectic) works well. However, this is not
the case in the next section where the symmetric property of $H$ is
essential. In order to deal with this problem, we prove two additional
lemmas that will be used later.

\begin{Lemma}
  \label{lem:H-symmetric}
  If $S$ is a symplectic chain-level splitting, then the homotopy
  operator $H$ associated with a chain-level splitting $S$ is
  symmetric, i.e. $H=H^\sym$.
\end{Lemma}

\begin{proof} This is a direct calculation using the definitions.
\end{proof}

In general we do not know if a chain-level symplectic splitting exists
lifting a given splitting on homology. However, the following lemma
shows that even if we use a non-symplectic lift, still the homotopy
operator $H$ is always symmetric up to homotopy (if the splitting is
homologically symplectic).

\begin{Lemma}
\label{lem:delta}
Assume that $S$ is a homologically symplectic chain-level
splitting. Then there exists an odd operator $\delta: L_- \otimes L_-
\ra \mathbb{K}$ such that 
\[[ b+uB,\delta] = H- H^\sym.\]
\end{Lemma}

\begin{proof}
Denote by $\sigma: L_- \otimes L_- \ra L_- \otimes L_-$ the
transposition operator $x\otimes y \mapsto (-1)^{|x||y|}y\otimes
x$. Since $H^\sym= \frac{1}{2} (H+H\sigma)$, it suffices to show that
$D= H-H\sigma$ is $(b+uB)$-exact. First we show it is $b+uB$-closed:
\[ [b+uB,D]= [b+uB,H] - [b+uB,H]\sigma= \Omega -\Omega\sigma =0.\]
Thus, $D$ descends to a map in homology
$H_*(L_-)\otimes H_*(L_-) \ra \bbK$.  We shall prove that this map is
zero. There is an increasing filtration of $L_-$ defined by
\[ F_k L_- = \{ x \in L_- \mid x \mbox{ is of the form\;\;} x_{-k}
  u^{-k} + \cdots + x_0 u^0.\}\]
It induces an exhaustive filtration in homology
\[ H_*(L_-) = \lim H_*(F_k L_-).\]
We shall prove that $D$ is zero in homology by induction on $k=i+j$
in the tensor product filtration $H_*(F_i L_-)\otimes H_*(F_j
L_-)$.

To begin, consider the case $x=x_0 \in F_0 L_-$ and
$y=y_0 \in F_0L_-$. The $(b+uB)$-closedness of $x$ and $y$ gives
$bx=by=0$. We have
\[ D(x,y)=\langle R_1x,y\rangle_{\sf Muk} - \langle
  x,R_1y\rangle_{\sf Muk}\]
Since $R_1=-S_1$, the above is indeed zero by the homologically
symplectic condition, see Definition~\ref{def:R-operator}.
  
For the general case let $x=x_{-i} u^{-i} + \cdots$ and
$y=y_{-j} u^{-j}+\cdots$ be two elements in $L_-$ that are
$(b+uB)$-closed. In particular, we have $bx_{-i}=by_{-j}=0$.  We
rewrite $x$ and $y$ as
\[ x =\tau_{\leq 0} S( x_{-i} u^{-i}) + x',\;\;\; y = \tau_{\leq 0} S(
  y_{-j} u^{-j}) + y', \]
with $x'\in F_{i-1} L_-$, $y'\in F_{j-1} L_-$ and both
$(b+uB)$-closed. By induction, it suffices to prove that
\[ D\big( \tau_{\leq 0} S( x_{-i} u^{-i}), \tau_{\leq 0} S( y_{-j}
  u^{-j}) \big)=0. \]
We first compute
\begin{align*}
H\big( \tau_{\leq 0} S( x_{-i} u^{-i}), & \tau_{\leq 0} S( y_{-j}
  u^{-j}) \big)  = H\big( S( x_{-i} u^{-i})- \tau_{\geq 1}S( x_{-i}
                  u^{-i}) , \tau_{\leq 0} S( y_{-j} u^{-j}) \big)\\
  & = - \langle S\tau_{\geq 1} R\big( S( x_{-i} u^{-i})- \tau_{\geq
    1}S( x_{-i} u^{-i})\big), \tau_{\leq 0} S( y_{-j}  u^{-j})\rangle_\res.
\end{align*}
Observe that we have
\begin{align*}
  S\tau_{\geq 1} R S( x_{-i} u^{-i}) & = S\tau_{\geq 1}
                                       (x_{-i} u^{-i})=0,\\
  S\tau_{\geq 1} R \tau_{\geq 1} S( x_{-i} u^{-i})
                                     &= \tau_{\geq 1} S( x_{-i} u^{-i}).
\end{align*}
The earlier expression simplifies to
\[ H\big( \tau_{\leq 0} S( x_{-i} u^{-i}), \tau_{\leq 0} S( y_{-j}
  u^{-j}) \big)= \langle \tau_{\geq 1} S( x_{-i} u^{-i}), \tau_{\leq
    0} S( y_{-j} u^{-j})\rangle_\res.\]
Similarly, for the term $H\sigma$ we have
\[ H\sigma \big( \tau_{\leq 0} S( x_{-i} u^{-i}), \tau_{\leq 0} S(
  y_{-j} u^{-j}) \big)=- \langle \tau_{\leq 0} S( x_{-i} u^{-i}),
  \tau_{\geq 1} S( y_{-j} u^{-j})\rangle_\res.\]
Taking the difference yields
\begin{align*}
  D\big( \tau_{\leq 0} & S( x_{-i} u^{-i}), \tau_{\leq 0} S( y_{-j}
  u^{-j}) \big)  = \\
  & = \langle \tau_{\geq 1} S( x_{-i} u^{-i}),
  \tau_{\leq 0} S( y_{-j} u^{-j})\rangle_\res + \langle \tau_{\leq 0}
   S(x_{-i} u^{-i}), \tau_{\geq 1} S( y_{-j} u^{-j})\rangle_\res \\
  & =  \langle S( x_{-i} u^{-i}), S( y_{-j} u^{-j})\rangle_\res=
    \langle x_{-i} u^{-i}, y_{-j} u^{-j} \rangle_\res = 0.
\end{align*}
In the third equality we have used the homologically symplectic condition
which is valid since both $x_{-i}$ and $y_{-j}$ both $b$-closed.
\end{proof}

\subsection{\texorpdfstring{Construction of the
    $L_\infty$ morphism $\cK$}{Construction of the L-infinity morphism
  K}}
\label{subsec:constr-K}

The idea for proving Theorem~\ref{thm:K} is as follows. First we
construct a pseudo-isotopic family of DGLAs which interpolates between
$\h_A^\triv$ and $\h_A$. By definition, this is a DGLA structure on
the tensor product $\h_A\otimes\Omega_{[0,1]}^\bullet$ which
specializes to $\h_A^\triv$ at $t=0$ and to $\h_A$ at $t=1$.  Then we
apply an integration procedure to obtain the desired $L_\infty$
isomorphism $\cK: \h_A \ra \h_A^\triv$.

Using the homotopy operator $H^\sym$ defined above we define a DGLA
structure on $\h_A\otimes\Omega_{[0,1]}^\bullet$ as follows. We take
the differential to be
\[ b+uB+\hbar t\Delta + d_{DR} + \hbar \Delta^{H^\sym} dt,\]
with $\Delta^{H^\sym}$ defined as in~(\ref{subsec:firsttriv}), and we
take the Lie bracket defined by
\[ t\cdot\{-,-\} + \{-,-\}^{H^\sym} dt,\]
with $\{-,-\}^{H^\sym}$ defined as the failure of $\Delta^{H^\sym}$
to be a first order differential operator on $\Sym\h_A$.  This is
similar to how $\{-,-\}$ was defined in~(\ref{subsec:dgla1}).

\begin{Lemma}
  With the differential and Lie bracket defined above
  $\h_A\otimes \Omega_{[0,1]}^\bullet$ is a DGLA.
\end{Lemma}

\begin{proof} This is a direct check using
Proposition~\ref{prop:homotopy}.
\end{proof}

As described in~\cite{Fuk} one can integrate this family of DGLAs to
obtain explicit $L_\infty$ isomorphisms
\[ \cK(t): \h_A(t) \ra \h_A(0)=\h_A^\triv.\]
In particular, for $t=1$ we obtain the desired isomorphism
$\cK=\cK(1)$.  The formulas for $K$ in~\cite{Fuk} requires the
computation of integrals on certain configuration spaces. We shall not
go into the details of these calculations but instead, in the
remainder of this section, we shall write down the final result which
turns out to be quite neat.  The $L_\infty$ morphism $\cK$ is given by
a Feynman sum over a certain class of graphs, with the vertices decorated
by the tensors $\beta^A_{g,n}$ from~(\ref{subsec:defi}), and with the
edges decorated by the operator $H^\sym$ defined earlier in terms of
the splitting $s$!

Before we give the exact formula for $\cK$ we need to introduce some
notations for graphs, following~\cite[Section 2]{GetKap}. By a labeled
graph we shall mean a graph $G$ (possibly with leaves) endowed with a
genus labeling function $g: V_G\ra \bbZ_{\geq 0}$ on the set of its
vertices $V_G$. The genus of a labeled graph is defined to be
\[ g(G) = \sum_{v\in V_G} g(v)+{\sf rank\,} H_1(G). \]
We will use the following notations for a labeled graph $G$:
\begin{itemize}
\item[--] $V_G$ denotes the set of vertices of $G$;
\item[--] $L_G$ denotes the set of leaves of $G$;
\item[--] The valency of a vertex $v\in V_G$ is denoted by $n(v)$;
\item[--] The graph $G$ is called stable if $2g(v)-2+n(v)>0$ for 
every vertex $v$.
\item[--] If $G$ has $m$ vertices, a marking of $G$ is a bijection
  \[ f: \{1,\cdots,m\} \ra V_G. \] An isomorphism between two marked
and labeled graphs is an isomorphism of the underlying labeled graphs
that also preserves the markings.
\end{itemize}

We will denote various isomorphism classes of graphs as follows:
\begin{itemize}
\item[--] $\Gamma(g,n)$ is the set of isomorphism classes of
graphs of genus $g$ with $n$ leaves;
\item[--] Using double brackets as in $\Gamma\laurent{g,n}$ requires
further that the graphs be stable;
\item[--] A subscript $m$ as in $\Gamma(g,n)_m$ or
$\Gamma\laurent{g,n}_m$ indicates that the graphs in discussion have
$m$ vertices;
\item[--] adding a tilde as in $\widetilde{\Gamma(g,n)}_m$ or
$\widetilde{\Gamma\laurent{g,n}}_m$ means that we consider
{\em marked} graphs.
\end{itemize}

We are now ready to write down the formula for the desired $L_\infty$
morphism $\cK: \h_A \ra \h_A^\triv$. Recall that an
$L_\infty$ morphism is a collection of even linear maps
\[ \cK_m: \Sym^m (\h_A[1]) \ra \h_A^\triv[1] \]
defined for each $m\geq 1$. Since $\h_A =
\Sym(L_-)[1]\series{\hbar,\lambda}$ already includes a shift 
by one, the parity of the elements in $\h_A[1]$ is the same as in
$\Sym L_-$.

The $m$-th component $\cK_m$ of $\cK$ will then be the
$\lambda$-linear extension of the map
\[ \cK_m = \sum_{g,n} \sum_{ (G,f)\in \widetilde{\Gamma(g,n)}_m}
  \frac{1}{|\Aut(G,f)|}\cdot \cK_{(G,f)}, \]
with the map $\cK_{(G,f)}$ described below.

Let $(G,f)\in \widetilde{\Gamma(g,n)}_m$ be a marked graph.  Set
$g_i = g(f(i))$, and $n_i = n(f(i))$.  We proceed to describe the
construction of the map
\[ \cK_{(G,f)}:\bigotimes_{i=1}^m \left (\Sym^{n_i} L_-\cdot
    \hbar^{g_i} \right) \ra \Sym^n L_-\cdot \hbar^g.\]
For elements $\gamma_i \in \Sym^{n_i}L_-, \; (i=1,\ldots,m)$ the
expression
$\cK_{(G,f)} (\gamma_1\hbar^{g_1},\ldots, \gamma_m\hbar^{g_m})$ is
computed by the following Feynman-type procedure:
\begin{enumerate}
\item For an element $\gamma = x_1x_2\cdots x_j\in \Sym^j L_-$ let
$\widetilde{\gamma} \in \left (L_-\right )^{\otimes j}$ be its
desymme\-tri\-za\-tion,
\[ \widetilde{\gamma} =\sum_{\sigma\in \Sigma_j} \epsilon \cdot
x_{\sigma(1)}\otimes\cdots\otimes x_{\sigma(j)} \in L_-^{\otimes j}.\]
Here $\epsilon$ is the Koszul sign for permuting the elements
$x_1,\ldots, x_i$ (with respect to the degree in $L_-$). Then 
decorate the half-edges adjacent to each vertex $v_i$ by
$\widetilde{\gamma}_i$. Tensoring together the results over all the
vertices yields a tensor of the form
\[ \bigotimes_{i=1}^m \widetilde{\gamma_{i}}\in L_-^{\otimes (\sum
    n_i)}.  \]
The order in which these elements are tensored is the one
given by the marking.

\item For each internal edge $e$ of $G$ contract the corresponding
components of the above tensor using the even symmetric bilinear form
  \[ H^{\sym}: L_-^{\otimes 2} \ra \bbK. \] When applying the
contraction we always permute the tensors to bring the two terms
corresponding to the two half-edges to the front, and then apply the
contraction map. The ordering of the set $E_G$ does not matter since
the operator $H^{\sym}$ is even; also the ordering of the two
half-edges of each edge does not matter since the operator $H^{\sym}$
is (graded) symmetric.
\item Read off the remaining tensor components (corresponding to the
leaves of the graph) in any order, and regard the result as the
element $\gamma$ in $\Sym^n L_-$ via the canonical projection map
$L_-^{\otimes n} \ra \Sym^n L_-$.
\end{enumerate}

To illustrate the construction, we write down the map $\cK_1$. The
following figures illustrate a few graphs that contribute in this
case. Their common feature is that they have only one vertex in which
case the marking $f$ is unique.
\begin{center} \includegraphics[scale=.5]{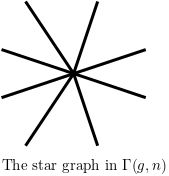}\hspace{1cm}
\includegraphics[scale=.5]{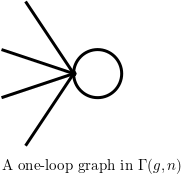}\hspace{1cm}
\includegraphics[scale=.5]{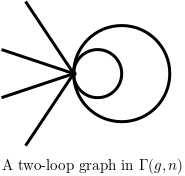}
\end{center} The contribution of the first type of graphs (the star
graphs) is the identity map
\[ \h[1]= \Sym L_-\series{\hbar, \lambda} \ra \h^\triv = \Sym
L_-\series{\hbar, \lambda}.  \] To see this, note that
de-symmetrization followed by projection produces a factor of $n!$,
which is canceled by the size of the automorphism group of the star
graph.  We also note that it is necessary to use the set $\Gamma(g,n)$
of all labeled graphs, as opposed to using just stable graphs. For
instance, the $(g,n)=(0,0)$ star graph (just a vertex with a genus
label of $0$) contributes to the identity map $\Sym^0 L_- =\bbK \ra
\Sym^0 L_- =\bbK$.

In a similar way one checks that the middle type of graph acts
precisely by the operator $\hbar\Delta^{H^{\sym}}$, while the third
type of graph contributes
$\frac{1}{2}(\hbar\Delta^{H^{\sym}})^2$. Here the extra factor $1/2$
comes from the automorphism that exchanges the two loops. In general
we have
\[ \cK_1= \sum_{l\geq 0} \frac{1}{l!} (\hbar\Delta^{H^{\sym}})^l=
e^{\hbar\Delta^{H^{\sym}}}.\] Indeed, the first $L_\infty$ morphism
identity is $(b+uB)\cK_1=\cK_1(b+uB+\hbar \Delta)$. It can be verified
directly using the formula above of $\cK_1$ and
Proposition~\ref{prop:homotopy}.

\subsection{\texorpdfstring{Proof of the fact that $\cK$ is an
    $L_\infty$ morphism}{Proof of the fact that K is an L-infinity morphism}}
The remainder of this section is devoted to proving
Theorem~\ref{thm:K}.  Replacing the contraction operator $H^{\sym}$ by
$t H^{\sym}$ in the construction of $\cK$ yields a family of
pre-$L_\infty$ morphisms
\[ \cK(t): \h_A(t) \ra \h_A(0)=\h_A^\triv.\]
On the other hand, the integration procedure described in~\cite{Fuk}
produces a family of $L_\infty$ morphisms denoted by
\[ \cK'(t): \h_A(t) \ra \h_A(0)=\h_A^\triv.\]
It will suffice to prove $\cK(t)=\cK'(t)$.  To accomplish
this it is most convenient to work with the bar construction of
DGLAs. Indeed, recall that we have defined the operator
\[ \Delta^{H^\sym}: \h_A[1] \ra \h_A[1]\]
which is a first order trivialization of $\Delta$,
i.e. $[b+uB,\Delta^{H^\sym}]=\Delta$.  Taking the commutator with the 
product of the algebra $\h_A[1]=\Sym L_-\series{\hbar,\lambda]}$ yields the
binary operation 
\[\{x,y\}^{H^{\sym}} := \Delta^{H^\sym}(x\cdot y) - \Delta^{H^\sym}(x)
  \cdot y - x \cdot \Delta^{H^\sym}(y).\]
Using the bar construction of DGLAs we extend the even maps
$\Delta^{H^\sym}$ and $\{-,-\}^{H^{\sym}}$ to coderivations
\[  \nu_1, \nu_2 : \Sym (\h_A [1])\ra \Sym (\h_A [1]). \]
Then, by the construction in~\cite[Section 9]{Fuk}, the family
$\cK'(t)$ is characterized by the differential equation
\[ \frac{d \cK'_m(t)}{dt} =
  \hbar \cK'(t)_m \nu_1 + \cK'_{m-1}(t) \nu_2, \]
with the initial condition that $\cK'(0)=\id$.

We proceed to prove that our family $\cK(t)$ also satisfies
the above differential equation and initial condition, which will show
that the two morphisms are equal.  Recall that by definition we have
\[\cK_m(t) = \sum_{g,n} \sum_{ (G,f)\in
\widetilde{\Gamma(g,n)}_m} \frac{1}{|\Aut(G,f)|}\cdot
\cK_{(G,f)}(t)\cdot \hbar^{g-\sum_{i=1}^m g(f(i))}. \]
The dependence of $\cK^S_{(G,f)}(t)$ on $t$ is simple: we
assign $t H^{\sym}$ to all edges of $G$. Thus, we observe that when
the differentiation acts on loop edges it yields the term
$\hbar \cK'(t)_m \nu_1$, while differentiating at non-loop
edges yields the other term $\cK'_{m-1}(t) \nu_2$.

In both cases the relevant coefficients involving automorphism group
of graphs are standard in the theory of graph sums. For example we
can use the techniques of integration over groupoids developed
in~\cite{Fio} to argue their validity as follows. Let us denote by
$\cD_m$ the groupoid whose objects are marked labeled graphs $(G,f)$
with $m$ vertices, and morphisms are graph isomorphisms preserving
marking and labelings.  Let $\cD'_m$ denote the groupoid whose objects
are triples $(G,f,e)$ with $e$ a non-loop edge of $G$. We have the
following maps of groupoids
\[ \cD_m \stackrel{{\sf F}}{\longleftarrow} \cD'_m
  \stackrel{{\sf C}}{\longrightarrow} \cD_{m-1}.\]
The functor ${\sf F}$ is by forgetting the edge $e$, and the functor
${\sf C}$ is
by contracting the edge $e$. The new vertex obtained from contracting
$e$ in the contracted graph $G/e$ is marked by $1$, while the
remaining vertices of $G/e$ are marked according to its ordering
defined by $f$. As in~\cite{Fio} we may consider the assignment
\[ (G,f)\mapsto \cK_{(G,f)}(t)\cdot \hbar^{g-\sum_{i=1}^m
    g(f(i))} \]
as a vector-valued function $\cK_m(t)$ on the groupoid
$\cD_m$, and the coefficient $\frac{1}{|\Aut(G,f)|}$ as a measure
$d\mu$ on the set of isomorphism classes of objects. In this way we
can rewrite the summation as
\[\sum_{g,n}\sum_{ (G,f)\in \widetilde{\Gamma(g,n)}_m}
\frac{1}{|\Aut(G,f)|}\cdot \cK_{(G,f)}(t)\cdot
\hbar^{g-\sum_{i=1}^m g(f(i))} = \int_{\cD_m} \cK_m(t) d\mu\]
The $t$-dependence is a $t$ factor for every edge of $G$, which
implies that we simply have a factor $t^{|E_G|}$, and therefore we get
\[   \frac{d}{dt}\int_{\cD_m}\cK_m(t) d\mu = \int_{\cD_m}
  |E_G|t^{|E_G|-1}\cK_m d\mu. \]
Using the basic properties of groupoid integrals we
can also simplify the right hand side of the differential equation
satisfied by $\cK_m'(t)$ as
\[ 
  \int_{\cD_{m-1}} \cK_{m-1}(t)\nu_2 d\mu
  =\int_{\cD'_m} {\sf C}^*\big( \cK_{m-1}(t)\nu_2 \big) d\mu
  = \int_{\cD_{m}} {\sf F}_*{\sf C}^*\big( \cK_{m-1}(t)\nu_2 \big) d\mu,
\]
where the first and second equalities follow from~\cite[Equation
6.1]{Fio} and~\cite[Proposition 3.2]{Fio}, respectively.

The vector-valued function
${\sf F}_*{\sf C}^*\big( \cK_{m-1}(t)\nu_2 \big)$ is computed as
\begin{align*}
  {\sf F}_*{\sf C}^*\big( \cK_{m-1}(t) & \nu_2 \big)(G,f)  = \sum_{e\in
        E_G^{{\sf non-loop}}} C^*\big(\cK_{m-1}(t)\nu_2 \big)(G,f,e)  \\
        & = \sum_{e\in E_G^{{\sf  non-loop}}} t^{|E_G|- 1} \cK_m(G,f)
      = |E_G^{{\sf non-loop}}|\cdot t^{|E_G|-1} \cK_m(G,f).
\end{align*}
The case when $e$ is a loop edge is similar, and it contributes in the
integrand by
\[ |E_G^{{\sf loop}}|\cdot t^{|E_G|-1}\cK_m(G,f). \]
Putting these calculations together we see that the integrand gives
$|E_G|\cdot t^{|E_G|-1} \cK_m(G,f)$, which matches exactly
with the differential equation satisfied by $\cK_m'(t)$. This finishes
the proof of Theorem~\ref{thm:K}.

\section{\texorpdfstring{The Koszul trivializing morphism $\hcK$}{The
    Koszul trivializing morphism K}}
\label{sec:trivial-2}
In~(\ref{subsec:dgla-2}) we defined a Koszul-type resolution $\hh_A$
of $\h_A^+$, given by
\[ \hh_A= \bigoplus_{k\geq 1, l\geq 0} \Hom^{\sf c} \left
    (\Sym^k ( L_+[1]), \Sym^l(L_-)\right)\series{\hbar, \lambda}.\]
In order to simplify notation we shall write $L_{k,l}$ for the space
\[ L_{k,l} = \Hom^{\sf c} \left (\Sym^k(L_+[1]), \Sym^l(L_-)
  \right). \]
Consider the trivialization $\hh_A^\triv$ of $\hh_A$, which is the
DGLA whose underlying $\Z/2\Z$-graded vector space is the same as that
of $\hh_A$, but with differential $b+uB+\iota$ and with zero Lie
bracket.  The purpose this section is to construct an $L_\infty$
quasi-isomorphism
\[ \hcK: \hh_A \ra \hh_A^{\triv}\]
from the data of a chain-level splitting $S: L \ra L_+$ of the Hodge
filtration on the cyclic homology of $A$.

An important technical difference from the previous section is that we
will need to assume that $S$ is symplectic, or at least homologically
symplectic (see Definition~\ref{def:R-operator}).  The construction of
$\hcK$ will be similar to the one in the previous section in the sense
that the component maps $ \hcK_m \; (m\geq 1)$ will be written as
summations over certain graphs, the partially directed graphs that
we introduce below. 

\subsection{More first order trivializations}
\label{subsec:moretriv}
We need to construct two more bounding homotopies, in a way similar to
the construction of the homotopy $H$ in
Proposition~\ref{prop:homotopy}, but for the operators $\Theta$ and
$\delta$.

Recall that in~(\ref{subsec:dgla-2}) we introduced the even operator
$\Theta: L_- \ra L_+[1]$ given by
\[ \Theta(\alpha) = B(\alpha_0)\]
for $\alpha = \alpha_0 + \alpha_{-1}u^{-1} + \cdots$.  Let
$S: L \ra L_+$ be a chain-level splitting.  Define the odd map
$F: L_- \ra L_+[1]$ by
\[ F(\beta) = -u^{-1}\cdot S\, \tau_{\geq 1}\, R (\beta).\]
(This is well defined because of the $\tau_{\geq 1}$ operator in the middle.)
Explicitly, for $i\geq 0$, $j\geq 0$, the component $F_{i,j}$ of $F$ that
maps $L\cdot u^{-i}$ to $L \cdot u^{j}$ is given by
  \[ F_{i,j}(x\cdot u^{-i}) = -\sum_{l=0}^j S_lR_{i+j+1-l}\, x \cdot
u^{j}.\]

\begin{Proposition}
\label{prop:homotopy2}
The map $F: L_-\ra L_+[1]$ is a bounding homotopy of $\Theta$:
\[[\eth, F]=-(b+uB) F + F(b+uB)= \Theta.\]
(Here the differential of $L_+[1]$ is $-(b+uB)$ due to the shift.)

Moreover, the map $F$ is compatible with the homotopy operator $H$
constructed in Proposition~\ref{prop:homotopy}, in the sense that for
$\alpha,\beta\in L_-$ we have
\[ \iota(\alpha)(F\beta) = H(\beta,\alpha), \]
where $\iota: L_- \ra \Hom^{\sf c} (L_+[1] ,\bbK)$ is given by
\[ \iota(\alpha)(\beta)=(-1)^{|\alpha|'} \langle \alpha, u\cdot \beta
  \rangle_{\sf res} \]
as in~(\ref{subsec:dgla-2}).
\end{Proposition}

\begin{proof} The proof is by direct computation.
\end{proof}

We use the homotopy operator $F$ to define an even operator
\[ \{-,-\}_\hbar^F : \Sym^2 (\hh_A[1]) \ra \hh_A[1]\]
using a definition similar to that of $\{-,-\}_\hbar$
in~(\ref{subsec:dgla-2}).  More precisely, in the definition of the
twisted composition map $\circ_r$, instead of using $r$ copies of
$\Theta$, we replace one $\Theta$ with an $F$.  This yields
\begin{align*}
(\Psi \circ^F_r \Phi) ( \beta_1\cdots \beta_{k_1+k_2-r})
& = \sum_{P,Q}\sum_{I,J} \epsilon_{I,J}\epsilon_{P,Q} (-1)^\star
\Phi(\beta_I)_P\cdot \Psi\big( F^{[r]} ( \Phi(\beta_I)_Q)\otimes
                                                         \beta_J\big)\\
F^{[r]} &= \sum_{j=1}^r \Theta^{j-1} \otimes F \otimes\Theta^{r-j}
\end{align*}
Here $\star= |\Psi||\Phi(\beta_I)_P|+ |\Psi|+|\Phi(\beta_I)_P|$. Its
sign diagram is depicted as
\[ \includegraphics[scale=.5]{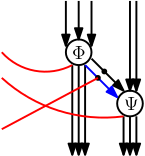}.\]
We use the convention that the blue directed edge represents the
homotopy operator $F$.  Since $F: L_- \ra L_+[1]$ is odd there is an
extra red line associated to it. This explains the sign difference
between $\circ^F_r$ and $\circ_r$.  Symmetrizing $\circ^F_r$ and
summing over $r$ as before we obtain the desired map
\[\{ \Psi, \Phi\}_\hbar^F = \sum_{r\geq 1}\{ \Psi,\Phi
    \}_r \cdot \hbar^{r-1}
  \]
where 
\[  \{ \Psi,\Phi \}^F_r = (-1)^{|\Psi|}
\big(\Psi\circ^F_r \Phi - (-1)^{|\Psi||\Phi|} \Phi \circ^F_r
\Psi\big). \]

If the splitting $S$ is only homologically symplectic and not
symplectic, we shall need to define a map
\[ \{-,-\}_\hbar^\delta : \Sym^2 (\hh_A[1]) \ra \hh_A[1]\]
that will fix the fact that $H$ is only homotopy-symmetric.  The
definition of the map $\{-,-\}_\hbar^\delta$ is exactly the same as
that of $\{-,-\}_\hbar^F$, but replacing in the above construction $F$
with the operator $\delta$ of Lemma~\ref{lem:delta}. 

\subsection{Partially directed graphs}
\label{subsec:pdgraphs}
Before we write down the formula for the quasi-isomorphism $\hcK$ we
need to introduce a new type of graphs which we will call partially
directed graphs.  The components of $\hcK$ will be written as sums
over these graphs.

\begin{Definition}
\label{defi:dir-graph}
A partially directed graph of type $(g,k,l)$ is given by a quadruple
$\GG = (G,L_G^{\sf in}\coprod L_G^{\sf out},E^\direct,T)$ consisting of the
following data:
\begin{itemize}
\item[--] A labeled graph $G$ of type $(g,k+l)$.
\item[--] A decomposition
\[ L_G = L_G^{\sf in}\coprod L_G^{\sf out}\]
of the set of leaves
$L_G$ such that $|L_G^{\sf in}|=k$ and $|L_G^{\sf out}|=l$. Leaves in
$L_G^{\sf in}$ will be called incoming, while leaves in $L_G^{\sf
  out}$ will be called outgoing.
\item[--] A subset $E^{\direct}\subset E_G$ of edges of $G$ whose
elements are called directed edges, and a direction is chosen on
them. Edges in $E_G-E^{\direct}$ are called undirected.
\item[--] A spanning tree $T\subset E_G$ of the graph $G$.
\end{itemize}

We require the following properties to hold:
\begin{itemize}
\item[--] There is no directed loop.
\item[--] Each vertex has at least one incoming half-edge.
\end{itemize}
  
A partially directed graph is called stable if the underlying
labeled graph is.
\end{Definition}

An isomorphism $\phi: \GG_1 \ra \GG_2$ of partially directed graphs
consists of an isomorphism of underlying labeled graphs $G_1$ and
$G_2$ which preserves the directions of leaves and edges, and such that
$\phi(T_1) = T_2$.  Denote by $\Gamma(g,k,l)$ the set of isomorphism
classes of partially directed graphs of type $(g,k,l)$ and by
$\Gamma\laurent{g,k,l}$ the subset of those that are stable.  A marking
of a partially directed graph $\GG$ is a bijection
$f: \{1,\cdots,m\} \ra V_G$.  Denote by $\widetilde{\Gamma(g,k,l)}_m$
the set of isomorphism classes of marked partially directed graphs
with $m$ vertices, and by $\widetilde{\Gamma\laurent{g,k,l}}_m$ the
corresponding set of marked stable partially directed graphs.

\subsection{The weight of a partially directed graph}
\label{subsec:weight}
We need to associate a rational number $\wt(\GG)\in \bbQ$ to a
partially directed graph $\GG$.  This will be used in the explicit formula for
$\hcK:\hh_A\ra\hh_A^\triv$.

Let $e\in T$ be an edge from the spanning tree $T$. We call it
contractible if the graph obtained from $\GG$ by contracting all the
directed edges connecting the two end points of $e$ is again a valid
partially directed graph. Denote the resulting partially directed
graph by $\GG/e$. Denote by $T^{\sf contr}\subset T$ the set of
contractible edges, and by $E_G^{\sf non-loop}$ the set of non-loop
edges of $G$.

We inductively define $\wt(\GG)$ as follows:
\begin{itemize}
\item If $|T|=\emptyset$ (equivalently, if $\GG$ only has one vertex)
  then $\wt(\GG)=1$. 
\item In general we set
\[ \wt(\GG)= \frac{1}{|E_G^{\sf non-loop}|}\sum_{e\in T^{\sf contr}} \wt(\GG/e).\]
\end{itemize}

An immediate corollary of this definition is that $\wt(\GG)=1$ if $g(\GG)=0$.
Indeed, in the genus zero case, since all edges in $T$ are
contractible and $T=E_G^{\sf non-loop}$, it is easy to prove that
$\wt(\GG)=1$ by induction. 

\subsection{\texorpdfstring{The
    $L_\infty$ quasi-isomorphism $\hcK$}{The L-infinity
    quasi-isomorphism hat-K}}  
\label{subsec:k-hat-symplectic}

As in~(\ref{subsec:constr-K}) we shall construct the $L_\infty$
quasi-isomorphism 
\[ \hcK: \hh_A\ra\hh_A^\triv \]
by integrating a DGLA structure on the tensor product
$\hh_A\otimes \Omega_{[0,1]}^\bullet$.  On this tensor product the
differential and Lie bracket are given by
\[  b+uB+\hbar t \Delta + \iota + d_{DR} +  \hbar \Delta^{H}dt \]
and
\[ \sum_{r\geq 1} t^r\{-,-\}_r\hbar^{r-1} + \sum_{r\geq 1}
t^{r-1}\{-,-\}^F_r\hbar^{r-1}dt+\sum_{r\geq 1}
                    t^{r-1}\{-,-\}^\delta_r\hbar^{r-1}dt, \]
respectively.

The result of this integration procedure (following Fukaya~\cite{Fuk})
gives an $L_\infty$ morphism $\hh_A \ra \hh_A^\triv$, written in terms
of a graph sum similar to the one we used to define the morphism $\cK$
in~(\ref{subsec:constr-K}).  However, the weights used in Fukaya's
formula are not directly computable.

Our approach is slightly different.  We will define a pre-$L_\infty$
morphism $\hcK: \hh_A \ra \hh_A^\triv$ by a sum over graphs with
explicitly computable weights, and we show that it agrees with the
morphism defined by Fukaya.  From this it will follow that the $\hcK$
we have defined is indeed an $L_\infty$ morphism.

For each $m\geq 1$ we need to construct a map
\[ \hcK_m: \Sym^m (\hh_A[1]) \ra
  \hh_A^\triv[1].\]
As before, the map $ \hcK_m$ is defined as a sum over graphs.  This
time, however, the sum is over partially directed graphs.  Let
$(\GG,f)\in 
\widetilde{\Gamma(g,k,l)}_m$ be a marked partially directed
graph with $m$ vertices.  We proceed to define a $\bbK$-linear map
\[ \hcK_{(\GG,f)} : \bigotimes_{i=1}^m L_{k_i,l_i}[1] \cdot
  \hbar^{g_i} \ra L_{k,l}[1] \cdot \hbar^g.\]
This map is only non-zero when the combinatorial types match at each
vertex, i.e., for each $1\leq i \leq m$, we have $k_i= k_{f(i)}$,
$l_i = l_{f(i)}$, and $g_i=g_{f(i)}$. When this is the case, the
result
$ \hcK_{(\GG,f)}(\gamma_1\cdot
\hbar^{g_1},\ldots,\gamma_m\cdot\hbar^{g_m})$ is defined as follows:
\begin{itemize}
\item The tensor $\widetilde{\gamma_i}$ is assigned to the vertex $f(i)$. Here,
for a map $\gamma_i\in L_{k_i,l_i}$, its de-symmetrization
\[ \widetilde{\gamma_i} \in \Hom^{\sf c} \left (
    (L_+[1])^{\otimes k_i}, L_-^{\otimes l_i}\right ) \]
is the composition
\[ (L_+[1])^{\otimes k_i} \ra \Sym^{k_i}(L_+[1])
\stackrel{\gamma}{\lra} \Sym^{l_i}(L_-) \ra L_-^{\otimes l_i}. \]
\item The homotopy operator $F$ in
  Proposition~\ref{prop:homotopy2} is assigned to each directed edge in the
  spanning tree $T$; the operator $\Theta$
  in~(\ref{subsec:dgla-2}) is associated to the other directed edges.

\item The operator $\delta$ of Lemma~\ref{lem:delta} labels each
  undirected edge in the spanning tree $T$; the homotopy
  operator $H^{\sym}$ in Proposition~\ref{prop:homotopy} is assigned
  to the other undirected edges.
  
\item The result $\hcK_{(\GG,f)}(\gamma_1\cdot
  \hbar^{g_1},\ldots,\gamma_m\cdot\hbar^{g_m})$ is obtained by
  composition along $\GG$, with signs given by sign diagrams using
  formulas like in the definition of $\{-,-\}_\hbar^F$.
\end{itemize}

Assembling all these maps together we obtain a graph sum
 \[ \hcK_m (\gamma_1\cdot \hbar^{g_1},\ldots,\gamma_m\cdot
\hbar^{g_m})= \sum_{(\GG,f)} \frac{\wt(\GG)}{\Aut(\GG,f)}
\hcK_{(\GG,f)} (\gamma_1\cdot
\hbar^{g_1},\ldots,\gamma_m\cdot\hbar^{g_m}).\]
The summation above is over all marked partially directed graphs with
$m$ vertices, for various $(g,k,l)$.

\begin{Theorem}
\label{thm:Khat-delta}
Let $S$ be a homologically symplectic chain-level splitting. Then the
map $\hcK$ defined in terms of $S$ is a quasi-isomorphism of DGLAs
$\hcK:\hh_A \ra \hh_A^\triv$.  Moreover, we
have a homotopy commutative diagram of DGLAs
\[\begin{CD}
    \h^+_A @>\epsilon >> \hh_A \\
    @V \cK VV @VV \hcK V\\
    \h_A^{+,\triv} @>\epsilon>> \hh_A^\triv.
\end{CD}\]
\end{Theorem}

\begin{proof}
The proof is similar to that of Theorem~\ref{thm:K}.  To set it up,
define a family of pre-$L_\infty$ morphisms
  $\hcK(t): \hh_A(t) \ra \hh_A^\triv$ by setting
\[ 
  \hcK_m(t) (\gamma_1\cdot
  \hbar^{g_1},\ldots,\gamma_m\cdot\hbar^{g_m}) = \sum_{(\GG,f)}
  \frac{t^{|E_G|} \wt(\GG)}{\Aut(\GG,f)} \hcK_{(\GG,f)} (\gamma_1\cdot
  \hbar^{g_1},\ldots,\gamma_m\cdot\hbar^{g_m}), \]
with the weight defined in~(\ref{subsec:weight}).  Observe that by
definition we have $\hcK_m(1)=\hcK_m$, and that $\hcK(0)$ is the
identity morphism of $\hh_A^\triv$.

There are three operators used in the earlier definition of the DGLA structure
on
\[ \hh_A \otimes \Omega_{[0,1]}^\bullet, \]
namely $\Delta^{H^\sym}$,
$\sum_{r\geq1} t^{r-1}\cdot \{-,-\}^F_r \cdot \hbar^{r-1}$, and
$\sum_{r\geq 1} t^{r-1}\cdot \{-,-\}^\delta_r \cdot \hbar^{r-1}$.
Using the bar construction for DGLAs we extend these
operators as coderivations, to obtain
\[ \widehat{\nu}_1(t),\, \widehat{\nu}_2^F(t),\, \widehat{\nu}_3^\delta(t)
  : \Sym (\hh_A[1]) \ra \Sym  (\hh_A[1]). \]
The integration procedure described in~\cite{Fuk}
produces a family of $L_\infty$ morphisms uniquely determined by a
differential equation. We shall prove that the pre-$L_\infty$ morphism
$\hcK(t)$ defined above satisfies the same
differential equation, which will prove that $\hcK(t)$ is
indeed an $L_\infty$ morphism.

This differential equation is given by
\[ \frac{d}{dt} \hcK_m(t) = \hbar \hcK_{m}(t)\widehat{\nu}_1 +
\hcK_{m-1}(t)\widehat{\nu}^F_2(t)+\hcK_{m-1}(t)\widehat{\nu}^\delta_2(t). \]
To prove that the $\hcK_m$'s satisfy the above identity, observe that
for a marked partially directed graph $(\GG,f)$, the $t$-dependence is
simply by $t^{|E_G|}$. Hence the left hand side equals to
\[ \int_{\mathbb{D}_m} |E_G|\cdot  t^{|E_G|-1}\cdot \wt\,\cdot\,
  \hcK_m \, d\mu, \]
where $\mathbb{D}_m$ denotes the groupoid of marked
partially directed graphs with $m$ vertices, and as in the proof of
Theorem~\ref{thm:K} we make use of groupoid integrals developed
in~\cite{Fio}.

Thus our goal is to show that the right hand side of the differential
equation is equal to the above groupoid integral.  We first analyze
the first term $\hbar \hcK_{m}(t)\widehat{\nu}_1$ on the right hand
side of the differential equation. According to the inductive
definition of the weights, we observe that $\wt(\GG)=\wt(\GG/e)$ if
$e$ is a loop edge. With this fact in mind, arguing the same way as in
the proof of Theorem~\ref{thm:K} we obtain
\[ \hbar \hcK_{m}(t)\widehat{\nu}_1= \int_{\mathbb{D}_m}
|E_G^{{\sf loop}}|\cdot t^{|E_G|-1}\cdot\, \wt\,\cdot\,
\hcK_m \, d\mu. \]
In the case of non-loop edges, consider the following maps between
groupoids
\[ \mathbb{D}_m \stackrel{{\sf F}}{\longleftarrow} \mathbb{D}'_m
\stackrel{{\sf C}}{\longrightarrow} \mathbb{D}_{m-1},\] where the
middle groupoid $\mathbb{D}'_m$ has objects consisting of triples
$(\GG,f,e)$ with $e\in T$ a contractible edge in the spanning tree of
$\GG$. The map ${\sf F}$ is the forgetful map $(\GG,f,e)\mapsto
(\GG,f)$, while the map ${{\sf C}}$ is the contraction map
$(\GG,f,e)\mapsto (\GG/e, f')$.

We apply the basic properties of groupoid integrals to simplify the
second and the third terms on the right hand side of
the differential equation. If we denote $\nu_2(t)$ the sum
$\big(\nu^F_2(t)+\nu^\delta_2(t)\big) $, then the second and
the third terms are given, as in the proof of Theorem~\ref{thm:K}, by~\cite{Fio} 
\begin{align*}
  \int_{\mathbb{D}_{m-1}} \wt \cdot
\hcK_{m-1}(t) \nu_2(t) d\mu & = \int_{\mathbb{D}'_m} {\sf C}^*\big( \wt \cdot
                              \hcK_{m-1}(t)\nu_2(t) \big) d\mu \\
  & = \int_{\mathbb{D}_{m}}{\sf F}_*{\sf C}^*\big( \wt\cdot
    \hcK_{m-1}(t)\nu_2(t) \big) d\mu.
\end{align*}
The vector-valued function
${\sf F}_*{\sf C}^*\big(\wt\cdot
\hcK_{m-1}(t)\nu_2(t) \big)$ is computed as
\begin{align*}
  {\sf F}_*{\sf C}^*& \big(\wt\cdot \hcK_{m-1}(t)\nu_2(t) \big)(\GG,f)  =
    \sum_{e\in T^{\sf contr}} {\sf C}^*\big( \wt \cdot
                      \hcK_{m-1}(t)\nu_2(t) \big)(\GG,f,e) \\
        & = \sum_{e\in T^{\sf contr}} \wt(\GG/e)\cdot t^{|E_G|-1} \cdot
\hcK_m(\GG,f) = |E_G^{\sf non-loop}|\cdot
\wt(\GG)\cdot t^{|E_G|-1} \cdot
\hcK_m(\GG,f),
\end{align*}
where the last equality follows from the inductive definition of
$\wt$.

The above expression and the contribution from the loop edges in
add up precisely to $|E_G|\cdot t^{|E_G|-1} \cdot \wt(\GG) \cdot
\hcK_m(\GG,f)$. This proves that $\hcK_m$'s  satisfy Fukaya's differential
equation, and thus $\hcK$ is an $L_\infty$ morphism.

For the second part of the theorem, it follows from a general fact on
maps of pseudo-isotopies~\cite[Theorem 2.7]{Tu3}. Indeed observe
that the map $\epsilon$ lifts to a map between pseudo-isotopies:
\[ \epsilon: \h^+_A \otimes \Omega_{[0,1]}^* \ra \hh_A \otimes
  \Omega_{[0,1]}^*.\]
This induces the desired homotopy commutative diagram after
integrating both sides.
\end{proof}
\section{Feynman sum formulas for CEI}
\label{sec:mainthm}

In this section we use the $L_\infty$-morphism $\hcK$ of
Theorem~\ref{thm:Khat-delta} in order to derive explicit formulas for the
categorical invariants $F_{g,n}^{A,s}$.  These formulas are given as
summations over partially directed stable graphs, with vertices
labeled by the tensors $\hbA_{g,k,l}$ defined using combinatorial
string vertices in~(\ref{subsec:defbeta}).

\subsection{\texorpdfstring{Expressing CEI in terms of
    $\hbA$}{Expressing CEI interms of hat-beta-A}}
Let $s: H\ra H_*(L_+)$ be a splitting of the Hodge filtration,
see Definition~\ref{def:splitting}.  As in the previous section we fix a
chain-level splitting $S: (L,b) \ra (L\series{u},b+uB)$ that lifts
$s$, of the form
\[S=\id+S_1u+S_2u^2+\cdots, \;\; S_j\in \End(L).\]
As in Section~\ref{sec:trivial-1} we also have its inverse operator
$R=S^{-1}$.

Denote by $L^{\sf Triv}=(L,b)$ the same underlying chain complex as
$L$, but endowed with trivial circle action. The chain map $S$ induces
an isomorphism $L^{\sf Triv}_- \ra L_-$, further inducing an
isomorphism of DGLAs which we still denote by
\[ S : \h_A^{\sf Triv} \ra \h_A^\triv.\]
Here the differential on the left hand side is just $b$ while on the
right hand side it is $b+uB$; both sides have zero Lie
bracket.  Analogously we also have an inverse isomorphism of DGLAs
\[ R: \h_A^\triv \ra \h_A^{\sf Triv}.\]
We also have the corresponding maps of Koszul resolutions
\[ \widehat{R}: \hh_A^\triv \ra \hh_A^{\sf Triv} \;\;\;\mbox{and}
  \;\;\; \widehat{S}:\hh_A^{\sf Triv} \ra \hh_A^\triv.\]
For example $\widehat{S}$ pre-composes with $R$ at inputs and
post-composes with $S$ at outputs.

Consider the following commutative diagram of DGLAs\footnote{The left
  square is only homotopy commutative by Theorem~\ref{thm:Khat-delta}.}:
\[
  \begin{CD} \h_A^+ @>\cK>> \h_A^{{\sf triv},+} @> R >> \h_A^{{\sf
Triv},+}\\ @V\biota VV @V\biota VV @V\biota VV\\ \hh_A
@>\hcK >> \hh_A^\triv @>\widehat{R}>> \hh_A^{\sf Triv}.
  \end{CD}
\]
It induces a commutative diagram of isomorphisms of the associated
Maurer-Cartan moduli spaces:
\[
  \begin{CD} \MC(\h_A^+) @>\cK_*>> \MC(\h_A^{\triv,+}) @> R_* >>
\MC(\h_A^{{\sf Triv},+})\\ @V\biota_* VV @V\biota_* VV
@V\biota_* VV\\ \MC(\hh_A) @>\hcK_*>> \MC(\hh_A^\triv)
@>\widehat{R}_*>> \MC(\hh_A^{\sf Triv}).
  \end{CD}
\]
The categorical enumerative invariant $F^{A,s}$ of
Definition~\ref{def:main} lives in the upper-right corner space
$\MC(\h_A^{{\sf Triv},+})$, while the Maurer-Cartan element $\hbA$
lives inside the lower-left corner $\MC(\hh_A)$. Our goal is to write
down a formula for $F^{A,s}$ in terms of $\hbA$, as the latter is
computable using combinatorial string vertices from
Definition~\ref{defi:vertices}.  Using Theorem~\ref{thm:Khat-delta}
it follows that
\[ \biota_* F^{A,s}=\widehat{R}_*\hcK_*\hbA.\]
This yields an explicit formula for the categorical enumerative
invariants $F^{A,s}_{g,n}$ as a Feynman sum over partially directed
stable graphs, using the formula for $\hcK$ described
in~(\ref{subsec:k-hat-symplectic}).  The precise result is summarized
in the following theorem.

\begin{Theorem}
\label{thm:main}
For any $g\geq 0$, $n\geq 1$ such that $2g-2+n>0$,
we have
\[\biota (F_{g,n}^{A,s}) = \sum_{m\geq 1} \sum_{\GG\in
    \Gamma((g,1,n-1))_m} (-1)^{m-1} \frac{\wt(\GG)}{\Aut(\GG)} \prod_{v}
  {\sf Cont} (v) \prod_{e} {\sf Cont} (e) \prod_{l} {\sf Cont} (l)\]
as homology classes in $H_*\big( L_{1,n-1}, b\big)$. In this formula
the contributions ${\sf Cont} (v)$, $ {\sf Cont} (l)$ and
$ {\sf Cont} (e)$ associated with
$\GG=(G,L_G^{\sf in}\coprod L_G^{\sf out},E^\direct,T)$ are as follows:
\begin{itemize}
\item[--] At every vertex $v$ we assign the tensor $\hbA_{g(v),k(v),l(v)}$.
\item[--] At incoming legs we assign the operator $S$ defined above; at 
outgoing legs we put the operator $R$.
\item[--] For directed edges in the spanning tree $T$ we assign the
homotopy operator $F$ from Proposition~\ref{prop:homotopy2}; for all
other directed edges we put the operator $\Theta$ from~(\ref{subsec:dgla-2}).
\item[--] For undirected edges in the spanning tree $T$ we assign the
homotopy operator $\delta$ in Lemma~\ref{lem:delta}; for all other
undirected edges we put the homotopy operator $H^{\sym}$ from
Proposition~\ref{prop:homotopy}.
\end{itemize} The total contribution $\prod_{v} {\sf Cont} (v)
\prod_{e} {\sf Cont} (e) \prod_{l} {\sf Cont} (l)$ is the composition
using the resulting sign diagram, which yields an element in
$L_{1,n-1}$. The extra sign is due to the sign in the definition of
$\{-,-\}_\hbar^F$ in~(\ref{subsec:moretriv}), since all the $\hbA$'s
are odd elements of $\hh_A$.
\end{Theorem}

\subsection{\texorpdfstring{Explicit formulas for small $g$ and
    $n$}{Explicit formulas for small g and n}}
\label{sec:formulas}

In this appendix we give explicit formulas for CEI of Euler
characteristic $\chi\geq -3$. 

The formulas will be written in terms of partially directed graphs.
Our conventions when drawing such a graph $(G,L_G^{\sf in}\coprod
L_G^{\sf out},E^{\sf dir},T)$ are as follows:
\begin{itemize}
\item[--] we shall omit the genus decoration of a vertex if it is
clear from the combinatorics of the graph;
\item[--] we shall omit the drawing of the spanning tree $T$ if there
is a unique choice of it; otherwise, the spanning tree $T$ will be
drawn in blue.
\end{itemize}

\paragraph{{\bf The formula for the $(0,1,2)$-component.}}  We begin
with the case when $g=0$ and $n=3$. In this case there is a unique
partially directed stable graph. Thus we have
\[ \biota(F^{A,s}_{0,3})= \frac{1}{2} \includegraphics[trim=-1cm 1cm
  0 0, scale=.25]{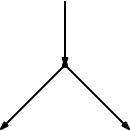}\]
The coefficient $\frac{1}{2}$ is due to the automorphism that switches
the two outputs in the stable graph.

\paragraph{{\bf The formula for the $(1,1,0)$-component.}}
\label{subsec:11}
In this case there are two stable graphs. The invariant $\biota(
F_{1,1}^{A,s})$ is given  by
\[\includegraphics[trim=-1cm 1cm 0 0, scale=.25]{png/110.png} \]

\paragraph{{\bf The formula for the $(0,1,3)$-component.}}  In this
case the formula for $\biota(F^{A,s}_{0,4})$ is
\[ \includegraphics[scale=.25]{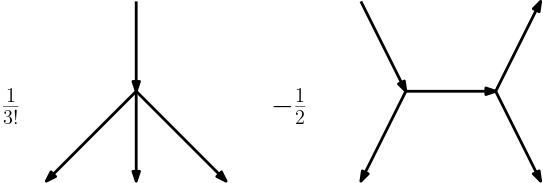} \]
The coefficient $\frac{1}{2}$ appears due to the symmetry of the two
outgoing leaves on the right hand side of the stable graph.

\paragraph{{\bf The formula for the $(1,1,1)$-component.}}
In this case we get
\begin{align*} \biota(F_{1,2}^{A,s}) &
= \begin{tikzpicture}[baseline={(current bounding
box.center)},scale=0.25] \draw [thick,directed] (3.4,4) to (3.4,2);
\node at (3.4,2) {$\bullet$}; \node at (2.4,2) {\scriptsize $g=1\qquad$}; \draw
[thick,directed] (3.4,2) to (3.4,0);
\end{tikzpicture} {}_{+\ \frac{1}{2}}\ \begin{tikzpicture}[baseline={(current
bounding box.center)},scale=0.25] \draw [thick,directed] (3.4,4) to
(3.4,2); \node at (3.4,2) {$\bullet$}; \draw [thick] (2.7,1.3) circle
[radius=1]; \draw [thick,directed] (3.4,2) to (4.8, 2);
\end{tikzpicture} \ {}_-\  \begin{tikzpicture}[baseline={(current bounding
box.center)},scale=0.25] \draw [thick,directed] (3.4,4) to (3.4,2);
\node at (3.4,2) {$\bullet$}; \draw [thick,directed] (3.4,2) to
(1.4,2); \draw [thick,directed] (3.4,2) to (7.4,2); \node at (7.4,2)
{$\bullet$}; \node at (8.4,2) {\scriptsize $\qquad g=1$};
\end{tikzpicture}\\ & {}_{-\frac{1}{2}}\  \begin{tikzpicture}[baseline={(current bounding
box.center)},scale=0.25] \draw [thick,directed] (3.4,4) to (3.4,2);
\node at (3.4,2) {$\bullet$}; \node at (3.4,-2) {$\bullet$}; \draw
[thick,blue,directed] (3.4,2) to [out=240, in=120] (3.4,-2); \draw
[thick,directed] (3.4,2) to [out=300, in=60] (3.4,-2); \draw
[thick,directed] (3.4,-2) to (3.4,-4);
\end{tikzpicture} \ {}_{-\frac{1}{2}}\ \begin{tikzpicture}[baseline={(current
bounding box.center)},scale=0.25] \draw [thick,directed] (3.4,4) to
(3.4,2); \node at (3.4,2) {$\bullet$}; \node at (3.4,-2) {$\bullet$};
\draw [thick,blue,directed] (3.4,2) to [out=240, in=120] (3.4,-2);
\draw [thick] (3.4,2) to [out=300, in=60] (3.4,-2); \draw
[thick,directed] (3.4,-2) to (3.4,-4);
\end{tikzpicture} \ {}_{-\frac{1}{2}}\ \begin{tikzpicture}[baseline={(current
bounding box.center)},scale=0.25] \draw [thick,directed] (3.4,4) to
(3.4,2); \node at (3.4,2) {$\bullet$}; \node at (3.4,-2) {$\bullet$};
\draw [thick,directed] (3.4,2) to [out=240, in=120] (3.4,-2); \draw
[thick,blue] (3.4,2) to [out=300, in=60] (3.4,-2); \draw
[thick,directed] (3.4,-2) to (3.4,-4);
\end{tikzpicture} \ {}_{-\frac{1}{2}}\ \begin{tikzpicture}[baseline={(current
bounding box.center)},scale=0.25] \draw [thick,directed] (3.4,4) to
(3.4,2); \node at (3.4,2) {$\bullet$}; \draw [thick,directed] (3.4,2)
to (5.4,2); \draw [thick,directed] (3.4,2) to (3.4,-1); \draw [thick]
(3.4,-2) circle [radius=1];
\end{tikzpicture}.
\end{align*}
Observe that in the first graph of the second line there are two
directed edges between the two vertices.  This explains how tensors
$\widehat{\beta}^A_{g,k,l}$ with $k\geq 2$ can contribute to the
categorical enumerative invariants.  In the third graph of the second
line the spanning tree has an undirected edge in which case we assign
the homotopy operator $\delta$ as in Lemma~\ref{lem:delta}. The
coefficient $-\frac{1}{2}$ in the first three graphs of the bottom
line is deduced using the inductive definition of
weights~(\ref{subsec:weight}):
\[ \wt(\GG)=\frac{1}{|E_G^{\sf non-loop}|}=\frac{1}{2}.\]
The negative sign is because there are two vertices so
$(-1)^{m-1}=-1$.

\paragraph{{\bf The formula for the $(0,1,4)$-component.}} In this
case $\biota(F_{0,5}^{A,s})$ is given by
\[ \includegraphics[scale=.3]{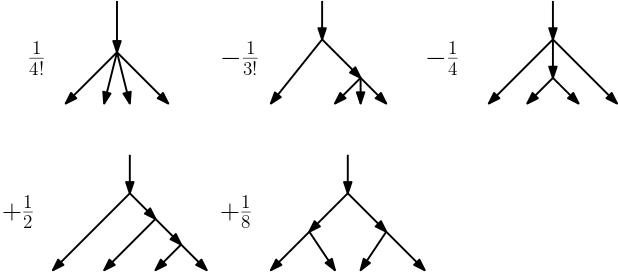}.\]
In genus zero the weights of partially directed graphs are $1$,
see~(\ref{subsec:weight}).  Hence only the automorphism
factors show up.  For example, the third graph in the bottom line has
$\wt=1$ and $|\Aut|=8$.

\paragraph{{\bf The formula for the $(1,1,2)$-component.}} In this
case $\biota(F_{1,3}^{A,s})$ is given by
\[ \includegraphics[scale=.5]{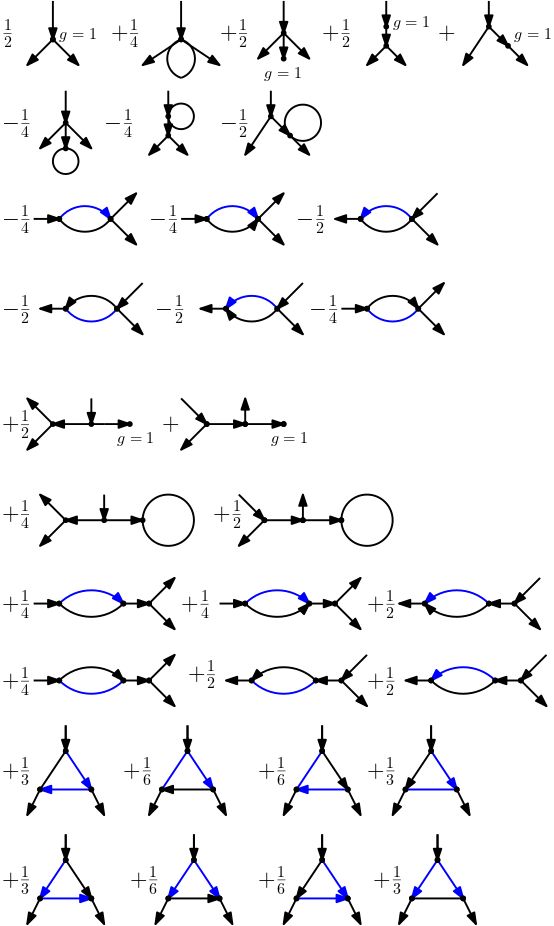}.\]
\newpage

\paragraph{{\bf The formula for the $(2,1,0)$-component.}} We also
list the graph sum for $\biota(F_{2,1}^{A,s})$:
\[\includegraphics[scale=.5]{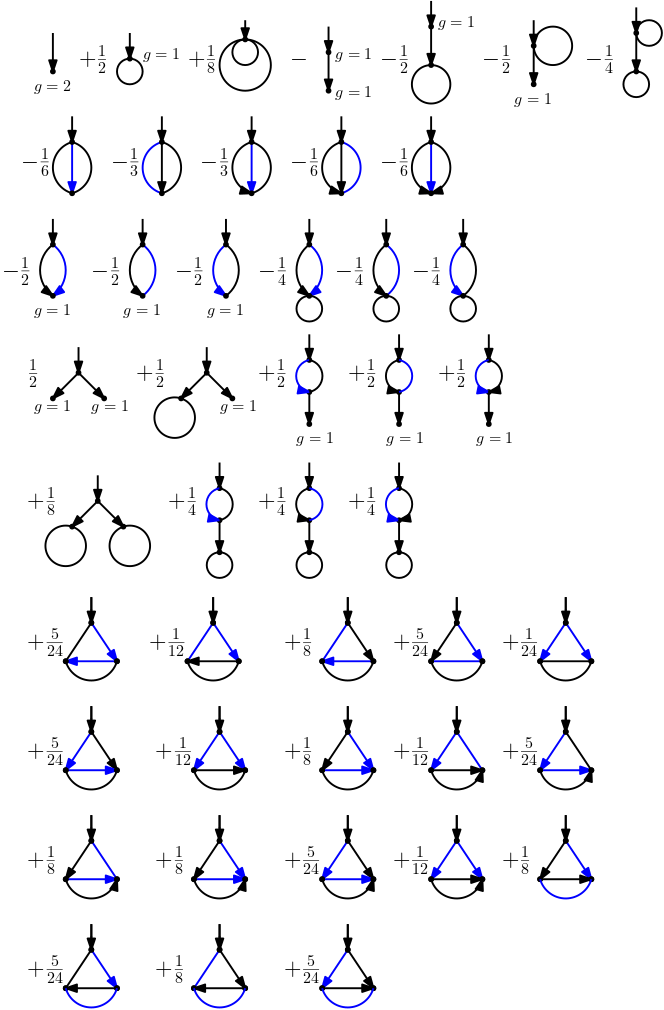}.\]
\newpage
\appendix 
\section{The integer-graded case}
\label{app:z-graded}

So far we have assumed that all our vector spaces were
$\Z/2\Z$-graded, and we did not concern ourselves with even graded
shifts (as these do not affect signs).  In this appendix we will
sketch how our results need to be modified in the $\Z$-graded case so
that all the maps involved will be of a well-defined homogeneous
degree.  This allows us to obtain an analogue of the dimension axiom
of Gromov-Witten theory in this case.

The choices we make are as follows.  The formal variables $u$ and
$\lambda$ have homological degree $-2$.  If the Calabi-Yau degree
(i.e., the homological degree of the cyclic pairing) of the algebra
$A$ is $d$, then $\hbar$ has degree $-2+2d$.  When the variable
$\hbar$ is used in the context of chains on moduli spaces of curves,
$d$ will be assumed to be zero (since the chain level operator giving
the Mukai pairing has degree zero), so $\deg\hbar = -2$ in that
setting.

The standard Mukai pairing has homological degree zero
\[ \pairing_\Muk:C_*(A) \otimes C_*(A) \ra \bbK. \]
Thus when viewed as an operator on $L = C_*(A)[d]$ the Mukai pairing
has degree $-2d$.

With the above conventions the Weyl algebra is $\Z$-graded, because the
generators of the ideal defining it are of homogeneous degree $2d-2$.
This is the reason we choose the residue pairing to take the
coefficient of $u^1$ and not $u^{-1}$.

We rewrite the definitions of $\h$ in~(\ref{subsec:dgla1}) and of
$\hh$ in~(\ref{subsec:dgla-2}) by adding even shifts:
\begin{align*}
  \h & = \Sym (L_-)\series{\hbar, \lambda}[1-2d], \\
  \hh & = \bigoplus_{k\geq 1,l} \Hom^\cont\left (\Sym^k(L_+[1-2d]),
        \Sym^lL_-\right)[2-2d]\series{\hbar, \lambda}.
\end{align*}
It is easy to check that with these gradings $\h$ and  $\hh$ are
$\Z$-graded DGLAs.

Now assume that the splitting $s$, its chain-level lift $R$, and the
inverse $T$  of $R$ all preserve degrees.  This implies that
\begin{itemize}
  \item[--] $H^{\Sym}: \Sym^2 L_- \ra \bbK$ has degree $2-2d$ in
    Proposition~\ref{prop:homotopy};
  \item[--] $\Theta: L_-\ra L_+[1-2d]$  has degree
    $2-2d$ in Proposition~\ref{prop:homotopy2};
  \item[--] $F: L_-\ra L_+[1-2d]$ has degree $3-2d$ in
    Proposition~\ref{prop:homotopy2}.
\end{itemize}
Using these facts it is easy to verify that the maps $\hcK_m$ 
have degree zero, as desired. 

We can now prove that the categorical enumerative invariants
$F_{g,n}^{A,s}$ satisfy the dimension axiom of Gromov-Witten theory
when the $A_\infty$-algebra is $\Z$-graded. More precisely we have the
following.

\begin{Theorem}
  \label{coro:dimension}
  Assume that $A$ is $\Z$-graded, of Calabi-Yau dimension
  $d$, and assume that the splitting $s$ preserves degrees. Then
  \[ \deg F^{A,s}_{g,n}= 2\big((d-3)(1-g)+n \big)\]
  as an element of $\Sym^n H_-$.
\end{Theorem}

\begin{proof}
  The element
  \[ \hcV = \sum_{g, k\geq 1, l} \hcV_{g,k,l}\,\hbar^g\,\lambda^{2g-2+k+l} \]
  is a solution of the Maurer-Cartan equation in $\hg$, so in
  particular its (unshifted) degree is $-1$.  Thus the degree of its component
  $\hcV_{g,k,l}$ of genus $g$, $k$ inputs and $l$ outputs is
  \[ \deg \hcV_{g,k,l} = 6g-7+3k+2l. \]
  Its image under $\rho^A$, as an element of $\Hom\left (\Sym^k(L_+[1]),
  \Sym^l(L_-)\right )$ will have degree 
  \[ \deg \hbA_{g,k,l} = (6g-7+3k+2l)+d(2-2g-2k). \]
  See Theorem~\ref{thm:tcft} for the provenance of the second term.
  
  A simple inspection shows that in the $\Z$-graded version of $\hh$
  (with the further shift by $2-2d+2dk-k$ in its definition, and with
  $\hbar$ of degree $2d-2$) the element
  $\hbA_{g,k,l}\,\hbar^g \,\lambda^{2g-2+k+l}$ has degree $-1$.  We
  conclude that $\hbA$ is a Maurer-Cartan element in the $\Z$-graded
  DGLA $\hh$.

  The splitting $s$ is degree preserving, so we can find a chain-level
  lift of it $R$ that also preserves degrees.  This implies that the
  inverse $T$ of $R$ also preserves degrees, and hence the $L_\infty$
  morphism $\hcK$ has degree zero.
  
  Therefore
  \[ F^{A,s} = \sum_{g,n} F^{A,s}_{g,n} \,\hbar^{g}\,\lambda^{2g-2+n},\]
  as a Maurer-Cartan element of
  $\h^{\sf Triv}=\Sym H_-\series{\hbar, \lambda}[1-2d]$, also has
  unshifted degree $-1$.  When we consider $F^{A,s}_{g,n}$ as an
  element of $\Sym^n H_-$ its degree is given by
  \begin{align*}
    \deg(F^{A,s}_{g,n}) &= (-1)+(2d-1)-\deg(\hbar)\cdot g
                          -\deg(\lambda)\cdot (2g-2+n) \\
                        &= (2d-2)-(2d-2)g+2(2g-2+n) = 2\big
                          ((d-3)(1-g)+n \big).
  \end{align*}
  This indeed matches with the virtual dimension formula in
  Gromov-Witten theory. 
\end{proof}

\end{document}